\documentclass[11pt,twoside,a4]{article} 

\RequirePackage[OT1]{fontenc}
\RequirePackage{amsthm,amsmath,amsfonts,amssymb}
\RequirePackage[colorlinks]{hyperref}
\RequirePackage{epsfig, graphicx, color}
\RequirePackage{latexsym}
\RequirePackage{color}
\usepackage{natbib}
\usepackage{xcolor}

\usepackage{fullpage}
\usepackage{epsf,epsfig,subfigure}
\usepackage{epsf,epsfig}
\usepackage{graphics,graphicx}
\usepackage{enumerate}
\usepackage{algorithmicx,algorithm}
\usepackage{bbm}

\usepackage{subfigure}

\newtheorem{theorem}{Theorem}
\newtheorem{lemma}[theorem]{Lemma}

\newtheorem{proposition}[theorem]{Proposition}

\theoremstyle{definition}

\newtheorem{assumption}{Assumption}

\theoremstyle{remark}

\usepackage{tikz}
\usetikzlibrary{arrows.meta,
	calc,
	positioning,
	quotes}
\usetikzlibrary{shapes.geometric}

\newcommand{\N}{\mathbb{N}}

\newcommand{\R}{\mathbb{R}}
\newcommand{\PP}{\mathbb{P}}
\newcommand{\E}{\mathbb{E}}
\newcommand{\argmin}{{\arg\hspace*{-0.5mm}\min}}

\renewcommand{\t}{\top}

\newcommand{\one}{\mathbbm{1}}
\newcommand{\given}{\,|\,}

\def\lo{\ell}

\def\var{\mathrm{Var}}
\newcommand{\Cov}{\mathrm{Cov}}
\newcommand{\cov}{\mathrm{Cov}}
\newcommand{\op}{\mathrm{op}}

\newcommand{\Var}{\mathrm{Var}}
\newcommand{\abs}[1]{\left| #1\right|}
\newcommand{\norm}[1]{\left\lVert #1 \right\rVert}
\newcommand{\ind}{\mathbbm{1}}
\newcommand{\pr}{\mathbb{P}}
\newcommand{\sgn}{\mathrm{sgn}}

\def\bX{\mathbf{X}}
\def\bU{\mathbf{U}}
\def\bZ{\mathbf{Z}}
\def\bY{\mathbf{Y}}

\def\be{\boldsymbol{\varepsilon}}

\def\cF{\mathcal{F}}

\def\cB{\mathcal{B}}
\def\cC{\mathcal{C}}
\def\cL{\mathcal{L}}
\def\cT{\mathcal{T}}

\def\rem{\mathcal{E}}
\newcommand{\bs}{\boldsymbol}
\newcommand{\mb}{\mathbf}

\newcommand{\iid}{\stackrel{\mathrm{i.i.d.}}{\sim}} 

\newcommand\independent{\protect\mathpalette{\protect\independenT}{\perp}}
\def\independenT#1#2{\mathrel{\rlap{$#1#2$}\mkern2mu{#1#2}}}

\newcommand{\vertiii}[1]{{\left\vert\kern-0.25ex\left\vert\kern-0.25ex\left\vert #1 
    \right\vert\kern-0.25ex\right\vert\kern-0.25ex\right\vert}}
    
\makeatletter
\newcommand{\customlabel}[2]{%
\protected@write \@auxout {}{\string \newlabel {#1}{{#2}{}}}}
\makeatother

\title{Latent confounding in high-dimensional nonlinear models}

\author{Yuhao Wang\thanks{Part of this work was done while YW was at the University of Cambridge, UK.
	}\\
	IIIS, Tsinghua University, Shanghai Qi Zhi Institute\\
    and Shanghai AI Lab, China\\
	\url{yuhaow@tsinghua.edu.cn}
    \and
	Rajen D. Shah\\
	Statistical Laboratory,\\ University of Cambridge, UK\\
	\url{r.shah@statslab.cam.ac.uk}}

\begin{document}

\maketitle

%
%
%
%
%
%
%
%
%

\begin{abstract}
We consider the the problem of identifying causal effects given a high-dimensional treatment vector in the presence of low-dimensional latent confounders. We assume a parametric structural causal model in which the outcome is permitted to depend on a sparse linear combination of the treatment vector and confounders nonlinearly. We consider a generalisation of the LAVA estimator of \citet{CHL17} for estimating the treatment effects and show that under the so-called `dense confounding' assumption that each confounder can affect a wide range of observed treatment variables, one can estimate the causal parameters at the same rate as possible without confounding.
Notably, the results permit a form of weak confounding in that the minimum non-zero singular value of the loading matrix of the confounders can grow more slowly than the $\sqrt{p}$, where $p$ is the dimension of the treatment vector.
We further use our generalised LAVA procedure within a generalised covariance measure-based test for edges in a causal DAG in the presence of latent confounding.
\end{abstract}

\section{Introduction}
A regression analysis of response $Y \in \R$ against covariates $X \in \R^p$ can in general only reveal aspects of the association between these. In some cases however, the ultimate goal of such an analysis is deriving causal conclusions. Such inferences can be justified, but one needs to make the potentially strong assumption of no unmeasured confounding. In this work we consider the problem of causal effect estimation in the presence of latent confounding. Specifically, we consider the following structural causal model
\begin{equation}\label{eq:model}
	Y = f(X^\t \beta^0 + U^\t \delta^0) + \varepsilon \quad \textrm{and} \quad X= \Gamma^\t U +Z  ,
\end{equation}
and suppose our observed data $\{(Y_i, X_i)\}_{i=1}^n$ are i.i.d.\ copies of $(Y,X)$.
Here, $f$ is a known increasing link function, errors $\varepsilon$ satisfy $\E(\varepsilon \given X, U) = 0$, $U \in \R^q$ represent latent confounders affecting $X$ and $Y$ through the unknown coefficient matrix and vector $\Gamma \in \R^{q \times p}$ and $\delta^0 \in \R^q$ respectively, and $\beta^0 \in \R^p$ captures the causal effect of interest. We assume that the errors $Z$ in the $X$-model are mean-zero and independent of the latent confounders $U$.

 
 While in general estimation of the target $\beta^0$ is not possible due to lack of identifiability, in certain cases where $q \ll p$, $\beta^0$ is sparse and each of the small number of latent confounders affects many the components of $X$, a flurry of recent work has shown that $\beta^0$ can in fact be estimated in a linear model setting, that is when $f$ is the identity function. This is sometimes referred to as `dense confounding' \citep{CBM20}. Such assumptions may be plausible in several settings: a small number of batch effects in genetic studies may simultaneously affect a large proportion of the measurements \citep{gagnon2013removing}; the returns of stocks in may be confounded by general market movements driven by a small number of factors \citep{menchero2010global}; or for example in studies on cell biology the activities of proteins and mRNA may be confounded by a few main environmental factors \citep{leek2007capturing,stegle2012using}.
 
 To see why dense confounding may permit estimation of $\beta^0$, consider a simple setting where we have a single latent factor so $q=1$ and $\Gamma \in \R^{1 \times p}$ is a row vector determining the affect of $U \in \R$ on each component of $X$. Further assume that $\Cov(Z) = I$ and without loss of generality, take $\Var(U) =1$. Writing $\gamma := \Gamma^{\top}$, we have that the population level linear regression coefficient
 \begin{align*}
 \{\E[XX^{\top}]\}^{-1} \E [XY] &= \beta^0 + (\gamma\gamma^{\top} + I)^{-1} \gamma \delta_0 \\
 &= \beta^0 + \frac{\gamma \delta_0}{1 + \|\gamma\|_2^2} =: \beta^0 + b^0,
 \end{align*}
 using the Sherman--Morrison formula for the penultimate equality. We thus see that the regression coefficient is a sum of a sparse vector $\beta^0$ and a dense vector $b^0$ that in this simple setting is proportional to $\gamma$ and crucially in the case where the $p$-vector $\gamma$ is such that its $\ell_2$ norm is of the order $\sqrt{p}$, we should have $\|b^0\|_2 = O(1/\sqrt{p})$. Based on this insight, \citet{CBM20} showed that one can use the LAVA estimator~\citep{CHL17} to estimate the sparse causal parameter $\beta^0$, which involves solving
 \[
(\hat{\beta}, \hat{b}) =  \argmin_{\beta, b \in \R^p} \left\{\frac{1}{n} \|\bY - \bX(\beta + b)\|_2^2 + \lambda_1 \|\beta\|_1 + \lambda_2 \|b\|_2^2 \right\}.
 \]
 Here $\bY := (Y_1,\ldots,Y_n)^{\top}$ and $\mb X \in \R^{n \times p}$ is the matrix with $i$th row given by $X_i$. The idea is that $\hat{b}$ should estimate $b^0$, and as the latter may be expected to have a small $\ell_2$ norm, a ridge penalty is appropriate. Since for any given $\beta$, there is a closed form expression for the minimising $b$ in the above, one can show that
 \begin{equation} \label{eq:Lasso}
 \hat{\beta} = \argmin_{\beta \in \R^p} \left\{ \frac{1}{n} \| A \mb Y - A \mb X\|_2^2 + \lambda_1 \|\beta\|_1\right\}
 \end{equation}
 for some matrix $A \in \R^{n \times n}$ depending on $ \mb X$. \citet{CBM20} study procedures of this form, where $A \mb X$ has the same singular vectors as $\mb X$, but shrinks its singular values, with the Lava transform being a special case. The analysis in \citet{CBM20}, developed further in  \citet{GCB22}, hinges on the representation of Lava as a spectral transformation, which in turn rests critically on a linear regression model. While the linear regression model has the attraction of theoretical tractability, in many cases it may be a poor approximation to the underlying data generating process; for instance in the case where $Y$ is binary, a logistic regression model may be more natural.
 
%
 \subsection{Our contributions}
 In this work, we study estimation of $\beta^0$ in nonlinear models based on a penalised M-estimation approach with a Lava penalty.
 While our approach cannot be reduced to a standard Lasso program as in~\eqref{eq:Lasso}, we show nevertheless in Section~\ref{sec:optimization} that an alternative representation can facilitate fast computation. In Section~\ref{sec:theory} we present finite sample guarantees on the estimation of $\beta^0$ and both in- and out-of-sample prediction errors. A noteworthy feature of our results is that our theory allows for settings where the minimum eigenvalue $\lambda_{\min}(\Gamma \Gamma^{\top}) = o(p)$. In contrast, most earlier work particularly in the context of nonlinear models such as ours, requires this minimum eigenvalue to be of order $p$, a condition known as `pervasiveness' in the factor modelling literature \citep{FWZZ21}. Intuitively, when the confounding is strong, the column space of the latent confounding matrix $\mb U \in \R^{n \times q}$ with $i$th row $U_i$ will approximately coincide with the space spanned by the first $q$ principal components of $\mb X$. The confounding can thus be learnt from the data and controlled for to estimate $\beta^0$, and several approaches for handling latent confounding take this approach (see below). An issue is that pervasiveness, which is necessary for estimation of the factors \citep{Onatski12}, may not hold in practice where there may be some weaker confounders that result in the smaller eigenvalues of $\Gamma \Gamma^{\top}$.
 Our result relies on an adapted form of restricted eigenvalue condition \citep{tsybakov2009simultaneous,van2009conditions} which we show in Section~\ref{sec:re} holds with high probability under relatively mild conditions.
 
 In Section~\ref{sec:testing} we use our results on the in-sample prediction error of our estimator to show that a form of generalised covariance measure \citep{shah2020hardness} may be used for testing the presence of an edge in a causal DAG in the presence of latent confounding. In Section~\ref{sec:numerical} we present the results of some numerical experiments on simulated data that illustrate the encouraging performance of our method.
 Supplementary material contains all the proofs of results that appear in the main text and additional numerical results.
 Before introducing our method, we discuss below some additional related literature.
 
 \subsection{Related literature}
 A notable early work in the study of high-dimensional latent variable models such as ours here is \citet{CPW12}, which looks at Gaussian graphical model estimation in the presence of latent confounders. \citet{FNM19} builds on this for learning causal structure in this setting. \citet{SFTM20} introduces an approach for high-dimensional correlation matrix estimation based on projection onto the right singular vectors of the data matrix. \citet{scheidegger2025spectral} uses a spectral transformation approach \citep{CBM20} to estimate the causal effect in an additive regression model; a key difference with our setting here is that the contribution of the latent factors is additive rather rather than nonlinear.

An alternative to the spectral transformation approach to handle latent confounding involves attempting to estimate the latent confounders and then controlling for them \citep{wang2019blessings}. As such, the problem studied here has some close connections to high-dimensional factor analysis, see for example \citet{FLM13} and the review paper \citet{FWZZ21} as well as references therein. \citet{fan2024latent} use the approach of adjusting for estimated latent factors in the context of linear models. \citet{ouyang2023high} and extend this to generalised linear models. This setting is closely related to ours here, but that under study here is slightly more general in that no parametric distributional assumptions are placed on the error $\varepsilon$ in \eqref{eq:model}. 

A distinct approach to estimating causal effects in the presence of latent confounding involves using instrumental variables; some recent works that have considered this in the context of high-dimensional data include \citet{gautier2011high,fan2014endogeneity,belloni2017program,gold2020inference,10.1093/jrsssb/qkad049}. Particularly in such settings however, the construction of instrumental variables can be challenging, requiring extensive domain knowledge, and valid instrumental variables may not be available.

More broadly, our work connects to a rich literature on inference and estimation in high-dimensional settings, particularly in generalised linear models \citep{van2008high}; see also \citet{buhlmann2011statistics} for a book-length treatment. In particular, our work bears some relation to work on the debiased Lasso \citep{ZZ14,Geer14,javanmard2014confidence} and related approaches such as the decorrelated score test of \citet{NL17}; see \citet{SB19} for a review. \citet{GCB22, ouyang2023high,fan2024latent} build on these to provide inference under latent confounding. Our analysis of the restricted eigenvalue condition we introduce is related to analyses of the classical restricted eigenvalue condition \citep{rudelson2013reconstruction,RWY10} and restricted strong convexity \citep{negahban2012unified}, though the modifications we introduce to handle latent confounding complicate matters.

\section{Generalised LAVA}\label{sec:optimization}
Recall that we assume our data follow the model \eqref{eq:model} with the link function $f : \R \to \R$; we further assume throughout that $f$ is strictly increasing and continuously differentiable with $\sup_{\eta \in \R} |f'(\eta)| < \infty$.
This for example incorporates a linear regression setting with $f$ as the identity, or a logistic regression setting with $f(\eta) = 1 / (1 + e^{-\eta})$.

To estimate $\beta^0$, we use an M-estimation strategy with a LAVA penalty. Let $F:\R \to \R$ be an antiderivative of $f$ and  set $\ell(y,\eta) := -y\eta + F(\eta)$. Writing
\[
L(\theta ) := \frac{1}{n} \sum_{i=1}^n \ell(Y_i, X_i^{\top}\theta),
\]
our estimator \emph{generalised LAVA} estimator of $\beta^0$ is given by $\hat{\beta}$ where for tuning parameters $\lambda_1,\lambda_2 \geq 0$,
\begin{equation} \label{eq:gen_LAVA}
(\hat{\beta}, \hat{b}) := \argmin_{\beta,b \in \R^p} \left\{ L(\beta+ b) + \lambda_1 \|\beta\|_1 + \lambda_2 \|b\|_2^2\right\}.
\end{equation}
Note that when no counfounding is present, under mild integrability conditions we have that $\E \nabla L(\beta^0) = \E \{-Y X + X f(X^{\top} \beta^0)\} = 0$, so $\beta^0$ is a stationary point of the population version of the cost function. \citet{HGC21} considered a similar M-estimation strategy (omitting $b$) in a context with no confounding.
The rationale for the inclusion of $\hat{b}$ is that it serves to capture the disturbance to the ideal model $\E(Y \given X) = f(X^{\top} \beta^0)$ induced by the confounding. Note however that unlike in the case of the linear model discussed in the introduction, it is not clear what the population level target of $\hat{b}$ is, in that under confounding, we do not necessarily have $\E(Y \given X) = f(X^{\top}(\beta^0  + b^0))$ for some $b^0 \in \R^p$. (Indeed, this is one of the difficulties in moving from the linear model to the nonlinear setting we study here.)
An alternative interpretation of the parameter $b$ in \eqref{eq:gen_LAVA} is that it induces a Huber loss-like penalty function on $\theta:= \beta + b$, as the following result shows.
\begin{proposition} \label{prop:theta_optim}
We have that $(\hat{\beta}, \hat{b})$  are minimisers of the generalised LAVA objective \eqref{eq:gen_LAVA} only if $\hat{\theta} := \hat{\beta} + \hat{b}$ minimises
\begin{equation} \label{eq:theta_optim}
L(\theta) + \sum_{j=1}^p \rho_{\lambda_1,\lambda_2} (\theta_j)
\end{equation}
over $\theta \in \R^p$, where
\begin{align*}
	\rho_{\lambda_1, \lambda_2}(t) := \left\lbrace\begin{array}{lr}
		\lambda_2 t^2 & |t| \leq \lambda_1 / (2 \lambda_2),\\
		\lambda_1 |t| - \frac{\lambda_1^2}{4\lambda_2} & |t| > \lambda_1 / (2 \lambda_2).
	\end{array}\right.
\end{align*}
Moreover, if $\hat{\theta}$ minimises \eqref{eq:theta_optim} over $\theta \in \R^p$, then the pair $(\hat{\beta}, \hat{b})$ given by
\begin{align*}
\hat{\beta}_j := \left\{\hat{\theta}_j - \frac{\mathrm{sgn}(\hat{\theta}_j)\lambda_1}{2\lambda_2} \right\} \ind_{\{|\hat{\theta}_j| > \lambda_1 / (2\lambda_2)\}}
\end{align*}
and $\hat{b} := \hat{\theta} - \hat{\beta}$ minimise the generalised LAVA objective.
\end{proposition}
We use the representation of the solution in the above, which has the advantage of reducing a convex program involving $2p$ variables, to one involving only $p$ variables, in our implementation used for the numerical results in Section~\ref{sec:numerical}. 

\section{Estimation and prediction guarantees}\label{sec:theory}
In this section, we present guarantees on the estimation accuracy of $\hat{\beta}$ as well as results on the quality of the fitted regression function $x \mapsto f(x^{\top}(\hat{\beta} + \hat{b}))$. The latter will prove useful for the significance test we present in Section~\ref{sec:testing}.

For a vector $u \in \R^p$ and a non-empty set $A \subseteq [p]$, we denote by $u_A$ the subvector of $u$ with components indexed by $A$ (in increasing order).
For $a, b, \in \R$, we use $a \lesssim b$ to denote that there exists a constant $C > 0$, which may depend on other quantities designated as constants, such that $a \leq C b$. Throughout, we assume the model~\eqref{eq:model} and consider the link function $f$ as fixed (so constants can depend on properties of $f$). We write $S:= \{ j : \beta^0_j \neq 0\}$ and $s := |S|$. Furthermore, we write $r_{\ell} := \lambda_{\min}(\Gamma \Gamma^{\top})$. We require the following assumptions on the scaling of the various parameters $n,p,s,q,r_{\ell}$:
\begin{assumption}\label{cd:pns}
We have
    \begin{enumerate}
        \item[(i)] $p \gtrsim n$ and $1 \leq q \lesssim 1$;
        \item[(ii)] there exists some sequence $a_n \to 0$ such that $\log p \leq a_n n^{1/7}$;
        \item[(iii)] $r_\lo \gtrsim \max \left\{p \left( \frac{\log p}{n}\right)^{\frac{1}{6}}, p^{5/6} \right\}$;
        \item[(iv)] $s \geq 1$ and there exists some constant sequence $b_n \to 0$ such that
\[
s \max\left\{\frac{1}{\sqrt{n}},\, \frac{\sqrt{p}}{r_\lo} \right\} \log p \leq b_n.
\]
    \end{enumerate}
\end{assumption}
Notably, the scaling required for $r_{\ell}$ permits asymptotic regimes where $r_{\ell} / p \to 0$, which as discussed in the introduction, precludes consistent estimation of the latent factors $U_i$. Condition (iv) on the scaling of $s$ is a little stronger than for example necessary for $\ell_1$ estimation consistency in standard high-dimensional linear regression.
We note that Assumptions~\ref{cd:pns}(i), (ii),  and also $s \geq 1$ in (iv) are used primarily for simplicity of the analysis.


We also make use of the following sub-Gaussianty conditions on the random variables involved in model \eqref{eq:model}. Note that in the below, $X_j$ refers to the $j$th component of the random vector $X \in \R^p$.
\begin{assumption}
    \label{cd:subg}
	We assume $Z$ and $U$ are mean-zero sub-Gaussian random vectors and $\varepsilon$ and $X_{j}$ for $j=1,\ldots,p$ are mean zero sub-Gaussian random variables. That is, there exist $\sigma_z, \sigma_u, \sigma_\varepsilon, \sigma_x, \sigma_\eta > 0$ such that for each unit vector $u \in \R^p, v \in \R^q$, and for all $t \in \R$, 
	\begin{gather*}
			\E[\exp(t  Z^\top u)] \leq \exp(t^2 \sigma_z^2 / 2), \quad \E[\exp(t  U^\top v)] \leq \exp(t^2 \sigma_u^2 / 2), \\
			\E[\exp(t  \varepsilon) \,|\, X, U] \leq \exp(t^2 \sigma_\varepsilon^2 / 2) \quad \text{almost surely,} \\
			\forall j \in [p],\; \E[\exp(t X_{j})] \leq \exp(t^2 \sigma_x^2 / 2) \\
				\E\left[\exp\left(t \left(X^\top \beta^0 + U^\top \delta^0\right)\right)\right] \leq \exp(t^2 \sigma_\eta^2 / 2).
	\end{gather*}
Without loss of generality, we assume $\Var(U) = I$, and we further assume $\|\delta^0\|_2 \lesssim 1$.
%
\end{assumption}
The final assumption on the strength of the confounding is required to ensure $\var(U^\top \delta^0) \lesssim 1$.


Our theoretical results below rely on a certain `good' event 
$\Omega(\tau, \kappa, c_p)$, for $\tau,\kappa, c_p > 0$, occurring with high probability, which is defined as follows:
\begin{align}\label{eq:rscevent}
	& \Omega(\tau, \kappa, c_p) := \bigg\{\forall \Delta_1 \in C(S), \Delta_2 \in \R^p, \nonumber\\
	&\quad \frac{1}{n} \sum_{i=1}^n \left(X_i^\top (\Delta_1 + \Delta_2)\right)^2 \one_{\{|X_i^\top \beta^0 + U_i^\top \delta^0 
		| \leq \tau \}} \geq \kappa \|\Delta_1 \|_2^2 - c_p r_\lo \sqrt{\frac{\log p}{n}} \|\Delta_2\|_2^2 \bigg\},
\end{align}
where 
\begin{equation} \label{eq:cone}
C(S) := \{\Delta \in \R^p: \|\Delta_{S^c}\|_1 \leq 3 \|\Delta_S\|_1\}.
\end{equation}
Intuitively, the event asks for a modified version of the restrictive eigenvalue condition \citep{tsybakov2009simultaneous,van2009conditions} to hold, except that here we consider a truncated $X_i$ and that we have added a new term $\Delta_2$ which does not belong to the cone $C(S)$. In Section~\ref{sec:re}, we discuss this further and present some sufficient conditions under which $\Omega(\tau, \kappa, c_p)$ holds with high probability.

We have the following result on the estimation error of the target $\beta^0$:
\begin{theorem}\label{thm:main}
Assume Assumptions~\ref{cd:pns} and \ref{cd:subg}. Given $m \in \N, \kappa, c_p, \tau > 0$ (treated as constants), we have that there exist constants $c_{\lambda_1}, c_{\lambda_2}, c > 0$
depending on $m, \kappa, c_p$ and $\tau$ such that by choosing $\lambda_1 = c_{\lambda_1} \sqrt{\frac{\log p}{n}}$ and $\lambda_2 = c_{\lambda_2} r_{\lo} \sqrt{\frac{\log p}{n}}$, with probability at least $\PP(\Omega(\tau, \kappa, c_p)) - c n^{-m}$,
\[
\|\hat{\beta} - \beta^0\|_2 \lesssim \sqrt{s \frac{\log p}{n}} \quad\text{and} \quad \|\hat{\beta} - \beta^0\|_1 \lesssim s \sqrt{\frac{\log p}{n}}.
\]
\end{theorem}
We see in particular that the $\ell_2$ and $\ell_1$ estimation errors coincide with those expected in high-dimensional generalised linear models without confounding present \citep{Geer14}.
We next provide results for in-sample and out-of-sample prediction risk. As we shall see in Section~\ref{sec:testing}, the in-sample prediction risk can be used to justify a procedure for testing the existence of an edge in a causal DAG in the presence of latent confounding.

\begin{theorem}[In sample prediction risk]\label{thm:inpred}
	Under the same conditions and with $c_{\lambda_1}, c_{\lambda_2}$ as in Theorem~\ref{thm:main}, we have that there exists constant $c> 0$ such that with probability at least $\PP(\Omega(\tau, \kappa, c_p)) - c n^{-m}$, 
	\[
	\frac{1}{n} \sum_{i = 1}^n \left\{f\left(X_i^\top (\hat{\beta} + \hat{b}) \right)- f(X_i^\top \beta^0 + U_i^\top \delta^0)\right\}^2  \lesssim s \frac{\log p}{n} + \frac{p \log p}{r_{\lo} n} + \frac{p^2}{r_{\lo}^2 n}.
	\]
\end{theorem}

One shortcoming of the results above is that choice of $\lambda_2$ depends on the unknown minimum eigenvalue $r_{\ell}$. The following result however suggests that the required choices of tuning parameters should deliver good generalisation performance. This motivates the use of cross-validation for choosing appropriate tuning parameters and we demonstrate empirically in Section~\ref{sec:numerical} that this can work well in practice.
\begin{theorem}[Out sample prediction risk]\label{thm:prediction}
	Let $(X_{\star}, Z_{\star}, U_{\star})$ be independent of $(X_i, Z_i, U_i)_{i=1}^n$ and have the same distribution as $(X, Z, U)$.
	Under the same conditions and with $c_{\lambda_1}, c_{\lambda_2}$ as in Theorem~\ref{thm:main}, we have that there exists constant $c> 0$ such that with probability at least $\PP(\Omega(\tau, \kappa, c_p)) -  n^{-m}$, 
	\[
	\E\left[\left\{f\left( X_{\star}^\top (\hat{\beta} + \hat{b})\right) - f\left(X_{\star}^\top \beta^0 + U_{\star}^\top \delta^0\right)\right\}^2 \,\Big|\, \hat{\beta}, \hat{b}\right] \lesssim s^2 \frac{\log p }{n} + \frac{p \log p}{r_\ell n} 
	+ \frac{p^2}{r_{\lo}^2 n}.
	\]
\end{theorem}
Compared to the bound in Theorem~\ref{thm:inpred}, we see that we have a term of the form $s^2 \log(p)/n$ rather than $s \log(p)/n$. This comes as a consequence of the fact that the maximum eigenvalue of $\E (X X^{\top})$ is not bounded in our setting an indeed can be as large as $O(p)$.

\section{Testing for edges in a causal DAG}\label{sec:testing}
In this section, we consider the setting where in addition to $X$ and $Y$, we observe i.i.d.\ copies of an additional  random variable $W \in \R$, so our observed  data consists of $(X_i, Y_i, W_i)_{i=1}^n \iid (X, Y, W)$. Our goal is to test for the presence of an edge between $Y$ and $W$ in the causal directed acyclic graph (DAG) shown in Figure~\ref{fig:DAG}.
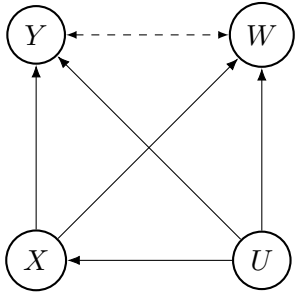
\begin{figure}
\begin{center}
		\begin{tikzpicture}
		[rv/.style={circle, draw, thick, minimum size=4mm}, rvc/.style={rectangle, draw, thick, minimum size=4mm}, node distance=30mm]
		\pgfsetarrows{latex-latex};
		\begin{scope}
			\node[rv]  (1)              {$Y$};
			\node[rv, right of=1] (2) {$W$};
			\node[rv, below of=1]  (3) {$X$};
			\node[rv, right of=3] (4) {$U$};
			\draw[-Latex] (3) -- (1);
			\draw[-Latex] (3) -- (2);
			\draw[-Latex] (4) -- (3);
			\draw[-Latex] (4) -- (1);
			\draw[-Latex] (4) -- (2);
			\draw[dashed] (1) -- (2);
		\end{scope}
	\end{tikzpicture}
\end{center}
\caption{Causal DAG depicting the hypothesis testing problem: under the null, there is no edge between $Y$ and $W$, while under an alternative, we have an edge here.\label{fig:DAG}}
\end{figure}
Here, under the null that no edge is present, $Y$ follows \eqref{eq:model} and $W$ satisfies the related regression model
\begin{align*}
	W = f^W(X^\top \beta^W + U^\top \delta^W) + \varepsilon^W.
\end{align*}
Similarly to \eqref{eq:model}, we require that the error $\varepsilon^W$ satisfies $\E(\varepsilon^W \given X, U) = 0$, and $f^W$ is a known link function (which may be different from $f$). Note that the null hypothesis we wish to test can equivalently be expressed as the conditional independence
\begin{equation} \label{eq:null}
	Y \independent W \given X, U,
\end{equation}
though unlike more conventional conditional independence testing problems, the conditioning variable $U$ is unobserved. A further interpretation of the the hypothesis testing problem at hand is as a variable significance test. To see this, consider $(W, X)$ as playing the role of $X$ in the model \eqref{eq:model}, so $W$ refers to a special predictor of interest within the vector of predictors $X$. The null hypothesis is then that the coefficient corresponding to $W$ is zero.

Our proposed testing procedure is based on the generalised covariance measure (GCM) \citep{shah2020hardness}, which is based on the observation that under the null \eqref{eq:null}, we have that $\E [\Cov(Y, W \given X, U)] = 0$. The GCM would involve regressing each of $\mb Y$ and $\mb W := (W_i)_{i=1}^n$ onto the conditioning variables, and the computing a test statistic based on an appropriately normalised empirical covariance between these residuals. A key feature of this approach is that while the mean squared error of the regression estimates may only converge at slower than parametric rates, the empirical covariance can nevertheless converge at the $1/\sqrt{n}$ rate to zero under the null since its bias can be controlled by a product of squared biases from each of the individual regressions. Thus if each of these individual mean squared errors are of order $o(n^{1/2})$, then this is sufficient to ensure that the bias in the final test statistic is negligible.

 While this approach cannot be used directly as $U$ is unobserved, since the method only relies on predictions of $Y$ and $W$, we can use our generalised LAVA procedure~\eqref{eq:gen_LAVA} to deliver these and rely on Theorem~\ref{thm:inpred} to guarantee the accuracy of the predictions. Specifically, our procedure takes the following steps:
\begin{enumerate}
	\item Form $(\hat{\beta}, \hat{b})$ and $(\hat{\beta}^W, \hat{b}^W)$ via minimising the generalised LAVA objective function \eqref{eq:gen_LAVA} taking $\bX \in \R^{n\times p}$ as our matrix of predictors and using response vectors $\mb Y$ and $\mb W$, and taking $\ell$ based on link functions $f$ and  $f_W$ respectively;
	\item Calculate the test statistic
	\[
	T := \frac{\frac{1}{\sqrt{n}} \sum_{i = 1}^n \hat{\varepsilon}_i \hat{\varepsilon}_i^W}{\sqrt{\frac{1}{n} \sum_{i = 1}^n \hat{\varepsilon}_i^2 (\hat{\varepsilon}_i^W)^2 - \left(\frac{1}{n} \sum_{i = 1}^n \hat{\varepsilon}_i \hat{\varepsilon}_i^W \right)^2}},
	\]
	where $\hat{\varepsilon}_i := Y_i - f(X_i^\top (\hat{\beta} + \hat{b}))$ and $\hat{\varepsilon}_i^W := W_i - f^W(X_i^\top (\hat{\beta}^W + \hat{b}^W))$ are the residuals from each of the regressions;
	\item Reject the null at the $\alpha$-level when $|T|> z_{\alpha/2}$ where $z_{\alpha}$ is the upper $\alpha$ point of a standard Gaussian distribution.
\end{enumerate}

The following result shows that the Type I error of the testing procedure is well-controlled.

\begin{theorem}\label{thm:inference}
	Consider the null setting \eqref{eq:null}.
Suppose Assumptions~\ref{cd:pns}--\ref{cd:subg} are satisfied as  well as versions of these conditions with $(\beta^0, f, \varepsilon)$ replaced by $(\beta^W, f^W, \varepsilon^W)$. Further assume that $\var(\varepsilon \varepsilon^W) \gtrsim 1$. Let constants $\tau, \kappa, c_p$ be given. There exist $c_{\lambda_1}, c_{\lambda_2}$ such that choosing $\lambda_1 = c_{\lambda_1} \sqrt{\frac{\log p}{n}}$ and $\lambda_2 = c_{\lambda_2} r_{\lo} \sqrt{\frac{\log p}{n}}$, we have that for any $\epsilon>0$ there exists a constant $c > 0$ such that for any $\alpha \in (0,1)$,
	\begin{align*}
	\PP(|T| \leq z_{\alpha/2}) \geq \; 1 - \alpha - \PP(\Omega^c(\tau, \kappa, c_p))  - c \left(s \frac{\log p}{\sqrt{n}} + \frac{p \log p}{r_\ell \sqrt{n}} + \frac{p^2}{r_\ell^2 \sqrt{n}} + n^{-1/2+\epsilon}\right).
\end{align*}
\end{theorem}

\section{On the restricted eigenvalue condition}\label{sec:re}
Our previous results have relied on an event $\Omega(\tau, \kappa, c_p)$, which asks for a modified restricted eigenvalue condition to hold, occurring with high probability. In this section, we discuss this further and provide some simple sufficient conditions for this. Recall that the regular restricted eigenvalue (RE)  condition \citep{tsybakov2009simultaneous} asks that for some $\kappa > 0$,
\begin{equation} \label{eq:RE}
\frac{1}{n} \sum_{i=1}^n (X_i^{\top} \Delta)^2 \geq \kappa \|\Delta\|_2^2
\end{equation}
holds for all $\Delta \in C(S)$, where the cone $C(S)$ is given by \eqref{eq:cone}. One then typically argues that  estimation error $\hat{\beta} - \beta^0 \in C(S)$ and moreover that a multiple of the left-hand side of \eqref{eq:RE} lower bounds the difference in the objective functions at $\beta^0$ and $\hat{\beta}$ respectively. There are however several complications in our setting that prevent this sort of analysis from going through directly. Firstly, as explained in Section~\ref{sec:optimization}, in addition to $\hat{\beta}$, as necessitated by the presence of latent confounding, our procedure produces a further estimate $\hat{b}$, and so defining the analogue of `estimation error' is not straightforward. Our proof of Theorem~\ref{thm:main} in the supplementary material however reveals that this estimation error may be taken to be $\hat{\beta} + \hat{b} - \beta^0 - b^0$ with $b^0$ a random quantity given by
\[
b^0 := \argmin_{b \in \R^p} \{L(\beta^0 + b) + \lambda_2 \|b\|_2^2\}.
\]
To reflect this, our RE condition \eqref{eq:rscevent} decomposes $\Delta$ in \eqref{eq:RE} into a sum of the quantities $\Delta_1$ and $\Delta_2$. The first, corresponding to the error $\hat{\beta} - \beta^0$, can be shown to lie in $C(S)$. However the second, corresponding to $\hat{b}- b^0$ will not in general satisfy any such approximate sparsity but can be shown to have small $\ell_2$ norm; this is the reason for our inclusion of the term involving $\|\Delta_2\|_2^2$ on the right-hand side of \eqref{eq:rscevent}.

A further complication is that depending on the link function, we cannot expect there to exist a quadratic lower bound of the form given by the left-hand side in \eqref{eq:RE}, to the difference in object function values at $(\hat{\beta}, \hat{b})$ and $(\beta^0,b^0)$ respectively. We thus truncate each summand by multiplying by $\ind_{\{|X_i^{\top} \beta^9 + U_i^{\top} \delta^0| \leq \tau\}}$.
To show that the event $\Omega(\tau, \kappa, c_p)$ occurs with high probability, we 
%
%
require the following conditions:
\begin{assumption}\label{cd:omesparsity}
	There exists some sequence $c_n \to 0$ such that $s \sqrt{ \log p / n} \leq c_n$. Furthermore,
	\[
	\frac{sp}{r_\ell}\sqrt{\frac{\log p}{n}} \lesssim 1.
	\]
\end{assumption}
The first condition is implied by Assumption~\ref{cd:pns}(iv). When $r_\ell \gtrsim p / \sqrt{ \log p}$, this also implies the second condition above, but when $r_\ell$ is smaller, then this may represent a slightly stronger sparsity requirement. 


\begin{theorem}\label{thm:resubg}
	Suppose there exists a constant $c_z>0$ such that $\lambda_{\min}(\E(Z Z^{\top})) \geq c_z$. Then under Assumptions~\ref{cd:subg} and~\ref{cd:omesparsity}, we have that for any given $m \in \N$, there exist constants $\kappa, c_p, \tau, c > 0$ such that
	\[
	\PP(\Omega(\tau, \kappa, c_p)) \geq 1 - c(n^{-m} + p^{-m}).
	\]
\end{theorem}
Note that the assumption on the minimum eigenvalue of $\cov(Z)$, in conjunction with 
 Assumption~\ref{cd:subg}, has implications for the size of $\|\beta^0\|_2$. Indeed, since
\[
\var(Z_i^\top \beta^0) = \var(X_i^\top \beta^0 + U_i^\top \delta^0) - \var( U_i^\top (\Gamma \beta^0 + \delta^0)) \leq  \var(X_i^\top \beta^0 + U_i^\top \delta^0) \lesssim 1.
\]
These together imply that $\|\beta^0\|_2 \lesssim 1$, i.e.\ $\|\beta^0\|_2$ is bounded from above by a constant. We compare our Theorem~\ref{thm:resubg} with the bound on the restricted eigenvalue of \citet[Prop.~1]{GCB22}, which studies estimation and inference in a linear regression setting with confounding, in Section~\ref{sec:bound_comments} of the Supplementary material.

\section{Numerical experiments} \label{sec:numerical}

In this section we report the performance of the generalised LAVA estimator in simulation settings with varying confounding and sparsity levels. In Section~\ref{sec:estimation} we present results on estimation, while in Section~\ref{sec:inf_exp} we present some results for the hypothesis testing procedure laid out in Section~\ref{sec:testing}.

\subsection{Estimation} \label{sec:estimation}
We generated datasets of sample size $n=2000$ and following the model \eqref{eq:model} with $p=500$ as follows. To generate the $X_i$, we first generated $Z_i \sim \mathcal{N}_p(0, \Sigma)$ and considered two possibilities for the covariance matrix $\Sigma$:
\begin{itemize}
	\item \emph{Toeplitz design:} $(\Sigma)_{i,j} = 0.9^{|i - j|}$;
	\item \emph{Exponential decaying design:} we first assign $(\Sigma^{-1})_{i,j} = 0.9^{|i - j|}$ and then scale the rows and columns such that all diagonal entries are $1$.
\end{itemize}
We set $U_i \sim \mathcal{N}_q(0, I)$ and varied 
$q \in \{5, 10, 20\}$. The loading matrix $\Gamma$ was generated anew for each dataset with each entry generated independently and $\Gamma_{jk} \sim \mathcal{N}(0, \nu^2 / k^2)$ with $\nu \in \{0.25, 0.5, 1, 2, 4, 8\}$.

We used a logistic regression model for the responses $Y_i$ with $\pr(Y_i = 1 \given X_i, U_i) = \{1 + \exp(-X_i^{\top} \beta^0 - U_i^{\top} \delta^0)\}^{-1}$ with $\delta^0$ a unit vector proportional to $(1,\ldots,1)^{\top}$. The parameter $\beta^0$ was generated anew in each dataset as follows: $s \in \{5, 10\}$ indices in $[p]$ were selected uniformly at random, and the corresponding entries of $\beta^0$ were generated as independent Rademacher random variables.
For each of the $2 \times 3 \times 6 \times 2=72$ $(\text{design type}, q, \nu, s)$ configurations, we generated $250$ datasets and applied our generalised LAVA procedure \eqref{eq:gen_LAVA}, the Lasso \citep{tibshirani1996regression} and the factor-based method of \citet{ouyang2023high}. For generalised LAVA, we first reparametrised the tuning parameters $(\lambda_1, \lambda_2) = (\lambda_1, \gamma \lambda_1 / 2)$ and considered $\lambda_1$ on the exponentially-spaced grid chosen for the tuning parameter $\lambda$ for the Lasso as implemented by the R package \text{glmnet} \citep{friedman2010regularization}; $\gamma$ was chosen from $\{p / 50, p / 40, \ldots, p / 10\}$. For both generalised LAVA and the Lasso, we selected their respective tuning parameters by 10-fold cross-validation (using the default options of \texttt{cv.glmnet} for the latter). For the method of \citet{ouyang2023high} we used code supplied by the authors: this computed an initial sparse estimator of $\beta$ and a non-sparse debiased version for a given individual component of $\beta$; for the purposes of estimation, we used this sparse estimator.



For each estimator $\hat{\beta}$ we considered as evaluation criterion, the squared prediction-type error
\begin{equation} \label{eq:err_metric}
\E[(X^\top (\hat{\beta} - \beta^0))^2 \given \hat{\beta}] = (\hat{\beta} - \beta^0)^\top (\Sigma + \Gamma^\top \Gamma) (\hat{\beta} - \beta^0),
\end{equation}
where in the above, $X \in \R^p$ is an independent data point following the same distribution as the training data.
\begin{figure}[t!]
	\subfigure[Toeplitz, $s = 5, q = 5$]{\includegraphics[width=0.33\textwidth]{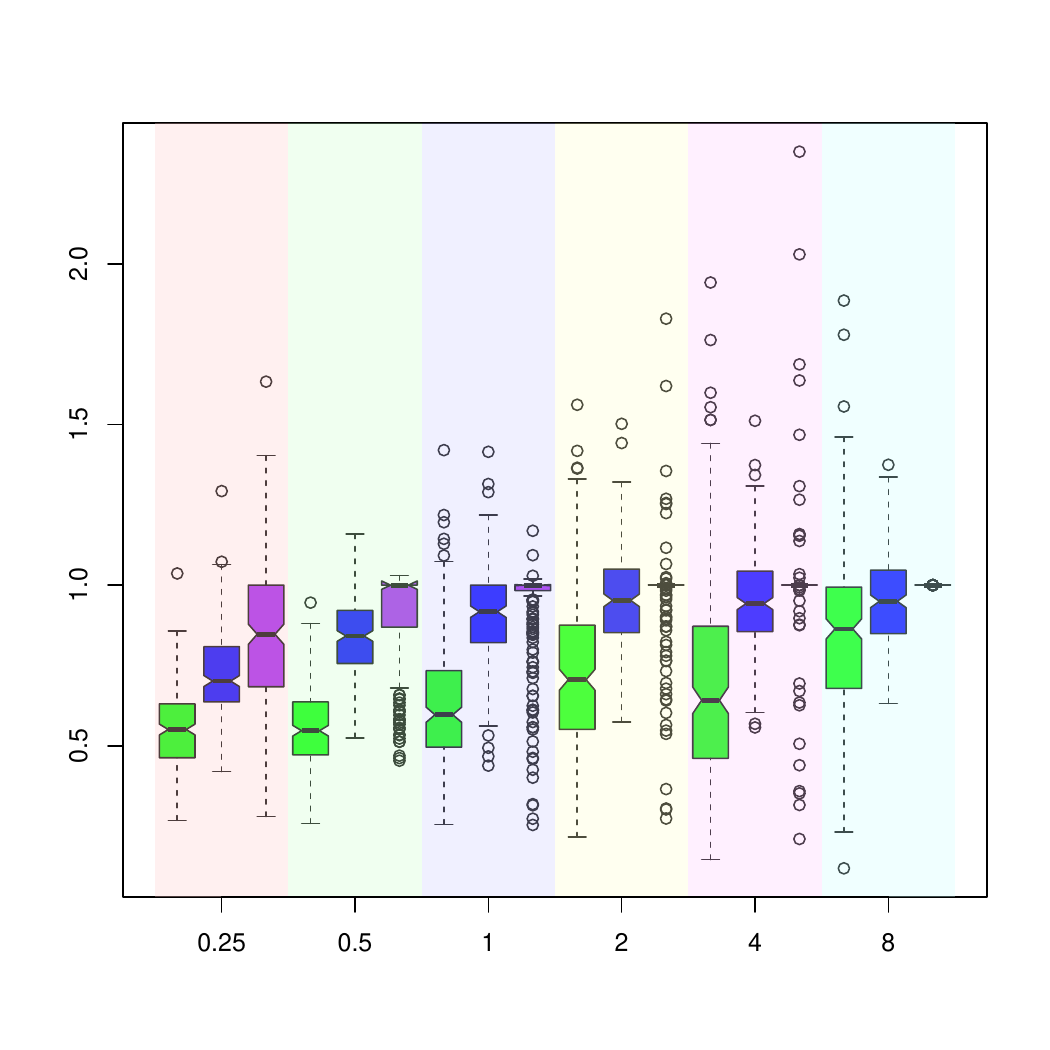}}
	\subfigure[Toeplitz, $s = 5, q = 10$]{\includegraphics[width=0.33\textwidth]{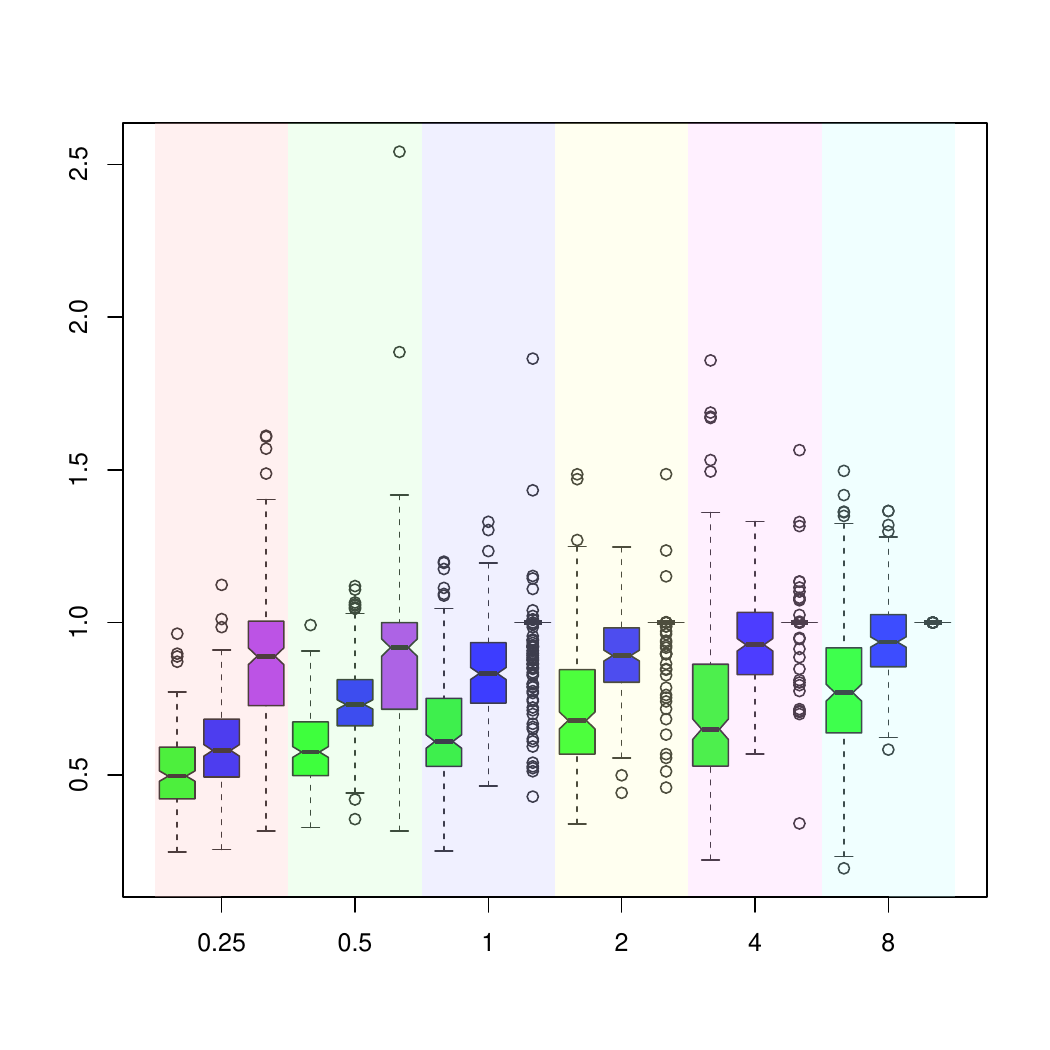}}
	\subfigure[Toeplitz, $s = 5, q = 20$]{\includegraphics[width=0.33\textwidth]{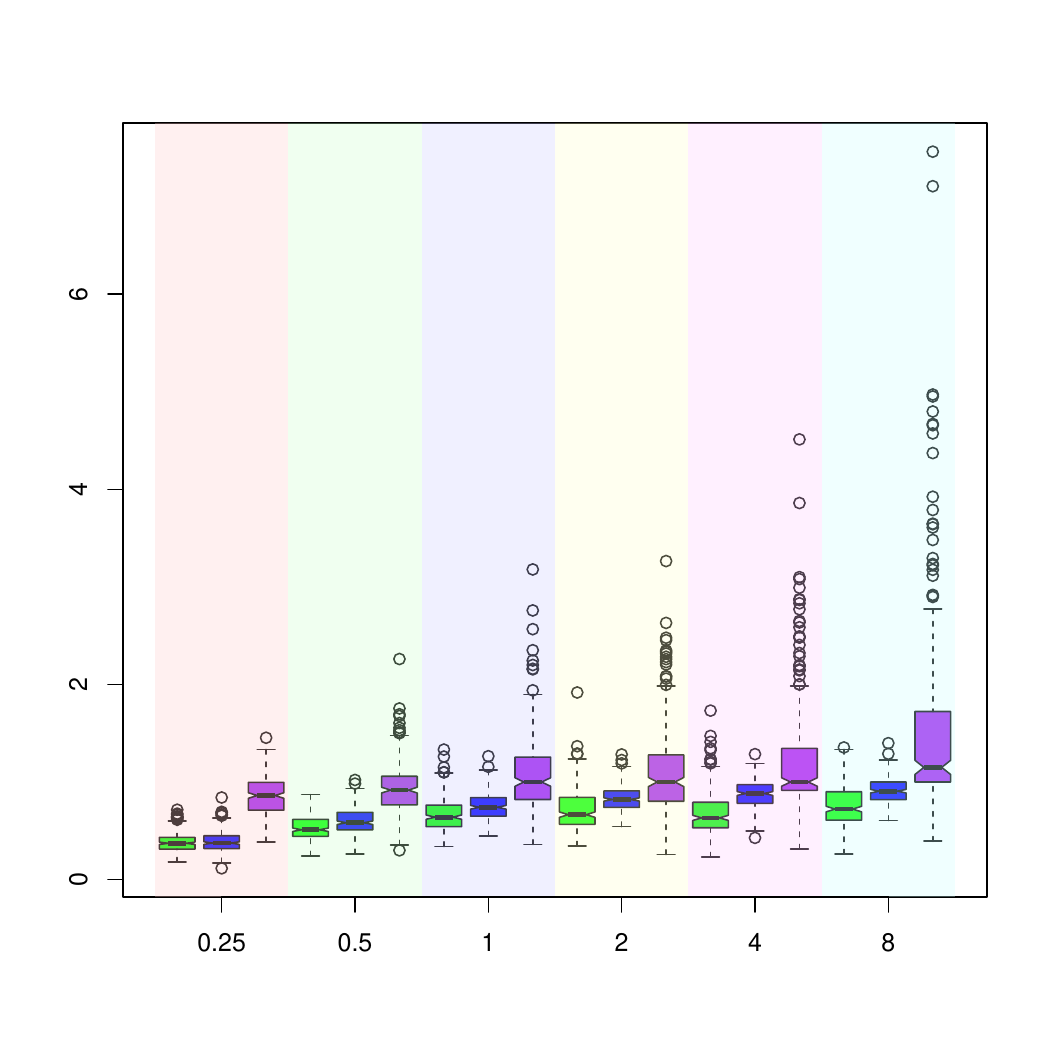}}
 \subfigure[Toeplitz, $s = 10, q = 5$]{\includegraphics[width=0.33\textwidth]{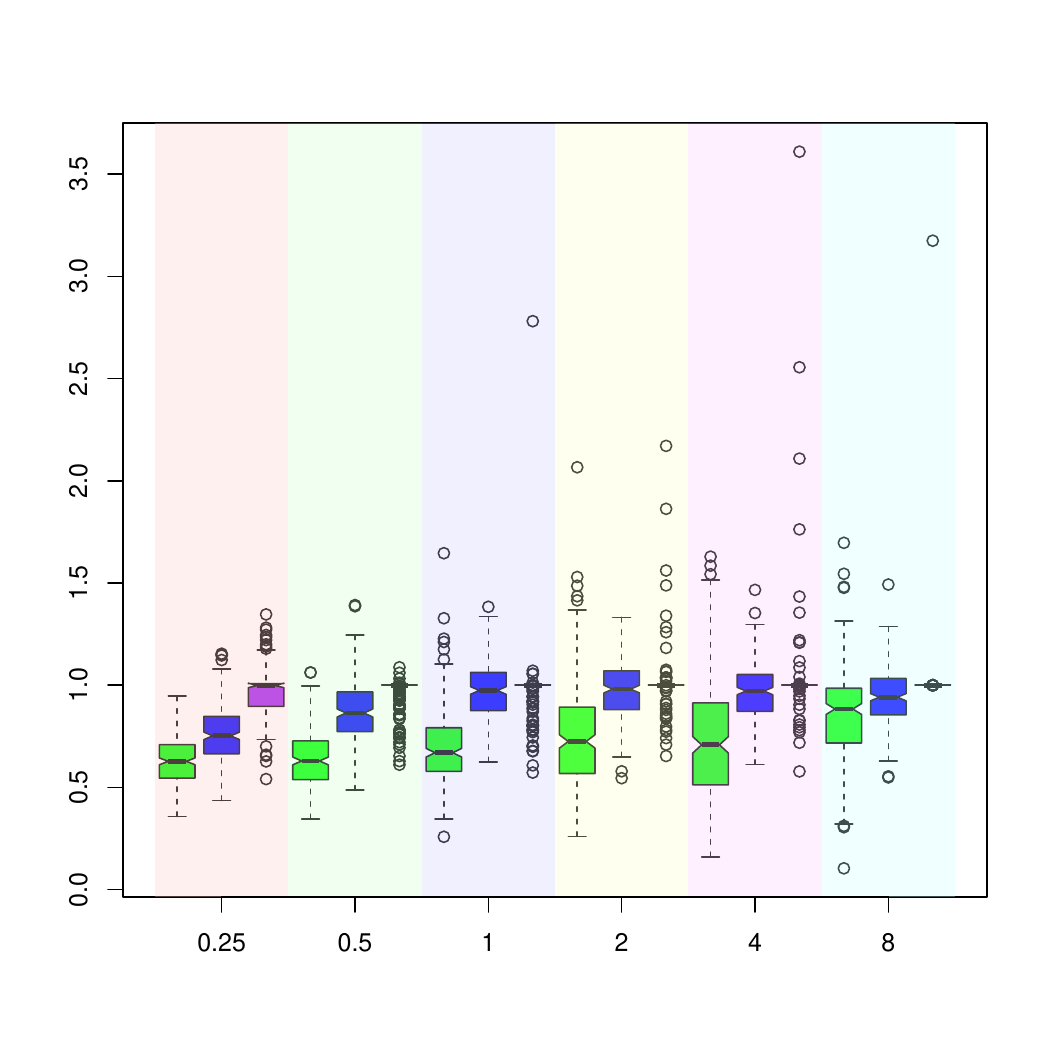}}
	\subfigure[Toeplitz, $s = 10, q = 10$]{\includegraphics[width=0.33\textwidth]{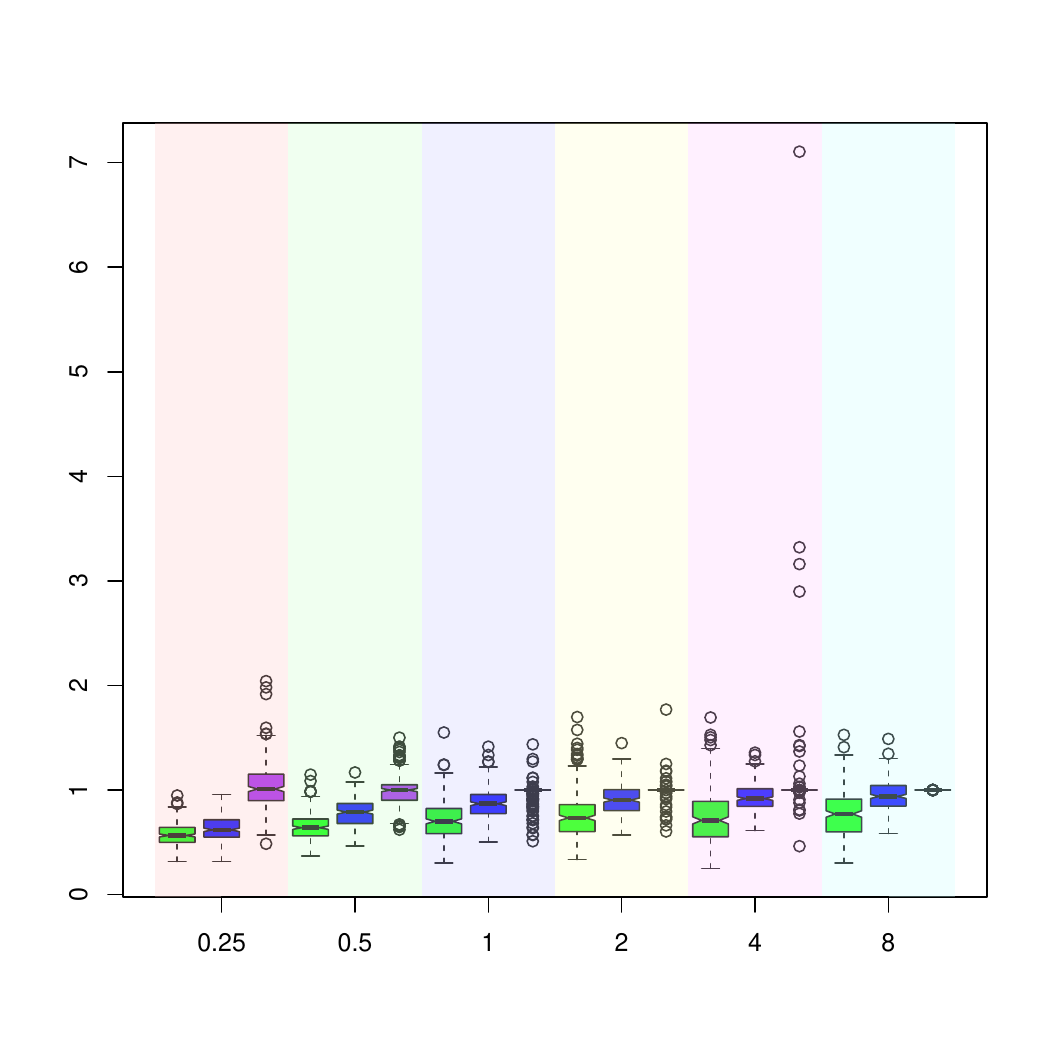}}
	\subfigure[Toeplitz, $s = 10, q = 20$]{\includegraphics[width=0.33\textwidth]{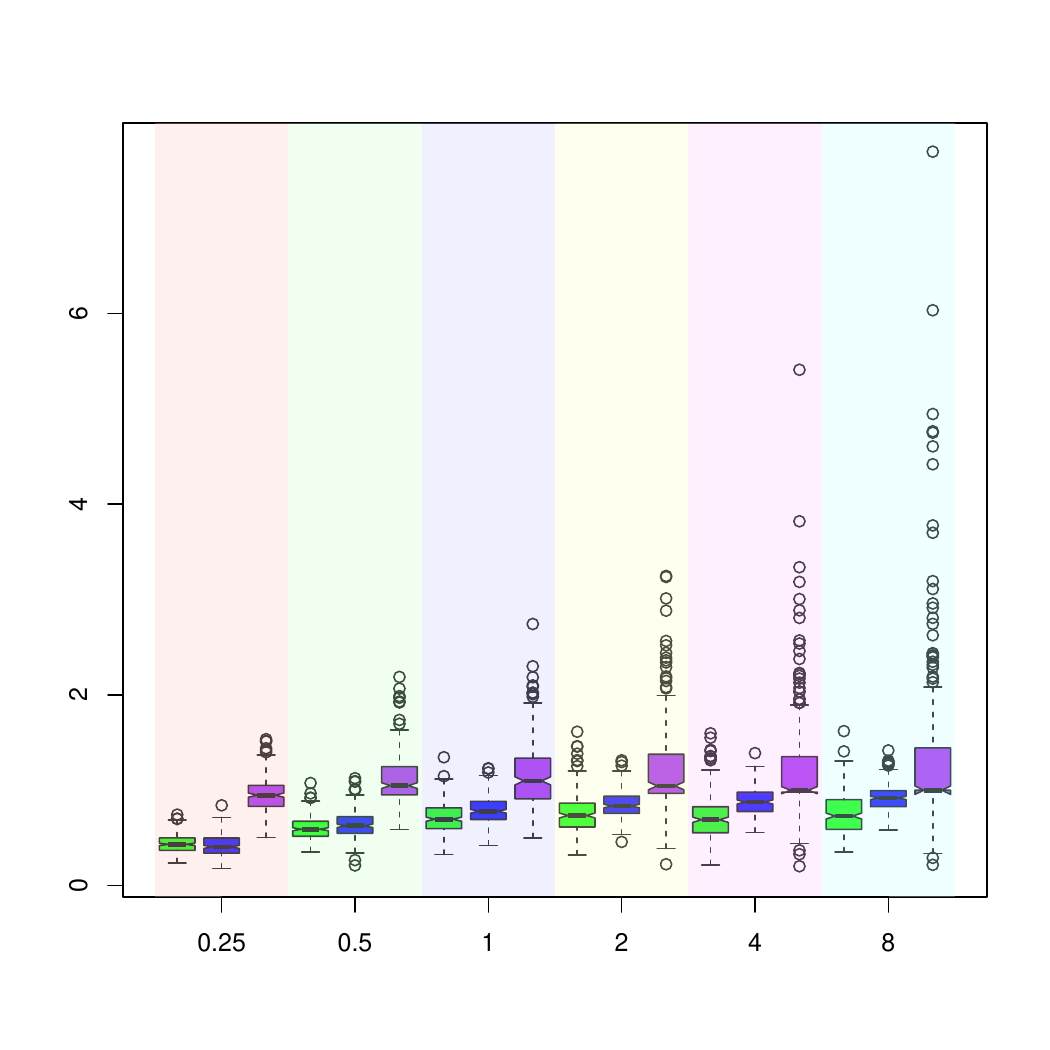}}
	\caption{Boxplots of errors \eqref{eq:err_metric} under Toeplitz designs for generalised LAVA (green), the Lasso (blue) and the factor-based method of \citet{ouyang2023high} (purple).}\label{fig:logitoep}
\end{figure}

\begin{figure}[t!]
	\subfigure[Exp. decaying, $s = 5, q = 5$]{\includegraphics[width=0.33\textwidth]{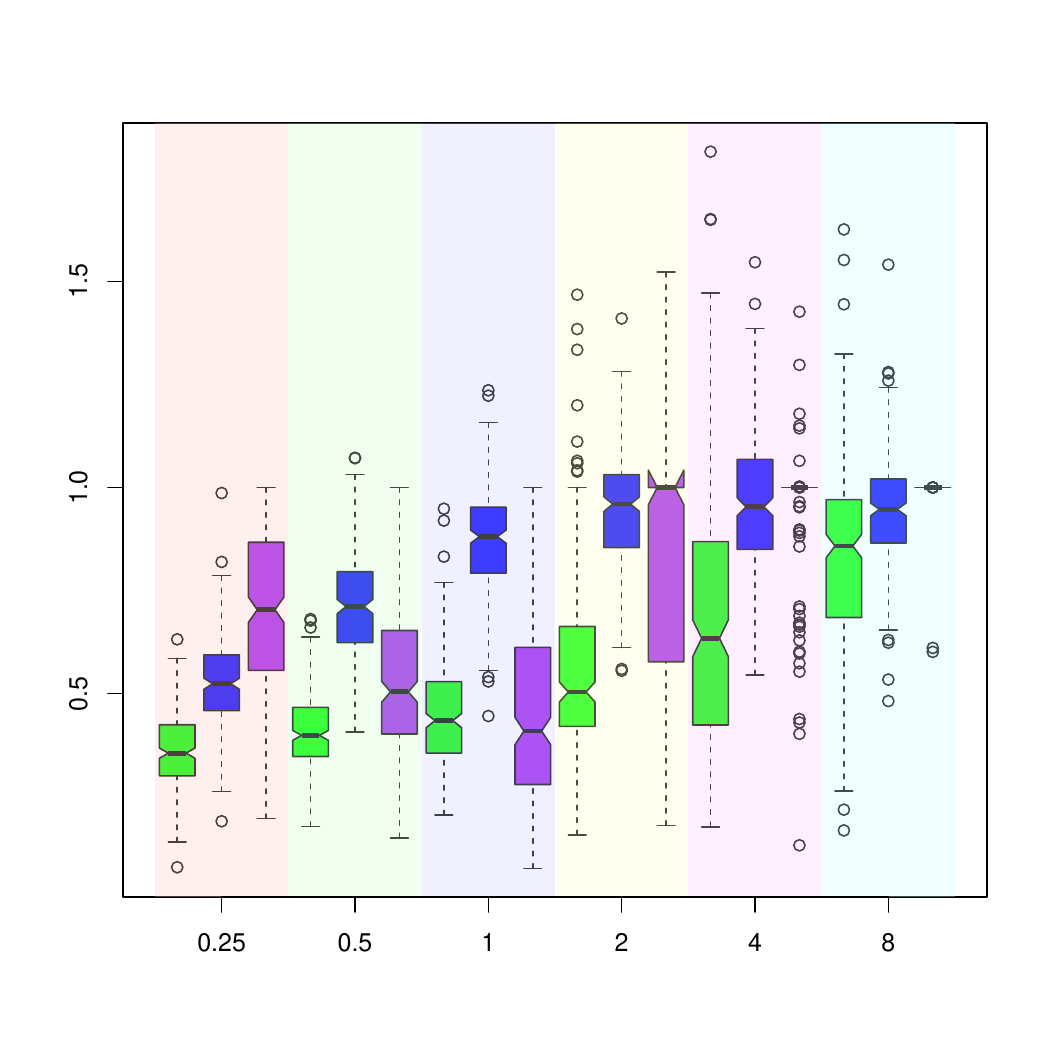}}
 \subfigure[Exp. decaying, $s = 5, q = 10$]{\includegraphics[width=0.33\textwidth]{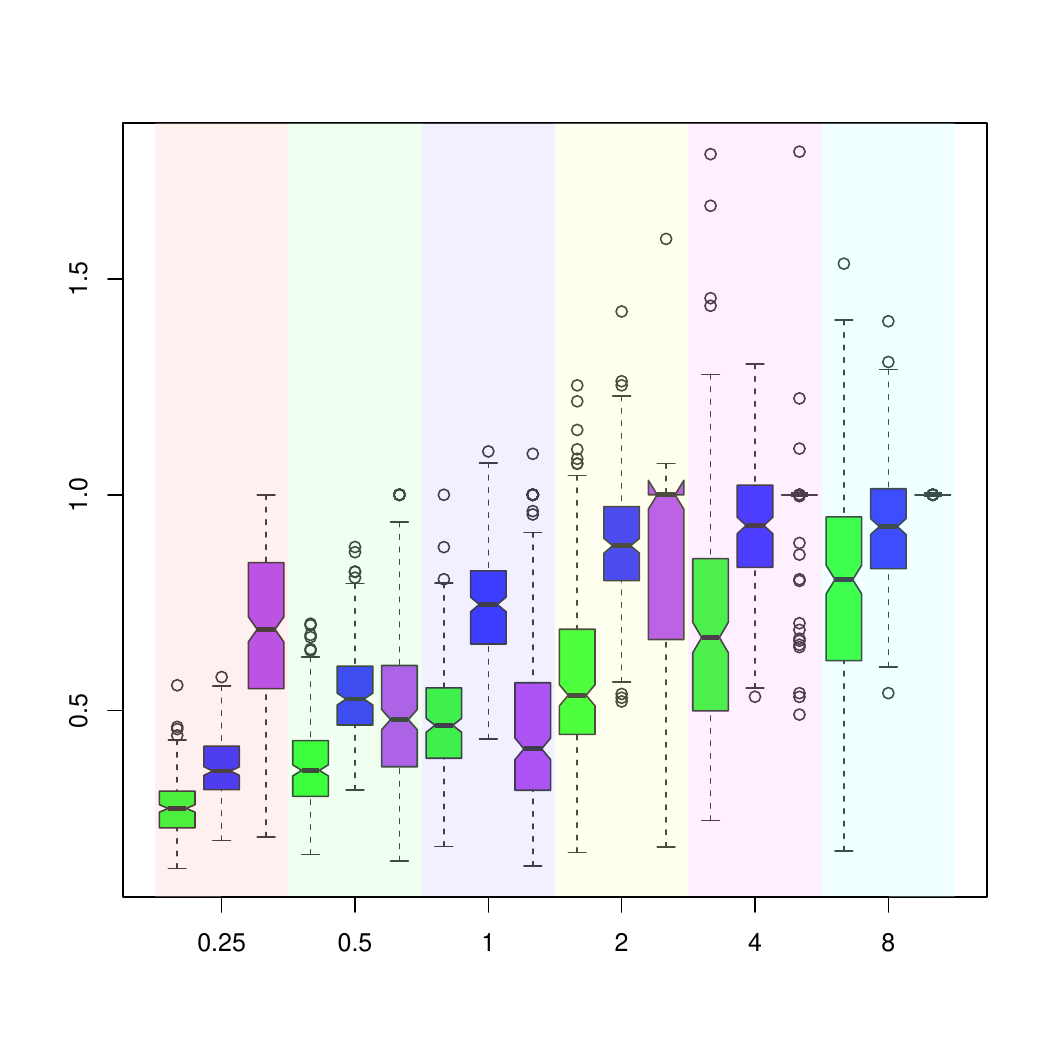}}
	\subfigure[Exp. decaying, $s = 5, q = 20$]{\includegraphics[width=0.33\textwidth]{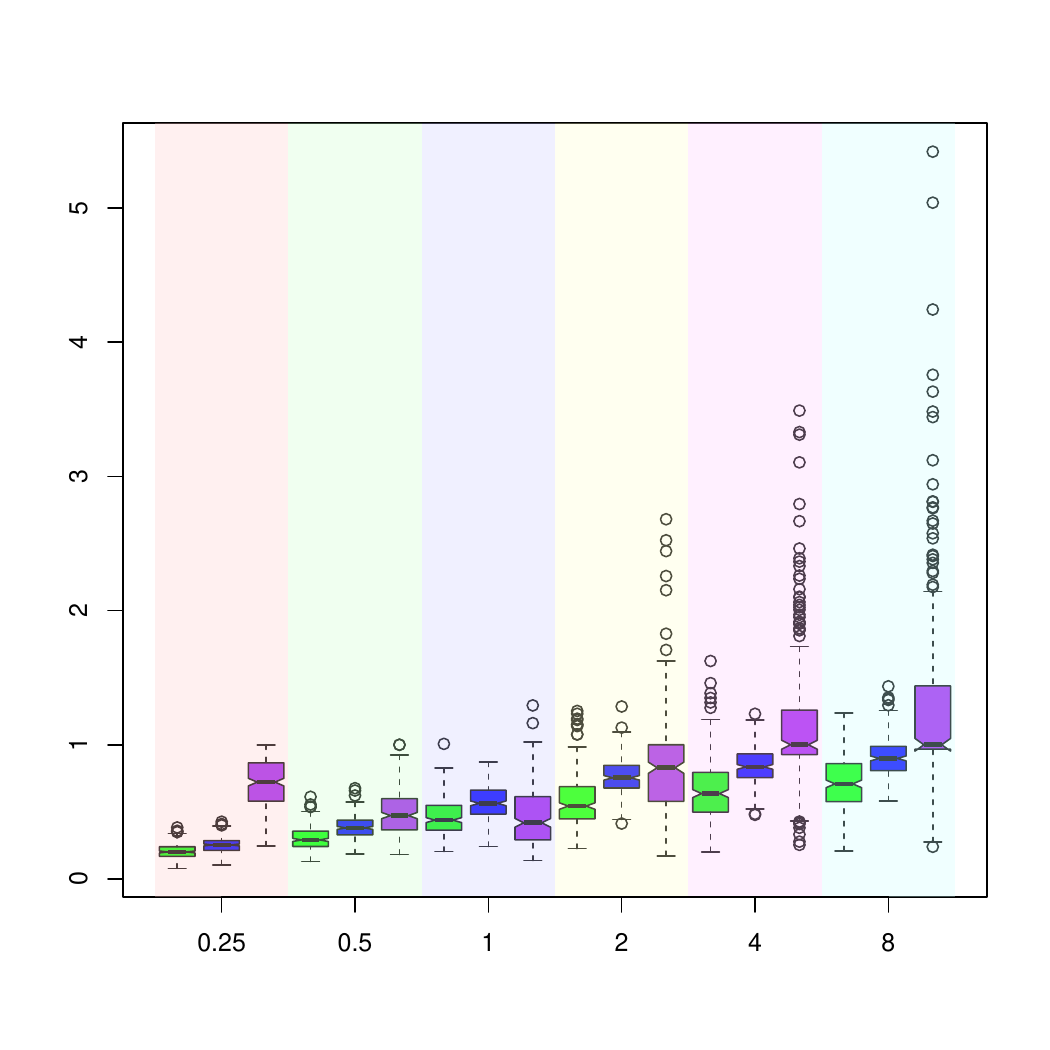}}
	\subfigure[Exp. decaying, $s = 10, q = 5$]{\includegraphics[width=0.33\textwidth]{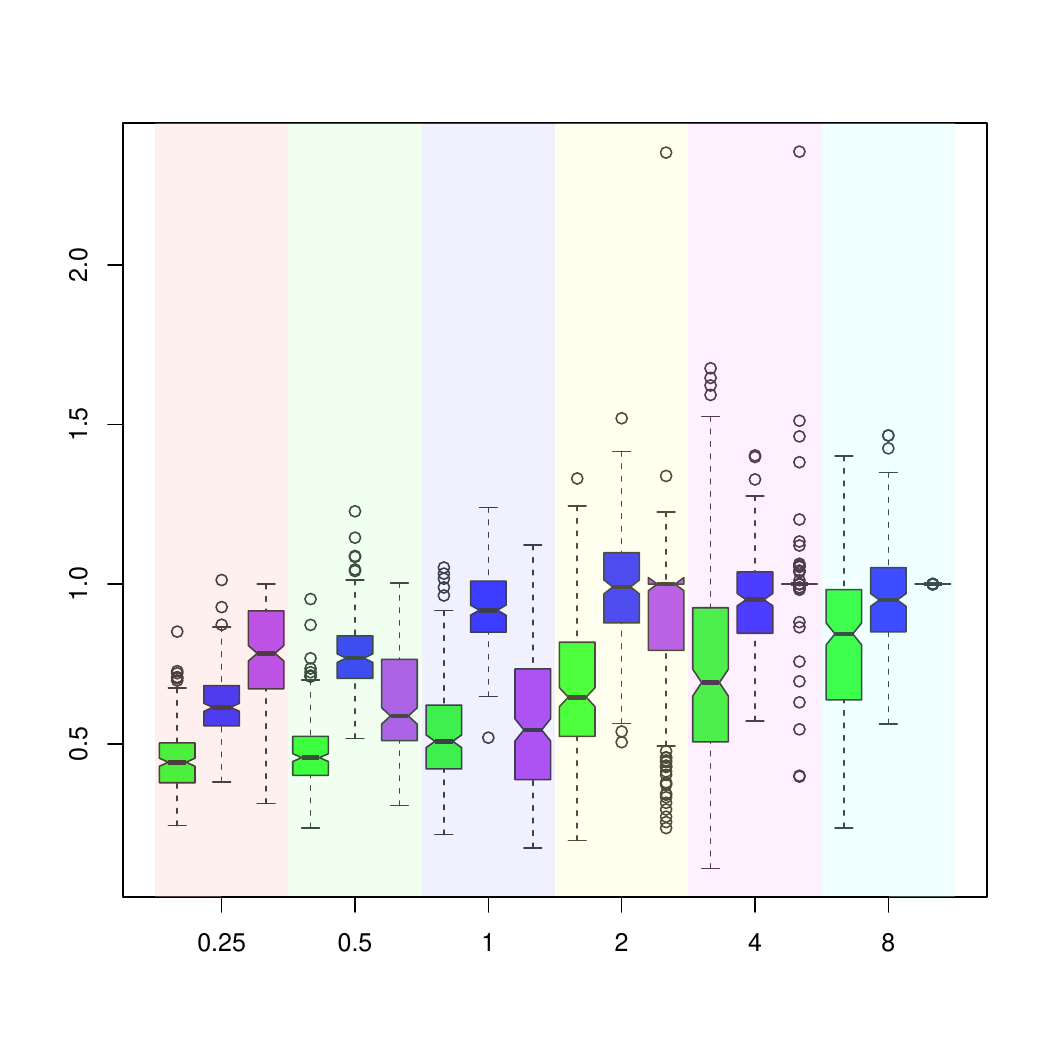}}
	\subfigure[Exp. decaying, $s = 10, q = 10$]{\includegraphics[width=0.33\textwidth]{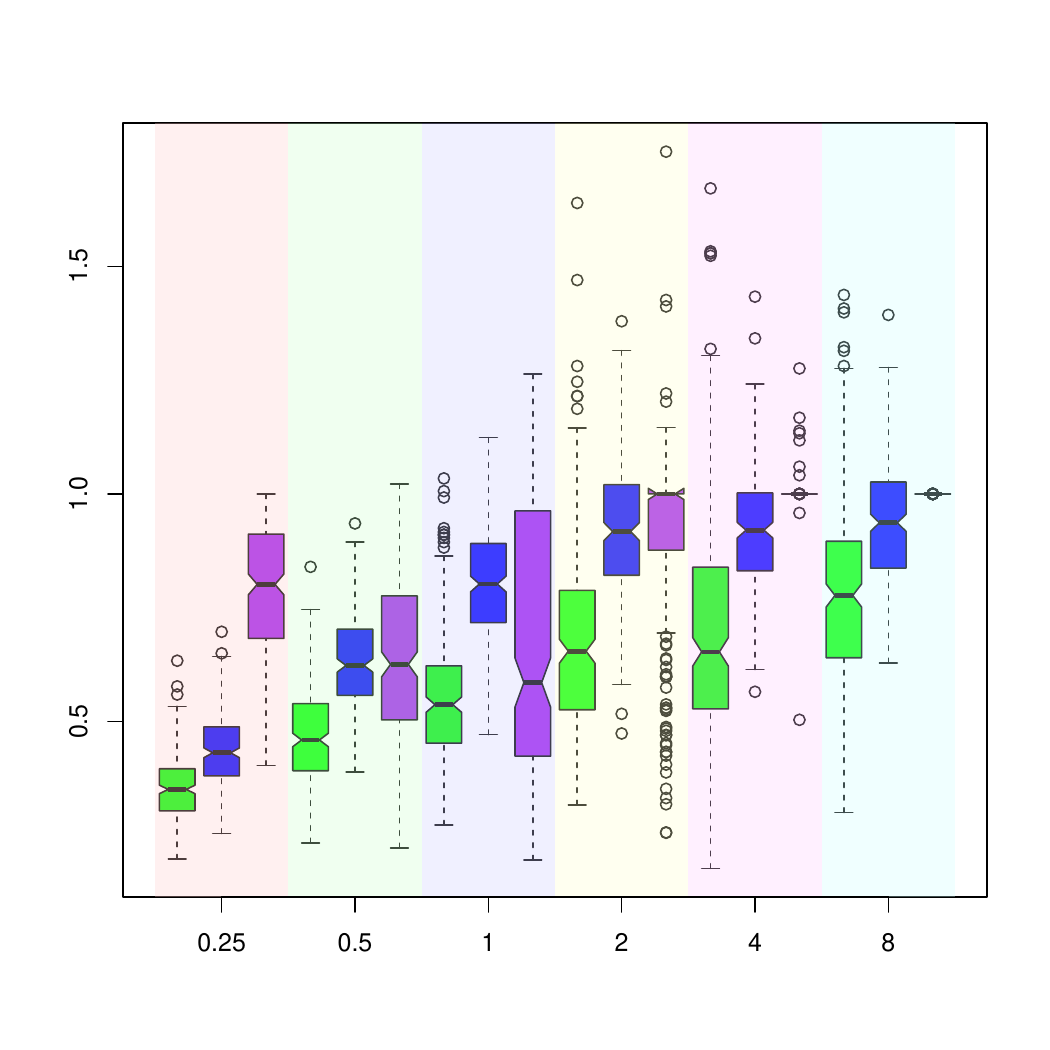}}
	\subfigure[Exp. decaying, $s = 10, q = 20$]{\includegraphics[width=0.33\textwidth]{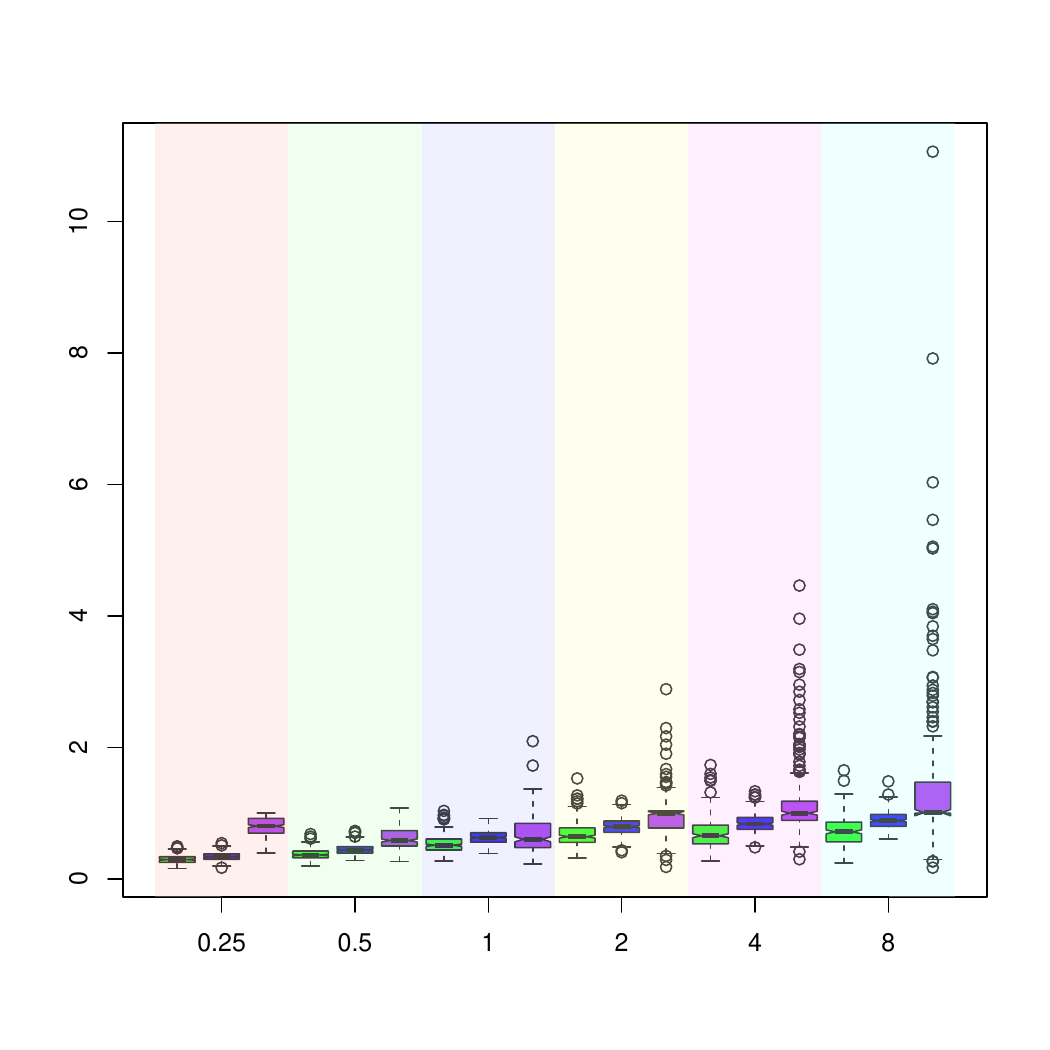}}
	\caption{Boxplots of errors \eqref{eq:err_metric} under exponential decay designs for generalised LAVA (green), the Lasso (blue) and the factor-based method of \citet{ouyang2023high} (purple).}\label{fig:logiexp}
\end{figure}

Figures~\ref{fig:logitoep} and \ref{fig:logiexp} show boxplots of the prediction errors in settings with Toeplitz and exponential decay designs respectively. 
From the result we can see that generalised LAVA on the whole tends to perform better than a naive application of the Lasso. The factor-based approach works well when the confounding strength $\nu = 1$, but appears to be the weakest of the three methods when the confounding strength is weak. We believe this is partly due to the difficulty in estimating the correct number of factors in these settings, which is a parameter required by the method; see Section~\ref{sec:fac_num} of the supplementary material. All the methods tend to perform worse for very strong confounding: this is due to the error metric \eqref{eq:err_metric} being less forgiving when $\nu$ is large.

\subsection{Inference} \label{sec:inf_exp}
Here we consider the performances of the testing approach of Section~\ref{sec:testing} and the factor-based approach of \citet{ouyang2023high}. We also compare to a version of our approach that sets $\lambda_2 = \infty$, and so effectively uses Lasso regressions instead of generalised LAVA. 

We generated datasets formed of i.i.d.\ copies of $(Y, X, W) \in \R \times \R^p \times \R$ where $X$ was generated as described in Section~\ref{sec:estimation} and we set $q=5, \nu=1$ and used the exponential design. Given the challenges of performing well-calibrated inference compared to estimation, here we took $n=4000$ and considered $p \in \{200,\, 400\}$. To generate $W$, we used a (confounded) linear model of the form
\[
W = X^{\top} \beta^W + U^{\top} \delta^W + \zeta
\]
where $\zeta \sim \mathcal{N}(0, 1)$ independently of $X$ and $U$. Similarly to the settings of Section~\ref{sec:estimation}, we set $\delta^W$ to be a unit vector proportional to $(1,\ldots,1)^{\top}$, and $\beta^W$ was generated (anew in each replication) using the same process as that of $\beta^0$. (In Section~\ref{sec:inf_exp2} of the supplementary material, we consider generating the components of $\delta^W$ as i.i.d.\ Rademacher random variables; the results are broadly in line with those here.) The response $Y$ was then generated according to a logistic regression model with
\[
\pr(Y=1 \given X, U, W) = \{1 + \exp(-X^\top \beta^0 - U^{\top} \delta^0 - b W)\}^{-1}
\]
where $b \in \{0, \,0.03,\, 0.06,\, 0.1,\, 0.13,\, 0.16,\, 0.2\}$; note that $b=0$ corresponds to the null $H_0: Y \independent W \given X, U$. We generated $\beta^0$ as in Section~\ref{sec:estimation} (independently of $\beta^W$), and considered settings where the sparsity level $s \in \{5, 10\}$ (for both $\beta^0$ and $\beta^W$).


\begin{table}[t!]
    \centering
    \begin{tabular}{c|c c c | c c c}
    \hline
    & \multicolumn{3}{c}{$p = 200$} & \multicolumn{3}{c}{$p = 400$} \\ \hline
    & Generalised LAVA & Lasso & Factor-based & Generalised LAVA & Lasso & Factor-based 
    \\ \hline
    $s = 5$ & 0.060 & 0.312 & 0.012 & 0.056 & 0.244 &  0.012 \\
    $s = 10$ & 0.052 & 0.244 & 0.028 & 0.084 & 0.188 & 0.056\\
    \hline
    \end{tabular}
    \caption{Average size of different methods estimated with $250$ replicates.}
    \label{tab:inference}
\end{table}


\begin{figure}[t!]
    \centering
    \subfigure[$p = 200, s = 5$]{\includegraphics[width=0.45\textwidth]{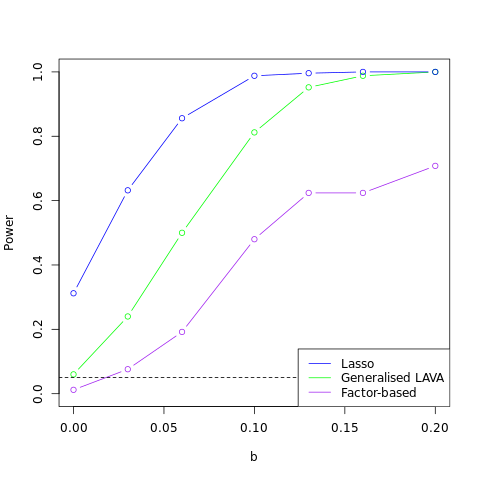}}
    \subfigure[$p = 200, s = 10$]{\includegraphics[width=0.45\textwidth]{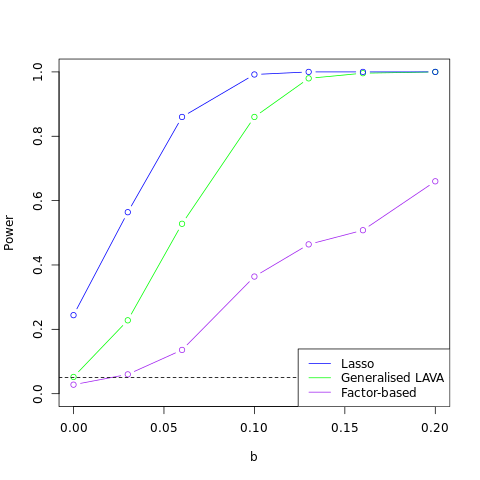}}
    \subfigure[$p = 400, s = 5$]{\includegraphics[width=0.45\textwidth]{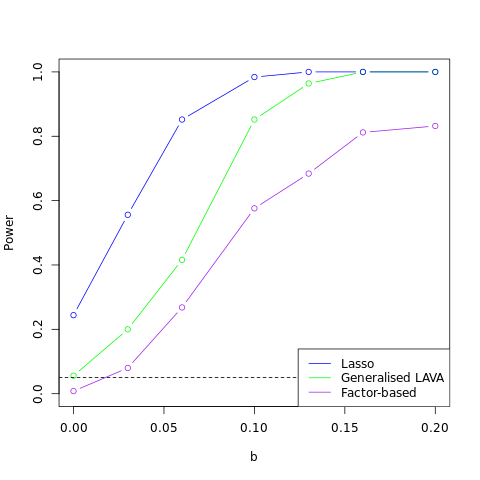}}
    \subfigure[$p = 400, s = 10$]{\includegraphics[width=0.45\textwidth]{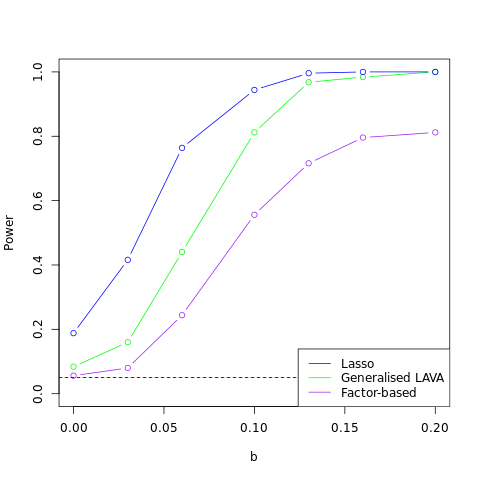}}
    \caption{Average power of the different methods for $b$ ranging from $0$ to $0.2$. The dashed line signifies $0.05$.}
    \label{fig:power}
\end{figure}

Table~\ref{tab:inference} shows the average sizes of different approaches across different regimes and average power curves are plotted in Figure~\ref{fig:power}. We see that the factor-based method maintains size control in most of the cases, but is usually conservative and as a result has lower power, while the naive Lasso approach is highly anti-conservative though powerful. The size of the test using generalised LAVA is also slightly above the $\alpha = 0.05$ nominal level with this largest in the $s = 10, p = 400$ case and perhaps here provides a reasonable trade-off in terms of power and size control, given the difficulties of the problem.


\subsection*{Acknowledgments}
Both authors were supported in part by an EPSRC `First Grant' EP/R013381/1 of the second author. The authors are grateful to Kaiyue Wen for his support with some aspects of the code.

\bibliographystyle{plainnat}

\bibliography{reference}
\newpage
\appendix

\counterwithin{figure}{section}
\section*{Supplementary material}

This supplementary material contains the proofs of Proposition~\ref{prop:theta_optim} (Section~\ref{sec:prop1}), Theorem~\ref{thm:main} (Section~\ref{sec:pfmain}), Theorems~\ref{thm:inpred} and~\ref{thm:prediction} (Section~\ref{sec:pfinf}), Theorem~\ref{thm:inference} (Section~\ref{sec:pfinference}) and Theorem~\ref{thm:resubg} (Section~\ref{sec:pfrsc}). Section~\ref{sec:subG} contains some basic results concerning sub-Gaussian random variables used in our proofs, Section~\ref{sec:bound_comments} provides some further discussion concerning Theorem~\ref{thm:resubg} and Section~\ref{sec:simu} presents some additional numerical results.

\section{Proof of Proposition~\ref{prop:theta_optim}} \label{sec:prop1}
		Writing $Q(\beta,b)$ for the LAVA objective, we have that
	\[
	0 \in \frac{\partial Q(\hat{\beta}, \hat{b})}{\partial \beta} -  \nabla_b Q(\hat{\beta},\hat{b}),
	\]
	so considering the $j$th component, we have
	\[
	\lambda_1 t = 2\lambda_2 \hat{b}_j
	\]
	where $t \in [-1,1]$ and $t = \sgn(\hat{\beta}_j)$ if $\hat{\beta}_j \neq 0$. Thus $\hat{\beta}_j = 0$ when $|\hat{b}_j| \leq 2\lambda_2 / \lambda_1$ and $\hat{b}_j = \sgn(\hat{\beta}_j) \lambda_1 / (2\lambda_2) $ when $\hat{\beta}_j \neq 0$. Now for any pair $\tilde{\beta}_j, \tilde{b}_j$ satisfying this property (with $\tilde{\beta}_j$ in place of $\hat{\beta}_j$ and $\tilde{b}_j$ in place of $\hat{b}_j$), which we call this property $A$ for future reference, we have
	\begin{align*}
		\rho_{\lambda_1,\lambda_2}(\tilde{\beta}_j + \tilde{b}_j) &= \lambda_1 \abs{\tilde{\beta}_j + \frac{\sgn(\tilde{\beta}_j) \lambda_1}{2\lambda_2}} - \frac{\lambda_1^2}{4\lambda_2} \\
		&= \lambda_1 |\tilde{\beta}_j| + \frac{\lambda_1^2}{4\lambda_2} =   \lambda_1 |\tilde{\beta}_j|  + \lambda_2 \tilde{b}_j^2.
	\end{align*}
	Therefore, writing $P(\theta)$ for the objective in \eqref{eq:theta_optim},
	\[
	Q(\tilde{\beta}, \tilde{b}) = P(\tilde{\beta} + \tilde{b}).
	\]
	Thus for any $\tilde{\theta} \in \R^p$, defining $(\tilde{\beta}, \tilde{b})$ by
	\begin{align*}
		\tilde{\beta}_j := \left\{\tilde{\theta}_j - \frac{\mathrm{sgn}(\tilde{\theta}_j)\lambda_1}{2\lambda_2} \right\} \ind_{\{|\tilde{\theta}_j| > \lambda_1 / (2\lambda_2)\}}
	\end{align*}
	and $\tilde{b} := \tilde{\theta} - \tilde{\beta}$,
	we have
	\[
	P(\tilde{\theta}) =Q(\tilde{\beta}, \tilde{b}) \geq Q(\hat{\beta}, \hat{b}) = P(\hat{\theta}),
	\]
	so $\hat{\theta}$ minimises $P$.
	
	Conversely, if $\hat{\theta}$ minimises $P$, then defining $(\hat{\beta}, \hat{b})$ as in the statement, we have for any $\tilde{\beta},\tilde{b}$ satisfying property $A$,
	\[
	Q(\hat{\beta}, \hat{b}) = P(\hat{\theta}) \leq P(\tilde{\beta}, \tilde{b}) = Q(\tilde{\beta},\tilde{b}),
	\]
	and we know any minimisers of the LAVA objective must satisfy property $A$.

\section{Proof of Theorem~\ref{thm:main}}\label{sec:pfmain}
We first introduce some notation used in the section. We write $e_j$ for the $j$th standard basis vector, and write $I$ for the identity matrix, where the dimensions of these will be clear from the context. For symmetric matrices $A, B$ of the same dimensions, we write $A \succeq B$ to indicate that $A-B$ is positive semi-definite. We use $\|A\|_{\op}$ and $\lambda_{\max}(A)$ interchangeably for the maximum singular value of an arbitrary matrix $A$, and write $\lambda_{\min}(A)$ for its minimum singular value.

For notational simplicity, we rewrite $\bX/\sqrt{n}$ as $\bX$, and analogously rescale all the other matrices and vectors $\bU$, $\bZ$, $\be$, $\bY$ relating to the observed data. We retain the original (unscaled) definitions of the random variables $X_i$ etc.\ forming these matrices, and take $X_{ij}$ to be the $j$th component of $X_i$ rather than the $ij$th entry of $\mb X$.

We also set
\[
c_d := \sup_t f'(t).
\]

Let $\hat{\beta}$ and $\hat{b}$ be the estimates from \eqref{eq:gen_LAVA}; note we have suppressed their dependence on $\lambda_1$ and $\lambda_2$. Let $b^0$ be 
the version of $\hat{b}$ obtained when fixing $\beta$ to be the true $\beta^0$ and only optimising over $b$:
\begin{align*}
b^0 := \argmin_{b\in \R^p}\{ L(b + \beta^0) + \lambda_2 \|b\|_2^2\}.
\end{align*}
Observe that then
\begin{align*}
L(\hat{b} + \hat{\beta}) + \lambda_1 \|\hat{\beta}\|_1 + \lambda_2 \|\hat{b}\|_2^2 \leq L(b^0 + \beta^0) + \lambda_1 \|\beta^0\|_1 + \lambda_2 \|b^0\|_2^2.
\end{align*}
By performing a Taylor expansion of $L$ about $b^0 + \beta^0$, we may rewrite the above inequality as
\begin{align}\label{eq:preopt}
\rem\big(\hat{\beta} + \hat{b}, \beta^0 + b^0\big) + \lambda_1 \|\hat{\beta}\|_1 +  &\lambda_2 \|\hat{b}\|_2^2 - \lambda_1 \|\beta^0\|_1 - \lambda_2 \|b^0\|_2^2 \notag\\
&\leq - (\hat{\beta} - \beta^0)^\top \nabla L(b^0 + \beta^0)  - (\hat{b} - b^0)^\top \nabla L(b^0 + \beta^0) ,
\end{align}
where $\rem$ is the remainder term.
Note that due to convexity of the loss function, $\rem$ is non-negative everywhere.

Next, using that
\begin{equation} \label{eq:b^0_deriv}
	-\nabla L(b^0 + \beta^0) = 2 \lambda_2 b^0
\end{equation}
and that
\[
\lambda_2 \|\hat{b}\|_2^2 = \lambda_2 \|\hat{b} - b^0\|_2^2 + 2 \lambda_2 (b^0)^\top (\hat{b} - b^0) + \lambda_2 \|b^0\|_2^2,
\]
we may rewrite inequality \eqref{eq:preopt} to obtain
\begin{equation}\label{eq:opt}
	\begin{split}
\rem\big(\hat{\beta} + \hat{b}, \beta^0 + b^0\big) + \lambda_1 \|\hat{\beta}\|_1 - \lambda_1 \|\beta^0\|_1 + \lambda_2 \|\hat{b} - b^0\|_2^2 &\leq 2 \lambda_2 (b^0)^\top (\hat{\beta} - \beta^0) \\
&\leq 2 \lambda_2 \| b^0 \|_\infty \| \hat{\beta} - \beta^0 \|_1,
\end{split}
\end{equation}
applying H\"older's inequality in the final line.
Equipped with this inequality, the rest of proof is organised as follows. We first state two key technical lemmas, Lemmas~\ref{lem:beta} and~\ref{lem:rsc}, of which the proofs are deferred to
Section~\ref{sec:lem_main}
as well as Lemmas~\ref{lem:subr} and \ref{lem:rscwain} relating to convex functions. We then use these to obtain our main result, which is proved in Section~\ref{sec:thm_main}. To introduce the results below, recall that
\[
C(S) := \{\Delta \in \R^p: \|\Delta_{S^c}\|_1 \leq 3 \|\Delta_{S}\|_1\}.
\]

\begin{lemma} \label{lem:beta}
Assume Assumptions~\ref{cd:pns}--\ref{cd:subg}.
Given $m \in \N$ and $c_{\lambda_2} > 0$, there exists $c, c' > 0$ such that on an event $\Omega_1$ with probability at least $1-c'n^{-m}$,
\[
\lambda_2 \|b^0 \|_\infty \leq c_b \sqrt{\frac{\log p}{n}}.
\]
\end{lemma}

\begin{lemma} \label{lem:rsc}
Assume Assumptions~\ref{cd:pns}--\ref{cd:subg}. Given $m \in \N, \tau, \kappa, c_p > 0$. Define
\[
b(\beta) = \underset{b}{\argmin}\; L(b + \beta) + \lambda_2 \|b\|_2^2.
\]
Then there exist constants $r, c', c_{\lambda_2} > 0$ depending on $m, \tau, \kappa, c_p$ such that by choosing $\lambda_2 = c_{\lambda_2} r_{\ell} \sqrt{\frac{\log p}{n}}$, on an event $\Omega_2$ with probability at least $1 - \PP(\Omega^c(\tau, \kappa, c_p)) - c' n^{-m}$, for any $\beta \in \R^p$ with $\beta - \beta^0 \in C(S)$ and 
\begin{align}\label{eq:ballsize}
	\|\beta-\beta^0\|_2 \leq r \left( \sqrt{s \log p} + \frac{\sqrt{spn \log p}}{r_{\lo}}\right)^{-1},
\end{align}
we have that
\begin{align*}
\rem(\beta + b(\beta), \beta^0 + b^0) + \lambda_2 \|b(\beta) - b^0\|_2^2 \geq \kappa \|\beta - \beta^0\|_2^2 / 2.
\end{align*}
\end{lemma}

\begin{lemma}\label{lem:subr}
	Let $\psi : \R^{l} \times \R^m \to \R$ be a convex function. Then $\phi:\R^{l} \to \R$ given by
	\[
	\phi(u) = \inf_{w \in \R^{m}} \psi(u,w)
	\]
	is also a convex function.
\end{lemma}
\begin{proof}
	Fix $t \in [0,1]$ and $u,v \in \R^l$. Given $\epsilon > 0$, let $w_u \in \R^m$ and $w_v\in \R^m$ be such that
	\[
	\phi(u) > \psi(u, w_u) - \epsilon \qquad \text{and} \qquad \phi(v) > \psi(v, w_v) - \epsilon.
	\]
	Then
	\begin{align*}
		\phi\big(tu + (1-t)v\big) &= \inf_{w \in \R^{m}} \psi\big(tu + (1-t)v,w\big) \\
		&\leq \psi\big(tu + (1-t)v,tw_u + (1-t) w_v\big) \\
		&\leq t\psi(u, w_u) + (1-t)\psi(v,w_v) \\
		&< t\phi(u) + (1-t)\phi(v) - \epsilon,
	\end{align*}
using the convexity of $\psi$ in the penultimate line.
	As $\epsilon > 0$ was arbitrary, we have that $\phi\big(tu + (1-t)v\big) \leq t\phi(u) + (1-t)\phi(v)$, and as $t \in [0,1]$ and $u$ and $v$ were arbitrary, we conclude that $\phi$ is convex as required.
\end{proof}

\begin{lemma}\label{lem:rscwain}
	Let $\psi:\R^p \to \R$ be a convex function and suppose $\Delta^*$ is a minimiser of $\psi$. Let cone $C \subseteq \R^p$ be such that $\Delta^* \in C$. If there exists $r>0$ such that
		\[
		\text{for all $\Delta \in C$ with $\|\Delta\|_2=r$ we have $\psi(\Delta)-\psi(0) > 0$},
		\]
	then $\|\Delta^*\|_2 \leq r$.
\end{lemma}
\begin{proof}
	The proof proceeds similarly to \citet[Lem.~9.21]{Wain19}. If $\Delta^*=0$ we are done, so we we may assume $\|\Delta^*\|_2 > 0$. Now let $\alpha = r / \|\Delta^*\|_2$. Then we have that $\alpha \Delta^* \in C$ and so $\psi(\alpha \Delta^*) - \psi(0) > 0$. Now if, for a contradiction, $\alpha < 1$, then by convexity of $R$,
	\[
	\psi(\alpha \Delta^*) \leq (1-\alpha)\psi(0) + \alpha \psi(\Delta^*) < (1-\alpha) \psi(\alpha \Delta^*) + \alpha \psi(\Delta^*).
	\]
This gives $\psi(\alpha \Delta^*) < \psi(\Delta^*)$, a contradiction, so we must have $\alpha \geq 1$, i.e.,  $\|\Delta^*\|_2 \leq r$ as required.
\end{proof}


\subsection{Proof of Theorem~\ref{thm:main}} \label{sec:thm_main}
\begin{proof}
	We work on $\Omega_1 \cap \Omega_2$, with these events given by Lemmas~\ref{lem:beta} and \ref{lem:rsc} respectively. 
Let function $\cF :\R^p \to \R$ be given by
\[
\cF(\Delta) := L(b(\beta^0 + \Delta) + \beta^0 + \Delta) - L(b^0 + \beta^0) + \lambda_2 (\| b(\beta^0 + \Delta)\|_2^2 - \| b^0\|_2^2) + \lambda_1 (\| \beta^0 + \Delta\|_1 - \| \beta^0\|_1).
\]
By performing Taylor expansion about $b^0 + \beta^0$, we have
\begin{align*}
	\cF(\Delta) = \;& \rem\big(\beta^0 + \Delta + b(\beta^0 + \Delta),  \beta^0 + b^0\big) + \lambda_1 \|\beta^0 + \Delta\|_1 + \lambda_2 \|b(\beta^0 + \Delta)\|_2^2 - \lambda_1 \|\beta^0\|_1 - \lambda_2 \|b^0\|_2^2 \\
	& + \nabla L (\beta^0 + b^0) \Delta + \nabla L ^\top (\beta^0 + b^0) (b(\beta^0 + \Delta) - b^0).
\end{align*}
It then follows from the same analysis as from~\eqref{eq:preopt} to~\eqref{eq:opt} that
\begin{align*}
\cF(\Delta) \geq \;& \rem\big(\beta^0 + \Delta + b(\beta^0 + \Delta), \beta^0 + b^0\big) + \lambda_1 \|\beta^0 + \Delta\|_1 - \lambda_1 \|\beta^0\|_1 \\
& + \lambda_2 \|b(\beta^0 + \Delta) - b^0\|_2^2 - 2 \lambda_2 \|b^0\|_\infty \|\Delta\|_1 \\
\geq \;& \rem\big(\beta^0 + \Delta + b(\beta^0 + \Delta), \beta^0 + b^0\big) + \lambda_2 \|b(\beta^0 + \Delta) - b^0\|_2^2 - (2 \lambda_2 \|b^0\|_\infty + \lambda_1)  \|\Delta\|_1.
\end{align*}

Let $c_{\lambda_1}$ be a constant such that
\begin{equation} \label{eq:c_lambda_1}
	c_{\lambda_1} \sqrt{\frac{\log p}{n}} =: \lambda_1 \geq 4 c_b \sqrt{\frac{\log p}{ n}} \geq \lambda_2 \|b^0\|_\infty,
\end{equation}
with the last inequality a consequence of working on $\Omega_1$. 
Then we have that for a constant $c>0$,
\[
\cF(\Delta) \geq \rem\big(\beta^0 + \Delta + b(\beta^0 + \Delta), \beta^0 + b^0\big) + \lambda_2 \|b(\beta^0 + \Delta) - b^0\|_2^2 - c \sqrt{\frac{\log p}{n}} \|\Delta\|_1.
\]

Equipped with the above inequality, we next prove that $\hat{\beta} - \beta^0$ satisfies \eqref{eq:ballsize}. To achieve this goal, we now consider the $\Delta$ within the region 
\[
\mathcal{K} := \left\{\Delta: \left( \sqrt{s \log p} + \frac{\sqrt{spn \log p}}{r_{\lo}}\right) \|\Delta\|_2 = r\right\} \cap C(S),
\]
where $r$ is defined in Lemma~\ref{lem:rsc}.
As $\Delta \in \mathcal{K}$, it certainly satisfies \eqref{eq:ballsize}. Thus as we are on $\Omega_2$,
\[
\cF(\Delta) \geq \kappa' \|\Delta\|_2^2 - c \sqrt{\frac{\log p}{n}} \|\Delta\|_1,
\]
writing $\kappa' := \kappa/2$ for simplicity.

Notice further that as $\Delta \in C(S)$, we have that $\|\Delta\|_1 \leq 4\sqrt{s} \|\Delta\|_2$ (see \eqref{eq:sparse_ineq}), and therefore
\[
\cF(\Delta) \geq \kappa' \|\Delta\|_2^2 - c \sqrt{s \frac{\log p}{n}} \|\Delta\|_2 = \left(\frac{r \kappa'}{\sqrt{s \log p} + \sqrt{spn \log p} / r_{\lo}} - c \sqrt{s \frac{\log p}{n}}\right)\|\Delta\|_2> 0,
\]
where the last inequality holds for all $n$ sufficiently large due to Assumption~\ref{cd:pns}(iv). 

From~\eqref{eq:opt}, we have that
\[
\lambda_1 (\|\hat{\beta}\|_1 - \|\beta^0\|_1) \leq 2 \lambda_2 \|b^0\|_\infty \|\hat{\beta} - \beta^0\|_1,
\]
and so by \eqref{eq:c_lambda_1}, we have that
\[
2(\|\hat{\beta}\|_1 - \|\beta^0\|_1) \leq \|\hat{\beta} - \beta^0\|_1.
\]
Rearranging and then applying the triangle inequality, we see that
\[
\|\hat{\beta}_{S^c}\|_1 \leq \|\hat{\beta}_S - \beta^0_S\|_1 + 2 \|\beta^0_S\|_1 - 2\|\hat{\beta}_S\|_1  \leq \|\hat{\beta}_S - \beta^0_S\|_1,
\]
so $\hat{\beta} - \beta^0 \in C(S)$.

Now observe the as a Lemma~\ref{lem:subr}, $\cF$ is a convex function. Furthermore, it is minimised by $\hat{\beta} - \beta^0$.
By combining this with the fact that $\cF(\Delta) > 0$ for all $\Delta \in \mathcal{K}$, we are able to apply Lemma~\ref{lem:rscwain} to get that $\hat{\beta} - \beta^0$ satisfies \eqref{eq:ballsize}.

From~\eqref{eq:opt} and Lemma~\ref{lem:rsc}, we then have that
\[
\kappa' \|\hat{\beta} - \beta^0\|_2^2 + \lambda_1 \|\hat{\beta}\|_1 - \lambda_1 \|\beta^0\|_1 \leq 2 \lambda_2 \| b^0 \|_\infty \| \hat{\beta} - \beta^0 \|_1,
\]
which means that there exists some constant $c > 0$ such that
\[
\kappa'\|\hat{\beta} - \beta^0\|_2^2 \leq \left(\lambda_1 + 2 \lambda_2 \| b^0 \|_\infty\right) \|\hat{\beta} - \beta^0\|_1 \leq c \sqrt{\frac{\log p}{n}} \|\hat{\beta} - \beta^0\|_1 \leq c \sqrt{s \frac{\log p}{n}} \|\hat{\beta} - \beta^0\|_2.
\]
Dividing by $\|\hat{\beta} - \beta^0\|_2$ then gives the result.
\end{proof}

\subsection{Proofs of Lemmas~\ref{lem:beta} and~\ref{lem:rsc}}\label{sec:lem_main}

We begin by introducing the some events we will work on in the proofs of Lemmas~\ref{lem:beta} and~\ref{lem:rsc}. To introduce these events, it is helpful to first derive an explicit form for $b^0$, which we do below.

Recall \eqref{eq:b^0_deriv} that
\[
2 \lambda_2 b^0 = -\nabla L(\beta^0 + b^0) = \frac{1}{n} \sum_{i=1}^n \{Y_i -  f(X_i^\top (\beta^0 + b^0))\} X_i. 
\]
Let us write
\begin{equation} \label{eq:L'}
	L' := \frac{1}{n} \sum_{i=1}^n \{f(X_i^\top \beta^0 + U_i^\top \delta^0) - Y_i\} X_i = - \frac{1}{n} \sum_{i=1}^n \varepsilon_i X_i.
\end{equation}
Applying the mean value theorem, we have that for all $i = 1, \ldots, n$, there exists $t_i \in [0,1]$ such that
\[
2 \lambda_2 b^0 =-L' - \frac{1}{n} \sum_{i=1}^n f'(X_i^\top \beta^0 + t_i U_i^\top \delta^0 + (1 - t_i) X_i^\top b^0) \cdot (X_i^\top b^0 - U_i^\top \delta^0) X_i.
\]
Let $\Lambda^0 \in \R^{n \times n}$ and $\bar{\Lambda} \in \R^{n\times n}$ be diagonal matrices such that
\begin{align} \label{eq:Lambda_0_def}
	\Lambda^0_{ii} &:= f'(X_i^\top \beta^0 + U_i^\top \delta^0), \\
	\bar{\Lambda}_{ii} &:= f'(X_i^\top \beta^0 + t_i U_i^\top \delta^0 + (1 - t_i) X_i^\top b^0). \label{eq:Lambda_bar_def}
\end{align}
Then
\begin{equation}\label{eq:b^0_decomp_prelim}
 2 \lambda_2 b^0 = - L' - \frac{1}{n} \sum_{i=1}^n \bar{\Lambda}_{ii} X_i (X_i^\top b^0 - U_i^\top \delta^0) = - L' - \bX^\top \bar{\Lambda} \bX b^0  - \bX^\top \bar{\Lambda} \bU \delta^0.
\end{equation}
From~\eqref{eq:b^0_decomp_prelim} we may obtain
\begin{equation}\label{eq:b^0_decomp}
	\lambda_2 b^0 = - \underset{=: Q_1}{\underbrace{\lambda_2 \left( 2 \lambda_2 I + \bX^\top \bar{\Lambda} \bX\right)^{-1}L'}} + \underset{=: Q_2}{\underbrace{\lambda_2 \left( 2 \lambda_2 I + \bX^\top \bar{\Lambda} \bX\right)^{-1} \bX^\top \bar{\Lambda} \bU \delta^0}}.
\end{equation}
The proof of Lemma~\ref{lem:beta} will proceed by controlling $\lambda_2 \|b^0\|_\infty$ via bounding $\|Q_1\|_\infty$ and $\|Q_2\|_\infty$ separately.

We now introduce the following sets of diagonal matrices (note that the first set of matrices is random):
\begin{gather}
	\label{eq:L_tilde_def}
	\tilde{\cL} := \{ \text{diagonal } \Lambda \in \R^{n \times n}: \exists\; t_1, \cdots, t_n \in [0,1] \;\text{with}\; \Lambda_{ii} = t_i \bar{\Lambda}_{ii} + (1 - t_i) (\Lambda^0)_{ii} \; \forall i \in [n]\}, \\
	\label{eq:L_def}
	\mathcal{L} := \{ \text{diagonal } \Lambda \in \R^{n \times n} : \min_{i} \Lambda_{ii} \geq c_d^{-1}\}.
\end{gather}

We now introduce the following sets of events indexed by a parameter $c_t \geq 0$:
\begin{align*}
	\cT_1(c_t)  &:= \left\{
	\forall \Lambda \in \cL, i \in [n], j \in [p], \quad 
	\left|X_i^\top \bX^\top (\lambda_2 \Lambda + \bX \bX^\top)^{-1} \bX_j \right| \leq c_t (1 + c_{\lambda_2}^{-1}) \left(\sqrt{\log n} + \frac{\sqrt{pn \log p}}{r_{\lo}} \right)\right\}, \\
		\cT_2(c_t)  & := \left\{\left\| (2 \lambda_2 \bar{\Lambda}^{-1} + \bX \bX^\top )^{-1} \bU \delta^0 \right\|_2 \leq \frac{c_t}{r_\ell}, \right.\\ 
		& \qquad \qquad \left.\|\bX b^0 - \bU \delta^0 \|_2 \leq c_t (1 + c_{\lambda_2}^{-1}) \left(\sqrt{\frac{p \log p}{r_{\lo} n}} + \frac{p}{r_\lo \sqrt{n}}\right) + c_t c_{\lambda_2}\sqrt{\frac{\log p}{n}} \right\},\\
	\cT_3(c_t) &:= \left\{\forall \Lambda \in \tilde{\cL}, \; \left\|\bX \left(2 \lambda_2 I + \bX^\top \Lambda \bX \right)^{-1} L'\right\|_2 \leq c_t \frac{p}{r_{\lo} \sqrt{n}}  + c_t \sqrt{\frac{p \log p}{r_{\lo} n}}\right\}, \\
	\cT_4(c_t) &:= \left\{\|L'\|_\infty \leq c_t \sqrt{\frac{\log p}{n}}, \;\; \max_{i,j} |X_{ij}| \leq c_t \sqrt{\log p}, \;\; \forall j = 1, \ldots, p,\, \|\bX_j\|_2 \leq c_t \right. \\
	& \qquad \qquad \left.\textrm{and} \;\; \left\|\bX^\top \left(2 \lambda_2 (\Lambda^0)^{-1} + \bX\bX^\top\right)^{-1} \bX L' \right\|_\infty \leq c_t \sqrt{\frac{\log p}{n}}\right\}.
\end{align*}

The rest of proof is organised as follows. In Section~\ref{sec:prelim}, we provide some auxiliary lemmas necessary for the proof of Lemmas~\ref{lem:beta} and~\ref{lem:rsc}; in Sections~\ref{sec:T_1}--\ref{sec:T_4}, we control the probabilities of the events defined above.
Lemmas~\ref{lem:beta} and \ref{lem:rsc} are then proved in Sections~\ref{sec:beta} and \ref{sec:rsc} respectively.

\subsubsection{Preliminary lemmas}\label{sec:prelim}

\begin{lemma}\label{lem:Loewner}
	Let $A \in \R^{p \times p}$ be a symmetric positive definite matrix and let $B \in \R^{p \times p}$ be a symmetric positive semi-definite matrix. Then for any choice of vector $x \in \R^p$, we have
	\begin{align*}
		\|(A + B)^{-1} x\|_2 \leq \|A^{-1} x\|_2.
	\end{align*}
\end{lemma}
\begin{proof}
	Observe that for a positive semi-definite matrix $H \in \R^{p \times p}$,
	\begin{align*}
		\|(A + H)^{-1} x\|_2^2 &= x^\top (A + H)^{-2} x \\
		& = x^\top (A^{-1} - A^{-1} H A^{-1})^2 x + o(\|H\|)\\
		&= x^\top (A^{-1} - A^{-1} H A^{-1})^2 x = x^\top A^{-1} (I - A^{-1 / 2} H A^{-1 / 2})^2 A^{-1} x + o(\|H\|)\\
		& = x^\top A^{-1} (I - 2 A^{-1 / 2} H A^{-1 / 2}) A^{-1} x + o(\|H\|)\\
		& = x^\top (A^{-2} - 2 A^{-3 / 2} H A^{-3 / 2}) x+ o(\|H\|).
	\end{align*}
	Thus viewing $H$ as a vector, the derivative of the function $T \mapsto \|T^{-1} x\|_2^2$ at $A$ is equal to $- 2 (x^\top A^{-3 / 2}) \otimes (A^{-3 / 2} x)$, where $\otimes$ denotes the Kronecker product.
	The mean value theorem then gives us that there exists some $0 \leq \eta \leq 1$ such that
	\begin{align*}
		\|(A + B)^{-1} x\|_2^2 - \|A^{-1} x\|_2^2 = -2 x^\top (A+\eta B)^{-3 / 2} B (A+\eta B)^{-3 / 2} x.
	\end{align*}
	As $(A+\eta B)^{-3 / 2} B (A+\eta B)^{-3 / 2}$ is a positive semi-definite matrix, the LHS of the above equality is always non-positive, and the desired result follows.
\end{proof}

\begin{lemma}\label{lem:subgx} 
Assume Assumptions~\ref{cd:pns}--\ref{cd:subg} and suppose $m \in \N$ is given. 
Given $c_t>0$, let  event $\mathcal{A}_1(c_t)$ be defined by the following conditions:
\begin{gather}
	\lambda_{\max} (\bZ) \leq c_t \sqrt{\frac{p}{n}},  \qquad\quad c_t^{-1} \leq \lambda_{\min} (\bU) \leq \lambda_{\max} (\bU) \leq c_t, \label{eq:subgx3} \\
	\max_{i,j} |X_{ij}| \leq c_t\sqrt{\log p} \label{eq:subgx2} \\
	\forall j = 1, \ldots, p, \quad \|\bX_j\|_2 \leq c_t, \label{eq:subgx1} \\
	\max_{i} \|\bU(\bU^\top \bU)^{-1} \bU^\top e_i\|_2 \leq c_t \sqrt{\frac{\log n}{n}} \label{eq:U_proj} \\
	\max_{i \neq j} |Z_i^\top X_j| \leq c_t \sqrt{p} \log p, \qquad\quad \max_i |Z_i^\top X_i| \leq c_t p \log p, 
	\label{eq:subgx4} \\
	\max_{j,k}|(\bZ^\top \bZ)_{jk} - \Sigma_{jk}| \leq c_t \sqrt{\frac{\log p}{n}}, \label{eq:subgx7} \\
	\|L'\|_2 \leq c_t \sqrt{\frac{p \log p}{n}}. \label{eq:subgx5} \\
		\forall \Lambda\; \in \mathcal{L} \text{ and } \forall \; c_{\lambda_2} > 0, \;\;
	\left\|\bX^\top \left(\lambda_2 \Lambda + \bX\bX^\top\right)^{-1}\right\|_{\op} \leq \frac{c_t (1 + c_{\lambda_2}^{-1})}{\sqrt{r_{\lo}}} \left(\sqrt{\frac{p}{r_\lo \log p}} \vee 1\right). \label{eq:subgx6}
\end{gather}
(Note that $\lambda_2$ depends on $c_{\lambda_2}$.)
Then for all $c_t$ sufficiently large, there exists $c' >0$ such that $\pr(\mathcal{A}_1(c_t)) \geq 1- c'n^{-m}$.
\end{lemma}
\begin{proof} Let $c_t > 0$ be given. We will use $a \lesssim b$ to denote that there exists constant $C >0$ potentially depending on $c_t$ and other constants such that $a \leq C b$. Note that it suffices to show a high probability bound relating to each of \eqref{eq:subgx3}--\eqref{eq:subgx6} as the intersection of the corresponding events will then have the required high probability bound after adjusting the constant $c'$. Below we treat each of the inequalities in turn.
	
	\eqref{eq:subgx3}: From \citet[Rem.~5.40]{vershynin_2012}, we have that there exists some constant $\sigma$ such that
	\begin{align*}
	\pr\left(\|\bZ^\top \bZ - \var(Z_i)\|_\op \geq \sigma \left(\sqrt{\frac{p + t}{n}} \vee \frac{p + t}{n}\right)\right) &\leq 2\exp(-t), \\
		\pr\left(\|\bU^\top \bU - \var(U_i)\|_\op \geq \sigma \left(\sqrt{\frac{q + t}{n}} \vee \frac{q + t}{n}\right)\right) &\leq 2\exp(-t).
	\end{align*}
	Recall that $q \lesssim 1$, setting $t = \log n$ yields the desired result.
	
	\eqref{eq:subgx2}: Using a sub-Gaussian tail bound, we have that for all $i = 1, \ldots, n$ and  $j = 1, \ldots, p$,
	\[
	\pr(|X_{ij}| \geq \sigma_x t) \leq 2 \exp(-t^2).
	\]
	Taking a union bound with $t = \sqrt{\log(p n)} \lesssim \sqrt{\log p}$ yields the desired result.

\eqref{eq:subgx1}: Using Lemma~\ref{lem:subgbern}, we have that there exists some constants $(\sigma, b)$ such that for each $j$,
\[
\pr(\|\bX_j\|_2^2 - \var(X_j) \geq t) \leq \exp\left( -\frac{nt^2}{2 (\sigma^2 + bt)}\right).
\]
Then by setting $t = 2 \sigma \sqrt{\log p / n}$ and taking the union bound over all $j=1,\ldots,p$, we obtain the desired result.

\eqref{eq:U_proj}: We have
\[
\|\bU(\bU^\top \bU)^{-1} \bU^\top e_i\|_2 = \sqrt{e_i^\top \bU(\bU^\top \bU)^{-1} \bU^\top e_i} = \frac{1}{\sqrt{n}} \sqrt{U_i^\top (\bU^\top \bU)^{-1} U_i}.
\]
Now similarly to the above, we have from Lemma~\ref{lem:subgbern} and Assumption~\ref{cd:pns}(i) that $\max_i \|U_i\|_2 \lesssim \sqrt{ \log n}$ with probability at least $1-c'n^{-m}$, for some constant $c'>0$. Thus, working also on the event given by \eqref{eq:subgx3}, we have that the above display is bounded above by a constant times $\sqrt{\log n / n}$.

\eqref{eq:subgx4}: We consider the first expression to begin with. Note that $X_i$ and $Z_j$ are independent for $i \neq j$. Thus by the sub-Gaussianity of $Z_i$, we have that there exists constant $c$ such that
\[
\pr(|Z_i^\top X_j| \geq t \|X_j\|_2 \given X_j) \leq 2\exp( -c t^2).
\]
Thus taking expectations, we obtain
\[
\pr(|Z_i^\top X_j| \geq t \|X_j\|_2) \leq 2\exp( -c t^2).
\]
Now on the event $\Omega$ relating to \eqref{eq:subgx2}, we have in particular that $\max_i \|X_i\|_2 \leq c' \sqrt{p \log p}$, so
\begin{align*}
 \pr\left(\cup_{i \neq j} \{|Z_i^\top X_j| \geq tc' \sqrt{p \log p}\} \right) &\leq \pr\left(\left(\cup_{i \neq j} \{|Z_i^\top X_j| \geq tc' \sqrt{p \log p} \}\right) \cap \Omega \right) + \pr(\Omega^c) \\
 &\leq p^2 \pr(|Z_i^\top X_j| \geq t \|X_j\|_2) + \pr(\Omega^c).
\end{align*}
Thus taking $t = c''\sqrt{\log p}$ for constant $c''$ sufficiently large gives the result.

The second expression in \eqref{eq:subgx4} follows from \eqref{eq:subgx2} and a version of \eqref{eq:subgx2} with $X_{ij}$ replaced with $Z_{ij}$.

\eqref{eq:subgx7}:
The proof of this is similar to that of  \eqref{eq:subgx1}.

\eqref{eq:subgx5}: Note that as $L' \in \R^p$, $\|L'\|_2 \leq \sqrt{p} \|L'\|_\infty$. Next, by Lemma~\ref{lem:subgbern}, we have that there exist constants $\sigma, b>0$ such that
\[
\pr\left(\abs{\frac{1}{n}\sum_{i=1}^n \varepsilon_i X_{ij}} \geq t \right) \leq 2 \exp\left(-\frac{nt^2}{2(\sigma^2+bt)}\right).
\]
Thus by a union bound, for constants $c, c' >0$ we have that with probability at least $1-c'n^{-m}$,
\[
\|L'\|_\infty \leq c \sqrt{\frac{\log p}{n}},
\]
which then gives the result.

	\eqref{eq:subgx6}: Observe that  
	\[
	\left\|\bX^\top \left(\lambda_2 \Lambda + \bX\bX^\top\right)^{-1} \right\|_{\op} \leq \left\|\bZ^\top \left(\lambda_2 \Lambda + \bX\bX^\top\right)^{-1} \right\|_{\op} + \left\|\Gamma^\top \bU^\top \left(\lambda_2 \Lambda + \bX\bX^\top\right)^{-1} \right\|_{\op}.
	\]
	We now work on the event on which \eqref{eq:subgx3} holds. Then considering the first term, we have
	\[
	\left\|\bZ^\top \left(\lambda_2 \Lambda + \bX\bX^\top\right)^{-1} \right\|_{\op} \leq c \sqrt{\frac{p}{n}} \lambda_2^{-1} c_d  \lesssim \frac{1}{c_{\lambda_2} r_{\lo}} \sqrt{\frac{p}{\log p}}.
	\]
	For the second term, we argue as follows:
	\begin{align*}
		\bX\bX^\top &= \bZ\bZ^\top + \bU \Gamma \Gamma^\top \bU^\top + \bZ \Gamma^\top \bU^\top + \bU \Gamma \bZ^\top \\
		&= - \bZ\bZ^\top + \frac{1}{2} \bU \Gamma \Gamma^\top \bU^\top +  \left(2\bZ\bZ^\top + \bZ \Gamma^\top \bU^\top + \bU \Gamma \bZ^\top + \frac{1}{2} \bU \Gamma \Gamma^\top \bU^\top \right) \\
		&= - \bZ\bZ^\top + \frac{1}{2}\bU \Gamma \Gamma^\top \bU^\top +  \left(\sqrt{2}\bZ + \frac{1}{\sqrt{2}} \bU \Gamma\right)\left(\sqrt{2}\bZ + \frac{1}{\sqrt{2}} \bU \Gamma\right)^\top.
	\end{align*}
	Thus by Lemma~\ref{lem:Loewner}
	\[
	\left\|  \Gamma^\top\bU^\top \left(\lambda_2 \Lambda + \bX\bX^\top\right)^{-1} \right\|_{\op}  \leq \left\|\Gamma^\top \bU^\top \left(\lambda_2 \Lambda - \bZ\bZ^\top + \frac{1}{2} \bU \Gamma \Gamma^\top \bU^\top \right)^{-1} \right\|_{\op}.
	\]
	Now using~\eqref{eq:subgx3} and recalling that $\Lambda \in \mathcal{L}$ , we have that $\lambda_2 \Lambda/2 - \bZ\bZ^\top$ is a positive semi-definite matrix for $n$ sufficiently large, which we may assume. Hence applying Lemma~\ref{lem:Loewner} once more, we have 
	\begin{align}
		\left\|\Gamma^\top \bU^\top \left(\lambda_2 \Lambda + \bX\bX^\top\right)^{-1} \right\|_{\op} &\leq 2\left\|\Gamma^\top \bU^\top \left(\lambda_2\Lambda +  \bU\Gamma\Gamma^\top\bU^\top\right)^{-1} \right\|_{\op} \notag \\
		& \leq 2\left\|\Gamma^\top \bU^\top \left(\lambda_2 c_d^{-1} I +  \bU\Gamma\Gamma^\top\bU^\top\right)^{-1} \right\|_{\op}. \label{eq:U_decomp}
	\end{align}
	We now consider the SVD decomposition
	\[
	\bU \Gamma = Q D P^\top,
	\]
	with $Q \in \R^{n \times q}$ having orthonormal columns, $D \in \R^{q \times q}$ a diagonal matrix and $P \in \R^{q \times q}$ an orthogonal matrix. Note that the smallest diagonal entry of $D$ is
	\begin{equation} \label{eq:smallest_diag}
		\lambda_{\min}(\bU)\lambda_{\min}(\Gamma) \gtrsim \sqrt{r_{\lo}}.
	\end{equation}
	Now using the SVD decomposition in \eqref{eq:U_decomp} we obtain
	\begin{align*}
		\left\|\Gamma^\top \bU^\top \left(\lambda_2 \Lambda + \bX\bX^\top\right)^{-1} \right\|_{\op} & \leq 2\left\|P D Q^\top \left(\lambda_2 c_d^{-1} I +  Q D^2 Q^\top\right)^{-1} \right\|_{\op} \\
		& = 2 \left\|D \left(\lambda_2 c_d^{-1} I + D^2 \right)^{-1} \right\|_{\op} \leq 2 \left\|D^{-1} \right\|_{\op} \lesssim 1/\sqrt{r_{\lo}},
	\end{align*} 
	where the first and penultimate inequalities follow from Lemma~\ref{lem:Loewner}, and the final inequality is due to \eqref{eq:smallest_diag}. Putting things together gives the desired result.
\end{proof}

\subsubsection{Bounds relating to \texorpdfstring{$\cT_1$}{T1}} \label{sec:T_1}

\begin{lemma}\label{lem:term1}
	Consider the setup of Lemma~\ref{lem:subgx}. Given $c_t > 0$, there exists $c>0$ such that on the event $\mathcal{A}_1(c_t)$,
	\[
	\forall \Lambda \in \mathcal{L}, i \in [n], j \in [p], \quad
	|U_i^\top \Gamma \bX^\top (\lambda_2 \Lambda + \bX \bX^\top)^{-1} \bX_j|  \leq c {(1 + c_{\lambda_2}^{-1})} \sqrt{\log n}.
	\]
\end{lemma}
\begin{proof}
	Let $c_t > 0$ be given. We will use $a \lesssim b$ to denote that there exists constant $C >0$ potentially depending on $c_t$ and other constants such that $a \leq C b$.
	
	Let us write $P \in \R^{n \times n}$ for the projection onto the column space of $\bU$. Observe that
	\[
	\bX \bX^\top = \bZ \bX^\top+ \bU \Gamma \bX^\top,
	\]
	so
	\[
	\bU \Gamma \bX^\top = P \bU  \Gamma \bX^\top = P(\bX \bX^\top - \bZ \bX^\top).
	\]

	Now let
	\begin{align*}
		\mathrm{I} &:=   e_i^\top P  \bX \bX^\top \left(\lambda_2 \Lambda + \bX \bX^\top\right)^{-1} \bX_j, \\
		\mathrm{II} &:=   e_i^\top P  \bZ \bX^\top \left(\lambda_2 \Lambda + \bX \bX^\top\right)^{-1} \bX_j.
	\end{align*}
	We control each of these in turn.
	For $\mathrm{I}$, observe that
	\begin{align*}
		|\mathrm{I}| & \leq \left\| \left(\lambda_2 \Lambda + \bX \bX^\top \right)^{-1} \bX \bX^\top P e_i \right\|_2 \left\|\bX_j \right\|_2
	\end{align*}
	Then observing that $\lambda_2 \Lambda \succeq \lambda_2 c_d^{-1} I$, it follows from Lemma~\ref{lem:Loewner} that
	\begin{align}
		|\mathrm{I}|  & \leq \left\|  \left(\lambda_2 c_d^{-1} I + \bX \bX^\top \right)^{-1} \bX \bX^\top P e_i \right\|_2 \left\| \bX_j \right\|_2 \nonumber\\
		& \leq \lambda_{\max} \left( \left(\lambda_2 c_d^{-1} I + \bX \bX^\top \right)^{-1} \bX \bX^\top \right) \left\|P e_i\right\|_2 \left\| \bX_j \right\|_2 \label{eq:term1_I_bd}
	\end{align}
	It then follows from the fact that $\lambda_{\max} \{(\lambda_2 c_d^{-1} I + \bX \bX^\top)^{-1}\bX \bX^\top\} \leq 1$, \eqref{eq:U_proj} and~\eqref{eq:subgx1} that $\sqrt{n} |\mathrm{I}| \lesssim \sqrt{\log n}$.
	
	For $\mathrm{II}$, it follows from a similar analysis to that above, that
	\begin{align*}
		|\mathrm{II}| & \leq \left\|  e_i^\top P \bZ \bX^\top \left(\lambda_2 \Lambda + \bX\bX^\top \right)^{-1} \right\|_2 \left\|\bX_j\right\|_2 \\
		&\leq \left\|  P e_i \right\|_2 \|\bZ\|_{\op} \left\|\bX^\top \left(\lambda_2 \Lambda + \bX\bX^\top \right)^{-1} \right\|_{\op} \left\|\bX_j\right\|_2.
	\end{align*}
	Now by applying~\eqref{eq:subgx3},~\eqref{eq:U_proj},~ \eqref{eq:subgx6} and~\eqref{eq:subgx1}, we have that
	\begin{equation}
	\sqrt{n} |\mathrm{II}| \lesssim { (1 + c_{\lambda_2}^{-1})} \max\left\{\frac{p}{r_{\lo} \sqrt{n \log p}}, \sqrt{\frac{p}{r_{\lo} n}}\right\}  \sqrt{\log n}. \label{eq:term1_II_bd}
	\end{equation}
	By combining this with Assumption~\ref{cd:pns}(iii), we have $\sqrt{n} |\mathrm{II}| \lesssim { (1 + c_{\lambda_2}^{-1})} \sqrt{\log n}$. Then the desired result follows as $\sqrt{n} \max\{|\mathrm{I}|, |\mathrm{II}|\} \lesssim {(1 + c_{\lambda_2}^{-1})} \sqrt{\log n}$.
\end{proof}

\begin{lemma} \label{lem:const} Assume Assumptions~\ref{cd:pns}--\ref{cd:subg}. Given $m \in \N$, there exists $c_{t1}>0$ and $c' > 0$ such that 
	\[
	\pr(\cT_1(c_{t1})) \geq 1 - c' n^{-m}.
	\]
\end{lemma}
\begin{proof}
Recall that $\cT_1(c_{t1})$ is defined by
\[
	\forall \Lambda \in \cL, i \in [n], j \in [p], \quad 
	\left|X_i^\top \bX^\top (\lambda_2 \Lambda + \bX \bX^\top)^{-1} \bX_j \right| \leq c_{t1} (1 + c_{\lambda_2}^{-1}) \left(\sqrt{\log n} + \frac{\sqrt{pn \log p}}{r_{\lo}} \right).
\]
In the following we work on $\mathcal{A}_1(c_t)$ for $c_t$ such that $\pr(\mathcal{A}_1(c_t)) \geq 1 - c'n^{-m}$ for some $c'>0$. We will use $a \lesssim b$ to denote that there exists constant $C >0$ potentially depending on $c_t$ and other constants such that $a \leq C b$.

	Fix $i$ and $\Lambda \in \mathcal{L}$. Observe that 
	\begin{align}\label{eq:term1}
		X_i^\top \bX^\top  & (\lambda_2 \Lambda + \bX \bX^\top)^{-1} \bX_j \nonumber\\
		& = U_i^\top \Gamma \bX (\lambda_2 \Lambda + \bX \bX^\top)^{-1} \bX_j + Z_i^\top \bX^\top (\lambda_2 \Lambda + \bX \bX^\top)^{-1} \bX_j.
	\end{align}
	Lemma~\ref{lem:term1} bounds the first term in absolute value by a constant times $(1 + c_{\lambda_2}^{-1}) \sqrt{\log n}$.
	
	We now focus on the second term of~\eqref{eq:term1}. Let $a, b \in \R^n$ be given by
	\begin{align*}
		a & :=e_i  \big(\bX  Z_i\big)_i = e_i X_i^{\top}Z_i/\sqrt{n} \\
		b &:= \bX Z_i - a.
	\end{align*}
	Then
	\begin{align*}
		|Z_i^\top \bX^\top (\lambda_2 \Lambda + \bX \bX^\top)^{-1} \bX_j| &\leq |a^\top (\lambda_2 \Lambda + \bX \bX^\top)^{-1} \bX_j| + |b^\top (\lambda_2 \Lambda + \bX \bX^\top)^{-1} \bX_j| \\
		&:= \mathrm{I} + \mathrm{II};
	\end{align*}
	we control these two terms in turn below.

	For $\mathrm{I}$, recall that only the $i$th entry of $a$ is non-zero, so $a = \lambda_2 \Lambda a / (\lambda_2 \Lambda_{ii})$. This gives us that
	\begin{align*}
		\mathrm{I} &= \frac{1}{\lambda_2 \Lambda_{ii}} \left| a^\top  \lambda_2 \Lambda (\lambda_2 \Lambda + \bX \bX^\top)^{-1} \bX_j \right| \\
		& \leq \frac{1}{\lambda_2 \Lambda_{ii}} \left| a^\top \bX_j \right| + \frac{1}{\lambda_2 \Lambda_{ii}}  \left| a^\top \bX \bX^\top \left(\lambda_2 \Lambda + \bX \bX^\top\right)^{-1} \bX_j \right|.
	\end{align*}
	Now applying~\eqref{eq:subgx4}, we have that $|a_i| = |X_i^{\top}Z_i|/\sqrt{n}  \lesssim  p \log p / \sqrt{n}$. By combining this with~\eqref{eq:subgx2}, we have that 
	\[
	\frac{1}{\lambda_2 \Lambda_{ii}} \left| a^\top \bX_j \right|  = \frac{1}{n \lambda_2 \Lambda_{ii}} X_i^{\top}Z_i X_{ij} \lesssim \frac{p (\log p)^{3/2}}{c_d^{-1} n\lambda_2} \lesssim \frac{p \log p}{c_{\lambda_2} r_\lo \sqrt{n}}.
	\]
	
	For the second summand in the decomposition of $\mathrm{I}$, we argue as follows, appealing to Lemmas~\ref{lem:subgx} and~\ref{lem:term1}:
	\begin{align*}
		\phantom{.} & \phantom{\leq}  \frac{1}{\lambda_2 \Lambda_{ii}} \left|a^\top \bX \bX^\top \left(\lambda_2 \Lambda + \bX \bX^\top\right)^{-1} \bX_j \right| \nonumber \\
		\phantom{.} & \lesssim  \frac{p \sqrt{\log p}}{r_{\lo}} \abs{ e_i^\top \bX \bX^\top \left(\lambda_2 \Lambda + \bX \bX^\top\right)^{-1} \bX_j} \nonumber\\
		\phantom{.} & \leq \frac{p \sqrt{\log p}}{r_{\lo}} \left( \abs{ e_i^\top \bZ \bX^\top \left(\lambda_2 \Lambda + \bX \bX^\top\right)^{-1} \bX_j} + \abs{ e_i^\top \bU \Gamma \bX^\top \left(\lambda_2 \Lambda + \bX \bX^\top\right)^{-1} \bX_j}\right) \nonumber\\
		\phantom{.} & \leq \frac{p \sqrt{\log p}}{r_{\lo}} \left( \|\bZ\|_{\op} \norm{\left(\lambda_2 \Lambda + \bX \bX^\top\right)^{-1} \mb X}_{\op} \|\bX_j\|_2 + \abs{ e_i^\top \bU \Gamma \bX^\top \left(\lambda_2 \Lambda + \bX \bX^\top\right)^{-1} \bX_j}\right) \nonumber\\
		\phantom{.} & \lesssim {(1 + c_{\lambda_2}^{-1})} \left(\frac{p^2}{r_{\lo}^2 \sqrt{n}} + \frac{p^{3/2}  \sqrt{\log p}}{r_{\lo}^{3/2} \sqrt{n}} + \frac{p  \sqrt{\log p \log n}}{r_{\lo}\sqrt{n}}\right).
	\end{align*}
	From Assumptions~\ref{cd:pns}(iii) and~(iv), we see that the above is bounded by a constant times $1 + c_{\lambda_2}^{-1}$.
	
	For $\mathrm{II}$, first note that due to~\eqref{eq:subgx4} we have $\|b\|_\infty \lesssim \sqrt{p/n}\log p$ whence $\|b\|_2 \lesssim \sqrt{p}\log p$.
	By combining this with~\eqref{eq:subgx1}, we have
	\begin{align*}
		\mathrm{II} \leq \|b\|_2 \lambda_{\min}^{-1} \left(\lambda_2 \Lambda + \bX \bX^\top\right) \|\bX_j \|_2 \lesssim \frac{\sqrt{pn \log p}}{c_{\lambda_2} c_d^{-1} r_\lo}.
	\end{align*}
	Putting things together then gives the desired bound.
\end{proof}


\subsubsection{Bounds relating to \texorpdfstring{$\cT_2$}{T2}}
\begin{lemma} \label{lem:T_2_first}
	Consider the setup of Lemma~\ref{lem:subgx}. Given $c_t > 0$ there exists constant $c>0$ (independent of $c_{\lambda_2}$), and $n_1 \in \mathbb{N}$ depending on $c_{\lambda_2}$, such that for all $n \geq n_1$ and $\Lambda \in \mathcal{L}$, on the event $\mathcal{A}_1(c_t)$,
	\[
	\left\| (2 \lambda_2 \Lambda + \bX \bX^\top )^{-1} \bU \delta^0 \right\|_2 \leq \frac{c}{r_\ell}.
	\] 
\end{lemma}
\begin{proof}
	Given $c_t>0$, let $n_1$ be such that for all $n \geq n_1$, 
	\begin{equation} \label{eq:c_t_choice0}
		\lambda_2 c_d^{-1} \geq c_t \sqrt{\frac{p}{n}}.
	\end{equation}
	We will use $a \lesssim b$ to denote that there exists constant $C >0$ potentially depending on $c_t$ and other constants such that $a \leq C b$. In the following, we work on the event $\mathcal{A}_1(c_t)$.
	
	Let
	\[
	\mathrm{I} := \left\| (2 \lambda_2 \Lambda + \bX \bX^\top )^{-1} \bU \delta^0 \right\|_2.
	\]
	We have
	\begin{align}\label{eq:8b1}
		\mathrm{I} 
		&= \left\| \left(2 \lambda_2 \Lambda + \bZ\bZ^\top + \bZ \Gamma \bU^\top + \bU \Gamma^\top \bZ^\top + \bU \Gamma^\top \Gamma \bU^\top\right)^{-1} \bU \delta^0 \right\|_2 \nonumber\\
		& \leq \left\| \left( 2 \lambda_2 c_d^{-1} I + \bZ\bZ^\top + \bZ \Gamma \bU^\top + \bU \Gamma^\top \bZ^\top + \bU \Gamma^\top \Gamma \bU^\top\right)^{-1} \bU \delta^0 \right\|_2,
	\end{align}
	appealing to Lemma~\ref{lem:Loewner}. 
	
	Now observe that
	\begin{align*}
		& 2 \lambda_2 c_d^{-1} I + \bZ\bZ^\top + \bZ \Gamma \bU^\top + \bU \Gamma^\top \bZ^\top + \bU \Gamma^\top \Gamma \bU^\top \\
		& = \lambda_2 c_d^{-1} I + \left(\lambda_2 c_d^{-1} I - \bZ\bZ^\top\right) + \left(2\bZ\bZ^\top + \bZ \Gamma \bU^\top + \bU \Gamma^\top \bZ^\top + \frac{1}{2}\bU \Gamma^\top \Gamma \bU^\top\right) + \frac{1}{2}\bU \Gamma^\top \Gamma \bU^\top \\
		& = \lambda_2 c_d^{-1} I + \left(\lambda_2 c_d^{-1} I - \bZ\bZ^\top\right) + \left(\sqrt{2} \bZ + \frac{1}{\sqrt{2}} \bU \Gamma^\top\right) \left(\sqrt{2} \bZ + \frac{1}{\sqrt{2}} \bU \Gamma^\top\right)^\top + \frac{1}{2}\bU \Gamma^\top \Gamma \bU^\top \\
		&\preceq  \lambda_2 c_d^{-1} I  + \frac{1}{2}\bU \Gamma^\top \Gamma \bU^\top,
	\end{align*}
	using \eqref{eq:subgx3} and \eqref{eq:c_t_choice0} for the final inequality. Thus from 
	Lemma~\ref{lem:Loewner}, we have that
	\[
	\mathrm{I}  \leq \left\| \left( \lambda_2 c_d^{-1} I + \frac{1}{2}\bU \Gamma^\top \Gamma \bU^\top\right)^{-1} \bU \delta^0 \right\|_2.
	\]
	Now observe that for any $\gamma > 0$, we have
	\begin{align*}
		(\gamma I + \bU \Gamma^\top  \Gamma \bU^\top )\bU &= \bU(\gamma I + \Gamma^\top  \Gamma \bU^\top  \bU) \\
		\bU (\gamma I + \Gamma^\top  \Gamma \bU^\top  \bU)^{-1} &= (\gamma I + \bU \Gamma^\top  \Gamma \bU^\top )^{-1} \bU.
	\end{align*}
Thus
\begin{align*}
	\mathrm{I}  \lesssim \|\mb U\|_{\op} \lambda_{\min}^{-1} (\Gamma^\top  \Gamma) \lambda_{\min}^{-1}( \bU^\top  \bU) \|\delta^0\|_2.
\end{align*}
From \eqref{eq:subgx3} and Assumption~\ref{cd:subg}, we see that $\mathrm{I}  \lesssim 1/r_\ell$ as required.
%
\end{proof}

\begin{lemma}\label{lem:eigmin}
	Consider the setup of Lemma~\ref{lem:subgx}.
	Let $\tau > 0$ be given and define
	\begin{equation} \label{eq:CB_def}
		\cC := \left\{i : |X_i^\top \beta^0 + U_i^\top \delta^0| \leq \tau \right\} \qquad \text{and}\qquad \cB := \left\{i \in \cC :  |X_i^\top \beta^0 + X_i^\top b^0| \leq 2\tau \right\}.
	\end{equation}
Given $c_t >0$ and $c_{\lambda_2}>0$, there exists $c>0$ depending on these and $\tau$, and $n_1$ depending on $c_{\lambda_2}$, such that for all $n \geq n_1$, on the event $\mathcal{A}_1(c_t)$, we have  $|\cC \setminus \cB| \leq c p  \sqrt{n\log p} / r_\lo$.
%
\end{lemma}
\begin{proof}
	Let $n_1$ be as in Lemma~\ref{lem:T_2_first}, and in the following fix $n \geq n_1$.
	We will use $a \lesssim b$ to denote that there exists constant $C >0$ potentially depending on $c_t, c_{\lambda_2}$ and other constants such that $a \leq C b$. In the following, we work on the event $\mathcal{A}_1(c_t)$.
			
Recall that $\bar{\Lambda}$ \eqref{eq:Lambda_bar_def} satisfies
\[
f(X_i^\top (\beta^0 + b^0)) - f(X_i^\top \beta^0 + U_i^\top \delta^0) = \bar{\Lambda}_{ii} (X_i^\top b^0 - U_i^\top \delta^0).
\]
Using the decomposition in~\eqref{eq:b^0_decomp} we obtain
\begin{align*}
	& \left(\frac{1}{n} \sum_{i=1}^n \left(f\left(X_i^\top \left(\beta^0 + b^0\right)\right) - f\left(X_i^\top \beta^0 + U_i^\top \delta^0\right)\right)^2\right)^{1/2} = \left\|\bar{\Lambda} \left(\bX b^0 - \bU \delta^0\right)\right\|_2 \\
	& \hspace*{1cm} = \left\| - \bar{\Lambda} \bX \left(2 \lambda_2 I + \bX^\top \bar{\Lambda} \bX\right)^{-1} L' + \bar{\Lambda} \bX \left(2 \lambda_2 I + \bX^\top \bar{\Lambda} \bX\right)^{-1} \bX^\top \bar{\Lambda} \bU \delta^0 - \bar{\Lambda} \bU \delta^0\right\|_2 \\
	& \hspace*{1cm} \leq \sqrt{c_d}\left\| \bar{\Lambda}^{1/2} \bX \left(2 \lambda_2 I + \bX^\top \bar{\Lambda} \bX\right)^{-1} L' \right\|_2 + \left\| \bar{\Lambda} \bX \left(2 \lambda_2 I + \bX^\top \bar{\Lambda} \bX\right)^{-1} \bX^\top \bar{\Lambda} \bU \delta^0 - \bar{\Lambda} \bU \delta^0 \right\|_2 \\
	&=: \sqrt{c_d} \mathrm{I} + \mathrm{II}.
\end{align*}

For $\mathrm{I}$, observe that
\begin{align}\label{eq:consteq1}
	\left\| \bar{\Lambda}^{1/2} \bX \left(2 \lambda_2 I + \bX^\top \bar{\Lambda} \bX\right)^{-1} L' \right\|_2 \leq \lambda_{\max}\left( \bar{\Lambda}^{1/2} \bX \left(2 \lambda_2 I + \bX^\top \bar{\Lambda} \bX\right)^{-1}\right) \|L' \|_2.
\end{align}
Then considering the singular value decomposition $\bar{\Lambda}^{1/2} \bX = P D Q^\top$, we have that
\[
\bar{\Lambda}^{1/2} \bX \left( 2 \lambda_2 I + \bX^\top \bar{\Lambda} \bX\right)^{-1} = P \tilde{D} Q^\top,
\]
with $\tilde{D} \in \R^{n \times n}$ a diagonal matrix with $\tilde{D}_{ii} = \frac{D_{ii}}{2 \lambda_2 + D_{ii}^2}$. Returning to~\eqref{eq:consteq1}, we have
\[
\mathrm{I} \leq \max_t \frac{t}{2\lambda_2 + t^2} \cdot \|L'\|_2 \leq \frac{1}{\sqrt{
		\lambda_2}} \|L'\|_2.
\]
By combining this with~\eqref{eq:subgx5},
\begin{equation} \label{eq:a_bd}
	\mathrm{I} \lesssim \frac{p^{1/2}  (\log p)^{1/4}}{r_{\lo}^{1/2} n^{1/4}}.
\end{equation}

To control $\mathrm{II}$, we argue as follows. First observe that for any matrix $A$, we have
\begin{align} \label{eq:kernel_ridge0}
A^\top (AA^\top + \lambda I)  & = (A^\top A + \lambda I)  A^\top.	\nonumber \\
(A^\top A + \lambda I)^{-1} A^\top & = A^\top(AA^\top + \lambda I)^{-1}.
\end{align}
Applying this with $A = \bar{\mb \Lambda}^{1/2}\mb X$, we have
\begin{align}
	(2 \lambda_2 I + \bX^\top \bar{\Lambda} \bX)^{-1} \bX^\top \bar{\Lambda} \bU \delta^0 &= \bX^\top \bar{\Lambda}^{\frac{1}{2}} (2 \lambda_2 I + \bar{\Lambda}^{\frac{1}{2}} \bX \bX^\top \bar{\Lambda}^{\frac{1}{2}} )^{-1} \bar{\Lambda}^{\frac{1}{2}} \bU \delta^0 \notag \\
	&= \bX^\top (2 \lambda_2 \bar{\Lambda}^{-1} +  \bX \bX^\top )^{-1} \bU \delta^0, \label{eq:kernel_ridge}
\end{align}
and so
\begin{align*}
	& \bar{\Lambda} \bX \left(2 \lambda_2 I + \bX^\top \bar{\Lambda} \bX\right)^{-1} \bX^\top \bar{\Lambda} \bU \delta^0 - \bar{\Lambda} \bU \delta^0 = \bar{\Lambda} \bX \bX^\top (2 \lambda_2 \bar{\Lambda}^{-1} + \bX \bX^\top )^{-1} \bU \delta^0 - \bar{\Lambda} \bU \delta^0 \\
	&= \bar{\Lambda} \left(\bX \bX^\top (2 \lambda_2 \bar{\Lambda}^{-1} + \bX \bX^\top )^{-1} - I\right) \bU \delta^0 = -2\lambda_2 (2 \lambda_2 \bar{\Lambda}^{-1} + \bX \bX^\top )^{-1} \bU \delta^0.
\end{align*}
Thus
\[
\frac{1}{2\lambda_2 } \mathrm{II}  = \left\| (2 \lambda_2 \bar{\Lambda}^{-1} + \bX \bX^\top )^{-1} \bU \delta^0 \right\|_2 \lesssim \frac{1}{r_{\ell}},
\]
using Lemma~\ref{lem:T_2_first}. Thus 
\begin{equation} \label{eq:II_bd}
\mathrm{II} \lesssim c_{\lambda_2} \sqrt{\log p / n}.
\end{equation}

From the above and \eqref{eq:a_bd}, we then have
\[
\left( \frac{1}{n} \sum_{i=1}^n \{f(X_i^\top (\beta^0 + b^0)) - f(X_i^\top \beta^0 + U_i^\top b^*)\}^2 \right)^{1/2} \lesssim \frac{p^{1/2} (\log p)^{1/4}}{r_{\lo}^{1/2} n^{1/4}} + \left( \frac{\log p}{n}\right)^{1/2}.
\]
This implies
\begin{align}\label{eq:maxbnd}
\left( \frac{1}{n} \sum_{i \in \cC \setminus \cB} \{f(X_i^\top (\beta^0 + b^0)) - f(X_i^\top \beta^0 + U_i^\top b^*)\}^2\right)^{1/2} \lesssim \frac{p^{1/2} (\log p)^{1/4}}{r_{\lo}^{1/2} n^{1/4}};
\end{align}
note that for all $n$ sufficiently large, we are permitted to omit the term involving $\sqrt{\log p/n}$ due to Assumption~\ref{cd:pns}(ii).

Now let
\begin{equation} \label{eq:c_tau_def}
c_\tau := \min_{|t| \leq 2 \tau} f'(t) > 0.
\end{equation}
Using the non-decreasing property of $f$, we have that $\forall i \in \cC \setminus \cB$, 
\[
|f(X_i^\top (\beta^0 + b^0)) - f(X_i^\top \beta^0 + U_i^\top \delta^0)| \geq \min\{f(2\tau) - f(\tau ), f(-\tau ) - f(-2\tau)\} \geq c_\tau \tau.
\]
Considering this with \eqref{eq:maxbnd}, we see that,
\[
|\cC \setminus \cB| \lesssim  \frac{p \sqrt{n \log p}}{r_\lo}. \qedhere
\]
\end{proof}

\begin{lemma} \label{lem:eigmin2}
	Assume Assumptions~\ref{cd:pns}--\ref{cd:subg}. There exists constant $c > 0$ such that given $m \in \mathbb{N}$ and $c_{\lambda_2}$, there exists $c'$ (depending on these) where on an event $\mathcal{A}_2$ with probability at least $1 - c'n^{-m}$,
	\begin{align}
\label{eq:eigmin}
			\inf_{\Lambda \in \tilde{\mathcal{L}}} \lambda_{\min} (\bU^\top \Lambda \bU) &\geq c, \\
\label{eq:eigmin2}
			\inf_{\Lambda \in \tilde{\mathcal{L}}} \lambda_{\min} (\Lambda \bU) &\geq c.
	\end{align}
\end{lemma}
\begin{proof}
	For $\tau > 0$ (to be chosen), let $A_i$ be the event
	\[
	A_i := \{|X_i^\top \beta^0 + U_i^\top \delta^0| \leq \tau\},
	\] and set $V_i := U_i \one_{A_i}$ and $\mu_V := \E V_i$.  	We will use $a \lesssim b$ to denote that there exists constant $C >0$ potentially depending on $\tau$ and other constants such that $a \leq C b$.
	From Assumption~\ref{cd:subg} and Lemma~\ref{lem:subg_trunc}, we have that for all $\tau > 0$, $V_i-\mu_V$ is a sub-Gaussian random vector with bounded variance proxy $\sigma_V^2$ depending on $\sigma_u^2$. Moreover, as $\Var(U_i) = I$ (from Assumption~\ref{cd:subg}), we have that all $\tau$ sufficiently large, $\lambda_{\min}(\Var(V_i)) > c> 1/2$ for a constant $c>0$; in all that follows, we fix $\tau$ at such a value.
	
	 Then from \citet[Rem.~5.40]{vershynin_2012} we have that there exists a constant $c'$ such that with probability at least $1 - c' n^{-m}$,
	\[
	\lambda_{\min}\left(\frac{1}{n} \sum_{i = 1}^n (V_i-\mu_V) (V_i-\mu_V)^\top \right) \geq 0.5 \quad\text{and}\quad \left\|\frac{1}{n} \sum_{i = 1}^n (V_i-\mu_V)\right\|_\op \leq \sqrt{\frac{\log n}{n}}.
	\]
	Let $\bar{V} = \sum_{i=1}^nV_i/n$ and let $\cC$ and $\cB$ be as in \eqref{eq:CB_def} with the choice of $\tau$ above; note that this choice does not depend on $c_{\lambda_2}$.
	Now
	\begin{align*}
		\frac{1}{n} \sum_{i \in \cC} U_i U_i^\top & = \frac{1}{n} \sum_{i =1}^n V_i V_i^\top \\
		& = \frac{1}{n} \sum_{i = 1}^n (V_i -\mu_V)(V_i-\mu_V)^\top +(\bar{V}-\mu_V)\mu_V^{\top} + \mu_V (\bar{V}-\mu_V)^{\top} + \mu_V\mu_V^\top.
	\end{align*}
	We have that
	\[
	\|(\bar{V}-\mu_V)\mu_V^{\top}\|_{\op} = \|\mu_V (\bar{V}-\mu_V)^{\top}\|_{\op} = \|\bar{V}-\mu_V\|_2 \|\mu_V\|_2.
	\]
	As $\bar{V}-\mu_V$ is a sub-Gaussian random vector with variance proxy $\sigma_V^2/n$, we have that
	\begin{align*}
		\pr((\bar{V}_j-\mu_{V,j})^2 \geq t^2/q) \leq 2\exp(-2nt^2/q\sigma_V^2),
	\end{align*}
	so by a union bound,
	\[
	\pr(\|\bar{V}-\mu_V\|_2^2 \geq t^2) \leq 2q\exp(-2nt^2/q\sigma_V^2),
	\]
	where recall $q \lesssim 1$ (Assumption~\ref{cd:pns}(i)).
	Thus given any $m$, taking $t = c\sqrt{\log n / n}$ for $c$ sufficiently large, we see that there exists constant $c''$ such that with probability at least $1-c''n^{-m}$, we have
	\[
	\|(\bar{V}-\mu_V)\mu_V^{\top}\|_{\op} \lesssim \sqrt{\log n / n},
	\]
	and so on an event with high probability, which we will work on in the following, we have
	\begin{align*}
		\lambda_{\min}\left( \frac{1}{n} \sum_{i \in \cC} U_i U_i^\top\right) &= \lambda_{\min}\left(\frac{1}{n} \sum_{i = 1}^n V_i V_i^\top\right) \\
		&\geq \lambda_{\min}\left(\frac{1}{n} \sum_{i = 1}^n (V_i-\mu_V) (V_i-\mu_V)^\top \right) - 2\|(\bar{V}-\mu_V)\mu_V^{\top}\|_{\op}  \\
		&\gtrsim 1.
	\end{align*}

Next, observe that for $i \in \cB$, we have that
\[
|X_i^\top\beta^0 + \eta U_i^\top \delta^0 + (1-\eta)X_i^\top b^0| \leq 2\tau
\]
for all $\eta \in [0,1]$. Thus for such $i$, recalling \eqref{eq:Lambda_bar_def} and \eqref{eq:Lambda_0_def}, we have that $\Lambda^0_{ii}, \,\bar{\Lambda}_{ii} \geq c_\tau >0$ with $c_\tau$ given by \eqref{eq:c_tau_def}. Thus for all $\Lambda \in \tilde{\mathcal{L}}$ we have $\Lambda_{ii}\geq c_\tau$ whenever $i \in \cB$. Let $c_t$ sufficiently large such that there exists $c > 0$ with $\pr(\mathcal{A}_1(c_t)) \geq 1 - cn^{-m}$ be given.
Now, working on the intersection of $\mathcal{A}_1(c_t)$ and the event above,
we have that for all $\Lambda \in \tilde{\mathcal{L}}$, 
\begin{align*}
	c_\tau^{-1} \lambda_{\min}(\bU^\top \Lambda \bU) & \geq  \lambda_{\min}\left( \frac{1}{n} \sum_{i \in \cB} U_i U_i^\top\right) \\
	&\geq \lambda_{\min}\left( \frac{1}{n} \sum_{i \in \cC} U_i U_i^\top\right) -  \lambda_{\max}\left( \frac{1}{n} \sum_{i \in \cC \setminus \cB} U_i U_i^\top\right)\\
	& \gtrsim 1 - |\cC \setminus \cB| \max_i \|U_i\|_2^2 /n.
\end{align*}

Note that on an event with probability at least $1-cn^{-m}$, $\max_i\|U_i\|_2^2 /n \lesssim \log n / n$, as each $U_{ij}^2$  is sub-exponential and $q \lesssim 1$ (Assumption~\ref{cd:pns}(i)). Working now also on this event and appealing to Lemma~\ref{lem:eigmin}, we see that there exists some $c$ depending on $c_{\lambda_2}$ with
\[
\lambda_{\min}(\bU^\top \Lambda \bU) \gtrsim 1 - c \frac{p \sqrt{\log p} \log n}{r_\lo \sqrt{n}}.
\]
Considering Assumption~\ref{cd:pns}(iii), we see that for all sufficiently large $n$ (depending on $c_{\lambda_2}$), $\lambda_{\min}(\bU^\top \Lambda \bU)$ is greater than a positive constant, on the intersection of the events above. Thus there exists a $c'$ depending on $c_{\lambda_2}$ such that for all smaller $n$, no claim is made in the statement of the result, but for all larger $n$, the probability is at least $1-c'n^{-m}$.

Turning to the second inequality in the statement, note that $\bU^\top \Lambda \bU = (\Lambda^{1/2} \bU)^{\top} (\Lambda^{1/2} \bU)$, and the $q$ singular values of $\Lambda^{1/2} \bU \in \R^{n \times q}$ are precisely the square-roots of the eigenvalues of $\bU^\top \Lambda \bU$.
\end{proof}

\begin{lemma}  \label{lem:Gamma_Lambda_bd} 
	Consider the setup of Lemma~\ref{lem:subgx}. On the event $\mathcal{A}_1(c_t) \cap \mathcal{A}_2$ (defined in Lemmas~\ref{lem:subgx} and \ref{lem:eigmin2}), we have that there exists $c >0$ depending on $c_t$ and quantities designated as constants, such that for all $\Lambda \in \tilde{\mathcal{L}}$,
	\[
	\norm{\Gamma \left( 2 \lambda_2 I + \bX^\top \Lambda \bX \right)^{-1} L' }_2 \leq c (1 + c_{\lambda_2}^{-1})  \left(\sqrt{\frac{p \log p}{r_{\lo} n}} + \frac{p}{r_\lo \sqrt{n}}\right).
	\]
\end{lemma} 
\begin{proof}
	We work on the event in the statement of the result. We will use $a \lesssim b$ to denote that there exists constant $C >0$ potentially depending on $c_t$ and other constants such that $a \leq C b$. In the following, fix $\Lambda \in \tilde{\mathcal{L}}$.
	
	We have
	\[
	\norm{\Gamma \left( 2 \lambda_2 I + \bX^\top \Lambda \bX \right)^{-1} L' }_2 \leq \lambda_{\max} \left(\Gamma \left(2 \lambda_2 I + \bX^\top \Lambda \bX\right)^{-1}\right) \|L'\|_2.
	\]
	Now
	\begin{align}\label{eq:dcont}
		\lambda_{\max} \bigg( \Gamma & \left( 2 \lambda_2 I + \bX^\top \Lambda \bX \right)^{-1} \bigg) = \lambda_{\max} \left( (\bU^\top \Lambda \bU)^{-1} \bU^\top \Lambda \bU\Gamma \left( 2 \lambda_2 I + \bX^\top \Lambda \bX \right)^{-1} \right) \nonumber\\
		& \leq \lambda_{\max} \left( (\bU^\top \Lambda \bU)^{-1} \right) \lambda_{\max} (\bU) \lambda_{\max} \left( \Lambda \bU\Gamma \left( 2 \lambda_2 I + \bX^\top \Lambda \bX \right)^{-1} \right) \\
		& = \lambda_{\min}^{-1} \left(\bU^\top \Lambda \bU \right) \lambda_{\max} (\bU) \lambda_{\max} \left( \Lambda \bU\Gamma \left( 2 \lambda_2 I + \bX^\top \Lambda \bX \right)^{-1} \right) \nonumber \\
		& \lesssim \lambda_{\max} \left( \Lambda \bU\Gamma \left( 2 \lambda_2 I + \bX^\top \Lambda \bX \right)^{-1} \right), \nonumber
	\end{align}
	where for the last inequality we have used Lemma~\ref{lem:eigmin2} and~\eqref{eq:subgx3}.
	Next,
	\begin{align*}
		\lambda_{\max} \left( \Lambda \bU\Gamma \left( 2 \lambda_2 I + \bX^\top \Lambda \bX \right)^{-1} \right) &= \lambda_{\max} \left( \Lambda (\bX - \bZ) \left( 2 \lambda_2 I + \bX^\top \Lambda \bX \right)^{-1} \right) \\
		&\leq \lambda_{\max} \left(\Lambda \bX \left(2 \lambda_2 I + \bX^\top \Lambda \bX \right)^{-1} \right) + \lambda_{\max} \left(\Lambda \bZ \left(2 \lambda_2 I + \bX^\top \Lambda \bX \right)^{-1} \right).
	\end{align*}
	Now by \eqref{eq:kernel_ridge}, we have that
	\begin{align} \label{eq:kernel_ridge2}
		\Lambda \bX \left(2 \lambda_2 I + \bX^\top \Lambda \bX \right)^{-1} = \Lambda (2\lambda_2 I + \Lambda \mb X \mb X^{\top} )^{-1} \mb X.
	\end{align}
	Further noting that
	\[
	\lambda_{\max} \left(\Lambda \bZ \left(2 \lambda_2 I + \bX^\top \Lambda \bX \right)^{-1} \right) \lesssim \lambda_{\max}(\mb Z) /\lambda_2, 
	\]
	we see from~\eqref{eq:subgx3} and~\eqref{eq:subgx6}, that
	\begin{align*}
		\lambda_{\max} \left(\Gamma \left(2 \lambda_2 I + \bX^\top \Lambda \bX \right)^{-1} \right) \lesssim   \frac{1 + c_{\lambda_2}^{-1}}{\sqrt{r_\lo}} +  (1 + c_{\lambda_2}^{-1}) \frac{\sqrt{p}}{r_{\lo} \sqrt{\log p}}.
	\end{align*}
	Thus from~\eqref{eq:subgx5} we have that
	\[
		\norm{\Gamma \left( 2 \lambda_2 I + \bX^\top \Lambda \bX \right)^{-1} L' }_2 \lesssim (1 + c_{\lambda_2}^{-1}) \left(\sqrt{\frac{p \log p}{r_{\lo} n}} + \frac{p}{r_\lo \sqrt{n}}\right). \qedhere
	\]
\end{proof}

\begin{lemma} \label{lem:Gamma_b^0}
	Consider the setup of Lemma~\ref{lem:subgx}. On the event $\mathcal{A}_1(c_t) \cap \mathcal{A}_2$ (defined in Lemmas~\ref{lem:subgx} and \ref{lem:eigmin2}), we have that there exists $c >0$ depending on $c_t$ and quantities designated as constants, such that
	\[
	\|\Gamma b^0 - \delta^0\|_2 \leq c \left( (1 + c_{\lambda_2}^{-1}) \left(\sqrt{\frac{p \log p}{r_{\lo} n}} + \frac{p}{r_\lo \sqrt{n}}\right) + c_{\lambda_2} \sqrt{\frac{\log p}{n}}\right)
	\]
\end{lemma}
\begin{proof}
	We work on the event in the statement of the result. We will use $a \lesssim b$ to denote that there exists constant $C >0$ potentially depending on $c_t$ and other constants such that $a \leq C b$.
	
	Using~\eqref{eq:b^0_decomp},
	\begin{align}\label{eq:gamma}
		\|\Gamma b^0 - \delta^0\|_2 &\leq \norm{\Gamma \left( 2 \lambda_2 I + \bX^\top \bar{\Lambda} \bX \right)^{-1} L' }_2 + \norm{\Gamma \left( 2 \lambda_2 I + \bX^\top \bar{\Lambda} \bX\right)^{-1} \bX^\top \bar{\Lambda} \bU \delta^0 - \delta^0 }_2\\
		&=: \mathrm{I} + \mathrm{II}.
	\end{align}
	For $\mathrm{I}$, we have from Lemma~\ref{lem:Gamma_Lambda_bd} that
	\begin{align*}
		\mathrm{I} \lesssim  (1 + c_{\lambda_2}^{-1}) \left(\sqrt{\frac{p \log p}{r_{\lo} n}} + \frac{p}{r_\lo \sqrt{n}}\right).
	\end{align*}
	
	For $\mathrm{II}$, by applying the same argument as in~\eqref{eq:dcont}, we see that
	\begin{align*}
		\mathrm{II} & \leq \lambda_{\min}^{-1} (\bU^\top \bar{\Lambda} \bU) \cdot \lambda_{\max}(\bU) \cdot \norm{\bar{\Lambda} \bU \Gamma \left( 2 \lambda_2 I + \bX^\top \bar{\Lambda} \bX \right)^{-1} \bX^\top \bar{\Lambda} \bU \delta^0 - \bar{\Lambda} \bU \delta^0 }_2 \\
		& \lesssim \norm{\bar{\Lambda} (\bX - \bZ) \left( 2 \lambda_2 I + \bX^\top \bar{\Lambda} \bX \right)^{-1} \bX^\top \bar{\Lambda} \bU \delta^0 - \bar{\Lambda} \bU \delta^0 }_2 \\
		& \leq  \norm{\bar{\Lambda} \bX \left( 2 \lambda_2 I + \bX^\top \bar{\Lambda} \bX\right)^{-1} \bX^\top \bar{\Lambda} \bU \delta^0 - \bar{\Lambda} \bU \delta^0 }_2 +  \norm{\bar{\Lambda} \bZ \left( 2 \lambda_2 I + \bX^\top \bar{\Lambda} \bX\right)^{-1} \bX^\top \bar{\Lambda} \bU \delta^0 }_2 \\
		& =: \mathrm{II}_1 + \mathrm{II}_2.
	\end{align*}
	Note that $\mathrm{II}_1$ is the same as $\mathrm{II}$ in Lemma~\ref{lem:eigmin}, which gives the bound
	$\mathrm{II}_1 \lesssim c_{\lambda_2} \sqrt{\frac{\log p}{n}}$ (see e.g. \eqref{eq:II_bd}). For $\mathrm{II}_2$, we note from \eqref{eq:kernel_ridge2} that
	\[
	\mathrm{II}_2 \lesssim \lambda_{\max}(\bar{\Lambda} \bZ) \lambda_{\max}\left( \bX^\top \left( 2 \lambda_2 \bar{\Lambda}^{-1} + \bX\bX^\top\right)^{-1} \right) \lambda_{\max}(\bU).
	\]
	Then by applying~\eqref{eq:subgx3} and~\eqref{eq:subgx6}, we see that $\mathrm{II}_2 \lesssim (1 + c_{\lambda_2}^{-1}) \frac{p}{r_\lo \sqrt{n}}$, and hence
	\[
	\mathrm{II} \lesssim  c_{\lambda_2} \sqrt{\frac{\log p}{n}} +  (1 + c_{\lambda_2}^{-1}) \frac{p}{r_\lo \sqrt{n}}.
	\]
	Putting things together, we obtain the desired result.
\end{proof}

\begin{lemma}\label{lem:rsc1}
	Assume Assumptions~\ref{cd:pns}--\ref{cd:subg}. Given $m \in \N$, there exists $c_{t2}>0$ and $c' > 0$ (with the latter depending on $c_{\lambda_2}$) such that 
	\[
	\pr(\cT_2(c_{t2})) \geq 1 - c' n^{-m}.
	\]
\end{lemma}
\begin{proof}
	Recall that $\mathcal{T}_2(c_{t2})$ is defined by the inequalities
	\begin{align*}
		\left\| (2 \lambda_2 \bar{\Lambda}^{-1} + \bX \bX^\top )^{-1} \bU \delta^0 \right\|_2 &\leq c_{t2} / r_\ell \\
		\|\bX b^0 - \bU \delta^0 \|_2 &\leq c_{t2} (1+ c_{\lambda_2}^{-1}) \left(\sqrt{\frac{p \log p}{r_{\lo} n}} + \frac{p}{r_\lo \sqrt{n}}\right) + c_{t2}c_{\lambda_2}\sqrt{\frac{\log p}{n}}.
	\end{align*}
	The first of these is addressed by Lemma~\ref{lem:T_2_first}: for all $c_t>0$, there exists $c_{t2}$ such that the first inequality holds on $\mathcal{A}_1(c_t)$ for all $n$ sufficiently large; by making $c'$ sufficiently large, we may ensure that only such $n$ are considered.
	
	In the following we work on $\mathcal{A}_1(c_t) \cap \mathcal{A}_2$ where $\mathcal{A}_2$ is specified in Lemma~\ref{lem:eigmin2}. We will use $a \lesssim b$ to denote that there exists constant $C >0$ potentially depending on $c_t$ and other constants such that $a \leq C b$.
	
	First note that from \eqref{eq:b^0_decomp}, we have
	\begin{align*}
	\|b^0\|_2 & \leq  \left\|\left(2 \lambda_2 I + \bX^\top \bar{\Lambda} \bX\right)^{-1} L' \right\|_2 +\left\| \left(2 \lambda_2 I + \bX^\top \bar{\Lambda} \bX\right)^{-1} \bX^\top \bar{\Lambda} \bU \delta^0 \right\|_2 \\
	& \leq \lambda_{\max}\left(\left(2 \lambda_2 I + \bX^\top \bar{\Lambda} \bX\right)^{-1}\right) \|L' \|_2 +\lambda_{\max}\left(\bX^\top \left(2 \lambda_2 \bar{\Lambda}^{-1} + \bX\bX^\top\right)^{-1}\right) \lambda_{\max}(\bU) \|\delta^0 \|_2,
	\end{align*}
using \eqref{eq:kernel_ridge} for the final line. 

It then follows from~\eqref{eq:subgx5},~\eqref{eq:subgx6},~\eqref{eq:subgx3} and Assumption~\ref{cd:subg} that $\|b^0\|_2 \lesssim (1 + c_{\lambda_2}^{-1})\frac{\sqrt{p}}{r_{\lo}}$. Using \eqref{eq:subgx3} once more, we obtain
	\begin{align*}
		\|\bX b^0 - \bU \delta^0\|_2 &= \|\mb Z b^0 + \bU(\Gamma b^0 + \delta^0)\|_2  \notag\\
		&\leq \|\mb Z\|_{\op} \|b^0\|_2 + \|\bU\|_{\op} \|\Gamma b^0 + \delta^0\|_2 \notag\\
		& \lesssim (1 + c_{\lambda_2}^{-1}) \frac{p}{r_\lo \sqrt{n}} + \|\Gamma b^0 - \delta^0\|_2 \\ 
		&\lesssim  (1 + c_{\lambda_2}^{-1}) \left(\sqrt{\frac{p \log p}{r_{\lo} n}} + \frac{p}{r_\lo \sqrt{n}}\right) + c_{\lambda_2} \sqrt{\frac{\log p}{n}},
	\end{align*}
using Lemma~\ref{lem:Gamma_b^0} in the final line.
\end{proof}

%

\subsubsection{Bounds relating to \texorpdfstring{$\cT_3$}{T3}}

\begin{lemma}
	Assume Assumptions~\ref{cd:pns}--\ref{cd:subg}. Given $m \in \N, c_{\lambda_2} > 0$, there exist constants $c_{t3}, c' > 0$ such that 
	\[
	\pr(\cT_3(c_{t3})) \geq 1 - c' n^{-m}.
	\]
\end{lemma}
\begin{proof}
	Let $\mathcal{A}_2$ be as in Lemma~\ref{lem:eigmin2}, and let $c_t > 0$ be such that there exists $c'>0$ where $\pr(\mathcal{A}_1(c_t) \cap \mathcal{A}_2) \geq 1 - c'n^{-m}$. In the following we work on the corresponding event. We will use $a \lesssim b$ to denote that there exists constant $C>0$ potentially depending on $c_t$, $c_{\lambda_2}$ and anything designated as a constant, such that $a \leq Cb$.

We have
\begin{align}\label{eq:t3decomp}
\big\|\bX \big( 2 \lambda_2 I &\; + \bX^\top \Lambda \bX \big)^{-1} L'\big\|_2 = \norm{(\bZ + \bU \Gamma) \left( 2 \lambda_2 I + \bX^\top \Lambda \bX \right)^{-1} L'}_2 \nonumber \\
& \leq \norm{\bZ \left( 2 \lambda_2 I + \bX^\top \Lambda \bX \right)^{-1} L'}_2 + \norm{\bU \Gamma \left( 2 \lambda_2 I + \bX^\top \Lambda \bX \right)^{-1} L'}_2 \\
& \leq \|\bZ\|_\op \norm{\left( 2 \lambda_2 I + \bX^\top \Lambda \bX \right)^{-1}}_{\op} \|L'\|_2 + \norm{\bU \Gamma \left( 2 \lambda_2 I + \bX^\top \Lambda \bX \right)^{-1} L'}_2  \nonumber \\
&:= \mathrm{I} + \mathrm{II}.
\end{align}

To control $\mathrm{I}$, note that using~\eqref{eq:subgx3} and \eqref{eq:subgx5}, we have $\mathrm{I} \lesssim p / (r_\ell \sqrt{n})$.

To control $\mathrm{II}$, observe that
\begin{align*}
	\mathrm{II} \leq \|\bU\|_{\op} \norm{\Gamma \left( 2 \lambda_2 I + \bX^\top \Lambda \bX \right)^{-1} L'}_2 \lesssim \sqrt{\frac{p \log p}{r_{\lo} n}} + \frac{p}{r_\lo \sqrt{n}},
\end{align*}
using \eqref{eq:subgx3} and Lemma~\ref{lem:Gamma_Lambda_bd}. Putting things together yields the desired result.
\end{proof}

\subsubsection{Bounds relating to \texorpdfstring{$\cT_4$}{T4}} \label{sec:T_4}

\begin{lemma}
	Assume Assumptions~\ref{cd:pns}--\ref{cd:subg}. Given $m \in \N, c_{\lambda_2} > 0$, there exist constants $c_{t4}, c' > 0$ such that 
	\[
	\pr(\cT_4(c_{t4})) \geq 1 - c' n^{-m}.
	\]
\end{lemma}

\begin{proof}
	In view of \eqref{eq:subgx2}, \eqref{eq:subgx1} and \eqref{eq:subgx5}, it suffices to show that there exist constants $c, c' >0$ such that with probability at least $1-c'n^{-m}$,
	\[
	\left\|\bX^\top \left(2 \lambda_2 (\Lambda^0)^{-1} + \bX\bX^\top\right)^{-1} \bX L' \right\|_\infty \leq c \sqrt{\frac{\log p}{n}}.
	\]
	Observe that
	\begin{align*}
		&\norm{\mb X \mb X^\top \left(2 \lambda_2 (\Lambda^0)^{-1} + \bX\bX^\top\right)^{-1} \mb X_j}_2^2  \\
		& \qquad =\bX_j^\top \left( 2 \lambda_2 (\Lambda^0)^{-1} + \bX\bX^\top\right)^{-1} \left( \bX\bX^\top \right)^2 \left( 2 \lambda_2 (\Lambda^0)^{-1} + \bX\bX^\top\right)^{-1} \bX_j \leq \|\mb X_j\|_2^2,
	\end{align*}
using Lemma~\ref{lem:Loewner} for the final inequality. Recalling the definition of $L'$ \eqref{eq:L'}, we see from Assumption~\ref{cd:subg} that $\bX_j^\top \left(2 \lambda_2 (\Lambda^0)^{-1} + \bX\bX^\top\right)^{-1} \bX L'$ is sub-Gaussian conditionally on $(\mb X, \bU)$ with variance proxy $ \sigma_{\varepsilon}^2 \|\mb X_j\|_2^2 / n$. Thus there exists constants $c, c'>0$ such that
\[
\PP\left( \left\{\norm{\bX^\top \left( 2 \lambda_2 (\Lambda^0)^{-1} + \bX\bX^\top\right)^{-1} \bX L' }_\infty \geq c \|\mb X_j\|_2 \sqrt{\frac{\log p}{n}}\right\} \;\bigg|\; \bX, \bU\right) \leq c' p^{-m}.
\]
Then in view of \eqref{eq:subgx1}, we see that for potentially different constants $c,c'>0$, we have the desired result.
\end{proof}

\subsubsection{Proof of Lemma~\ref{lem:beta}} \label{sec:beta}
Let $c_t > 0$ be such that there exists $c'>0$ where $\pr(\cT_1(c_t) \cap \cT_2(c_t) \cap \cT_3(c_t) \cap \cT_4(c_t)) \geq 1 - c'n^{-m}$.
In the following we work on the corresponding event. We will use $a \lesssim b$ to denote that there exists constant $C>0$ potentially depending on $c_t$, $c_{\lambda_2}$ and anything designated as a constant, such that $a \leq Cb$.

Recall the definitions of $Q_1$ and $Q_2$ in \eqref{eq:b^0_decomp}, which sum to give $\lambda_2b^0$.
We first tackle $\|Q_1\|_\infty$.

Recall that $\bar{\Lambda}_{ii} := f' (X_i^\top \beta^0 + t_i U_i^\top \delta^0 + (1 - t_i) X_i^\top b^0)$ where $t_i \in [0,1]$ \eqref{eq:Lambda_bar_def}. Now for $\eta \in [0,1]$, let $\Lambda(\eta) := (1 - \eta) \Lambda^0 + \eta \bar{\Lambda}$. Further define 
\[
Q(\eta) := \lambda_2 \left( 2 \lambda_2 I + \bX^\top \Lambda(\eta) \bX\right)^{-1} L',
\]
where $\hat{\Lambda} := \bar{\Lambda} - \Lambda^0$. Then 
\[
\frac{\partial Q(\eta)}{\partial \eta} = \lambda_2 \left( 2\lambda_2 I + \bX^\top \Lambda(\eta) \bX\right)^{-1} \bX^\top \hat{\Lambda} \bX \left( 2 \lambda_2 I + \bX^\top \Lambda(\eta) \bX\right)^{-1} L'.
\]
Fix $j \in [p]$ and write $e_j \in \R^p$ for the $j$th basis vector. Noting that $Q_1 = Q(1)$, we have from the mean value theorem, that
\begin{align*}
	\frac{1}{\lambda_2} e_j^\top Q_1 &= e_j^\top \underbrace{\left( 2 \lambda_2 I + \bX^\top \Lambda^0 \bX\right)^{-1} L'}_{=:Q_{11} / (2 \lambda_2)} + \underbrace{e_j^\top \left( 2\lambda_2 I + \bX^\top \tilde{\Lambda} \bX\right)^{-1} \bX^\top \hat{\Lambda} \bX \left( 2 \lambda_2 I + \bX^\top \tilde{\Lambda} \bX\right)^{-1} L'}_{=:Q_{12j} / (2 \lambda_2)} , 
\end{align*}
where $\tilde{\Lambda} \in \R^{n \times n}$ is a diagonal matrix such that $\tilde{\Lambda}_{ii} := t_i' \Lambda^0_{ii} + (1 - t_i') \bar{\Lambda}_{ii}$ for some $t_i' \in [0,1]$ (note we have suppressed dependence on $j$). 

To control $Q_{11}$, we argue as follows:
\begin{align} \label{eq:xtransform}
2 \lambda_2 \left( 2 \lambda_2 I + \bX^\top \Lambda^0 \bX\right)^{-1} & = I - \left( 2 \lambda_2 I + \bX^\top \Lambda^0 \bX\right)^{-1} \bX^\top \Lambda^0 \bX \nonumber \\
&= I - \bX^\top (\Lambda^0)^{1/2} \left( 2 \lambda_2 I + (\Lambda^0)^{1/2} \bX \bX^\top (\Lambda^0)^{1/2}\right)^{-1}  (\Lambda^0)^{1/2} \bX \nonumber \\ 
&= I - \bX^\top \left( 2 \lambda_2 (\Lambda^0)^{-1} + \bX\bX^\top\right)^{-1}  \bX,
\end{align}
using \eqref{eq:kernel_ridge0} with $A= (\Lambda^0)^{1/2}\mb X$ in the penultimate line.
This gives us that
\[
\|Q_{11}\|_\infty \leq \|L' \|_\infty + \norm{\bX^\top \left( 2 \lambda_2 (\Lambda^0)^{-1} + \bX\bX^\top\right)^{-1} \bX L' }_\infty \lesssim \sqrt{\frac{\log p}{n}},
\]
using the fact that we are working on  $\cT_4$. Thus by H\"older's inequality,
\begin{equation} \label{eq:Q_11_bd}
	\max_j |e_j^\top Q_{11}| \leq \|Q_{11}\|_\infty \lesssim  \sqrt{\frac{\log p}{n}}.
\end{equation}

We next turn to $Q_{12j}$. Observe that
\begin{align*}
Q_{12j}  &= \bX_j^\top \hat{\Lambda} \bX \left( 2 \lambda_2 I + \bX^\top \tilde{\Lambda} \bX\right)^{-1} L'\\
& \qquad - \bX_j^\top \tilde{\Lambda} \bX \left( 2 \lambda_2 I + \bX^\top \tilde{\Lambda} \bX\right)^{-1} \bX^\top \hat{\Lambda} \bX \left( 2 \lambda_2 I + \bX^\top \tilde{\Lambda} \bX\right)^{-1} L'.
\end{align*}
Then using the fact that (analogously to \eqref{eq:xtransform})
\begin{align*}
\bX^\top \tilde{\Lambda} \bX \left( 2 \lambda_2 I + \bX^\top \tilde{\Lambda} \bX \right)^{-1} &= \bX^\top \tilde{\Lambda}^{1/2} \left( 2 \lambda_2 I + \tilde{\Lambda}^{1/2} \bX \bX^\top \tilde{\Lambda}^{1/2} \right)^{-1} \tilde{\Lambda}^{1/2} \bX^\top \\
&= \bX^\top \left( 2 \lambda_2 \tilde{\Lambda}^{-1} + \bX\bX^\top \right)^{-1} \bX,
\end{align*}
one can further rewrite $Q_{12j}$ as
\begin{align}
Q_{12j} &=  \bX_j^\top \hat{\Lambda} \bX \left( 2 \lambda_2 I + \bX^\top \tilde{\Lambda} \bX\right)^{-1} L' \nonumber\\
&\qquad - \bX_j^\top \left( 2 \lambda_2 \tilde{\Lambda}^{-1} + \bX\bX^\top\right)^{-1} \bX \bX^\top \hat{\Lambda} \bX \left( 2 \lambda_2 I + \bX^\top \tilde{\Lambda} \bX\right)^{-1} L'. \label{eq:2nd_term}
\end{align}
Now by the Cauchy--Schwarz inequality and H\"older's inequality, we have
\begin{align*}
	\bX_j^\top \hat{\Lambda} \bX \left( 2 \lambda_2 I + \bX^\top \tilde{\Lambda} \bX\right)^{-1} L' &\leq \|\hat{\Lambda} \bX_j\|_2 \norm{\bX \left( 2 \lambda_2 I + \bX^\top \tilde{\Lambda} \bX\right)^{-1} L'}_2 \\
	&\leq \frac{1}{\sqrt{n}}\max_i |X_{ij}| \|\hat{\Lambda}\|_{\mathrm{F}} \norm{\bX \left( 2 \lambda_2 I + \bX^\top \tilde{\Lambda} \bX\right)^{-1} L'}_2,
\end{align*}
where $\|\cdot\|_{\mathrm{F}}$ denotes the Frobenius norm of the matrix.
Arguing similarly for \eqref{eq:2nd_term}, we see that as we are working on $\cT_1$ and $\cT_4$, we have
\begin{equation}\label{eq:thmtbeta1}
\max_{j \in [p]} |Q_{12j}| \lesssim \left( \sqrt{\log p} + \frac{\sqrt{pn \log p}}{r_{\lo}}\right)\cdot \frac{1}{\sqrt{n}} \|\hat{\Lambda}\|_{\mathrm{F}}  \cdot \norm{\bX (2 \lambda_2 I + \bX^\top \tilde{\Lambda} \bX)^{-1} L' }_2.
\end{equation}
To control $\|\hat{\Lambda}\|_{\mathrm{F}} / \sqrt{n}$, note that
\begin{align*}
\|\hat{\Lambda}\|_{\mathrm{F}} / \sqrt{n} \lesssim \|\bX b^0 - \bU \delta^0\|_2.
\end{align*}
Thus as we are on $\cT_2$, we have
\begin{equation}\label{eq:thmtbeta2}
\frac{1}{\sqrt{n}} \|\hat{\Lambda}\|_{\mathrm{F}}  \lesssim \sqrt{\frac{p \log p}{r_{\lo} n}} + \frac{p}{r_{\lo}\sqrt{n}} + \sqrt{\frac{\log p}{n}} \lesssim \frac{p \sqrt{\log p}}{r_\lo \sqrt{n}}.
\end{equation}
Finally putting together~\eqref{eq:thmtbeta1},~\eqref{eq:thmtbeta2} and using the fact that we are working on $\cT_3$, we have that
\begin{align*}
\max_{j \in [p]} |Q_{12j}| &\lesssim 
\frac{p^2 \log p}{r_\lo^2 n} + \frac{p^{5/2} \log p}{r_\lo^3 \sqrt{n}} + \frac{p^{3/2}(\log p)^{3 / 2}}{r_\ell^{3/2} n} + \frac{p^2 (\log p)^{3 / 2}}{r_\ell^{5/2} \sqrt{n}} \\
&\lesssim \sqrt{\frac{\log p}{n}},
\end{align*}
using Assumptions~\ref{cd:pns}(i)--(iii). Thus applying H\"older's inequality as in \eqref{eq:Q_11_bd}, we see that $
\|Q_1\|_\infty \lesssim \sqrt{\frac{\log p}{n}}$.

Finally we consider $Q_2$. By \eqref{eq:kernel_ridge},
\[
\|Q_2\|_\infty = \lambda_2 \max_j \bigg| e_j^\top  \left( 2 \lambda_2 I + \bX^\top \bar{\Lambda} \bX\right)^{-1} \bX^\top \bar{\Lambda} \bU \delta^0 \bigg| = \lambda_2 \max_j \bigg| \bX_j^\top  \left( 2 \lambda_2 \bar{\Lambda}^{-1} + \bX\bX^\top\right)^{-1} \bU \delta^0 \bigg|.
\]
Then using the Cauchy--Schwarz inequality and the fact that we are on $\cT_4$, 
\begin{align*}
\|Q_2\|_\infty & \lesssim \norm{ \lambda_2 \left( 2 \lambda_2 \bar{\Lambda}^{-1} + \bX^\top \bX\right)^{-1} \bU \delta^0 }_2.
\end{align*}
Finally using the fact that we are on $\cT_2$, we have $\| Q_2 \|_\infty \lesssim  \sqrt{\frac{\log p}{n}}$. In light of our control of $Q_1$ and $Q_2$ and the decomposition in~\eqref{eq:b^0_decomp}, the desired result follows.

\subsubsection{Proof of Lemma~\ref{lem:rsc}}\label{sec:rsc}

\begin{proof}
		Let $c_t > 0$ be such that there exists $c'>0$ where 
		$\pr(\cT_1(c_t) \cap \cT_2(c_t)) \geq 1 - c'n^{-m}$
		(note that $c'$ may depend on $c_{\lambda_2}$).
	In the following we work on 
	$\cT_1(c_t) \cap \cT_2(c_t) \cap \Omega(\tau, \kappa, c_p)$. We will use $a \lesssim b$ to denote that there exists constant $C>0$ potentially depending on $c_t$ and anything designated as a constant, such that $a \leq Cb$. For convenience, we will write $\tau' = \tau/2$.
	
	Fix $\beta\in \R^p$ such that $\beta -\beta^0 \in C(S)$ and \eqref{eq:ballsize} is satisfied.
For simplicity, throughout this proof we denote $b(\beta)$ as $b$.

Recall that $\rem$ denotes the remainder in a first-order Taylor expansion of $L$ about $\beta^0 + b^0$. Considering the Lagrange form of the remainder, there exists $t \in [0,1]$ such that defining diagonal matrix $\Lambda^{(2)} \in \R^{n \times n}$ by $\Lambda^{(2)}_{ii} := f'( t X_i^\top (\beta^0 + b^0) + (1 - t) X_i^\top (\beta + b))$, we have
\begin{align*}
\rem(\beta + b, \beta^0 + b^0) = \big( (\beta + b) - (\beta^0 + b^0)\big)^\top \bX^\top \Lambda^{(2)} \bX \big( (\beta + b) - (\beta^0 + b^0)\big).
\end{align*}
Let $\Delta := (\beta + b) - (\beta^0 + b^0)$, $\tilde{\Delta}_{i} := X_i^\top b^0 - U_i^\top \delta^0$ and $\eta_i := X_i^\top \beta^0 + U_i^\top \delta^0$. Then for all $i \in [n]$ such that
\[
|X_i^\top \Delta| \leq \tau', \; |\eta_i| \leq 2\tau', \; \text{and } |\tilde{\Delta}_{i}| \leq \tau' ,
\]
we must have
\[
| t X_i^\top (\beta^0 + b^0) + (1 - t) X_i^\top (\beta + b)| \leq |\Delta| + |\tilde{\Delta}_{i}| + |\eta_i| \leq 4 \tau'.
\]
Thus for such $i$, $\Lambda^{(2)}_{ii} \geq \min_{|t| \leq 4 \tau'} f'(t)$, which is a positive constant depending only on $\tau$.
This gives us that
\begin{align*}
\rem( & \beta + b,  \beta^0 + b^0)\\
& \gtrsim  \Delta^\top \frac{1}{n} \sum_{i=1}^n X_i X_i^\top \one_{\left\{|X_i^\top \Delta| \leq \tau', |\eta_i| \leq 2\tau', |\tilde{\Delta}_{i}| \leq \tau' \right\}} \Delta \\
& = \Delta^\top \frac{1}{n} \sum_{i=1}^n X_i X_i^\top \one_{\left\{|X_i^\top \Delta| \leq \tau', |\eta_i| \leq 2\tau'\right\}} \Delta  \\
& \quad -  \Delta^\top \frac{1}{n} \sum_{i=1}^n X_i X_i^\top \one_{\left\{|X_i^\top \Delta| \leq \tau', |\eta_i| \leq 2\tau'\right\}} \one_{\left\{ |\tilde{\Delta}_{i}| > \tau'\right\}} \Delta \\
& \geq \Delta^\top \frac{1}{n} \sum_{i=1}^n X_i X_i^\top \one_{\left\{|X_i^\top \Delta| \leq \tau', |\eta_i| \leq 2\tau'\right\}} \Delta - \Delta^\top \frac{1}{n} \sum_{i=1}^n X_i X_i^\top \one_{\left\{ |\tilde{\Delta}_{i}| > \tau'\right\}} \Delta.
\end{align*}
Now by applying H\"{o}lder's inequality to the second term and using~\eqref{eq:subgx2}, we have
\begin{align}
& \rem(\beta + b, \beta^0 + b^0) \nonumber \\
& \gtrsim \Delta^\top \frac{1}{n} \sum_{i=1}^n X_i X_i^\top \one_{\left\{|X_i^\top \Delta| \leq \tau', |\eta_i| \leq 2\tau'\right\}} \Delta - c_t  \frac{1}{n} \sum_{i=1}^n \one_{\left\{ |\tilde{\Delta}_{i}| > \tau'\right\}} \log(p) \|\Delta\|_1^2. \nonumber
\end{align}

Now observe that
\[
\frac{1}{n} \sum_{i=1}^n \one_{\left\{ \tilde{\Delta}_{i}^2> \tau'^2\right\}} \leq \frac{1}{\tau'^2} \|\mb X b^0 - \mb U \delta^0\|_2^2. 
\]
Thus
as we are on $\mathcal{T}_2(c_t)$, we have that
\[
\frac{1}{n} \sum_{i=1}^n \one_{\left\{ |\tilde{\Delta}_{i}| > \tau'\right\}} \lesssim  (1 + c_{\lambda_2}^{-1})^2\left(\frac{p \log p}{r_{\lo} n} + \frac{p^2}{r_{\lo}^2 n}\right) +  c_{\lambda_2}^2\frac{\log p}{n}, 
\]
which gives us that there exists $c>0$ with
\begin{equation} \label{eq:E_eq}
	\begin{split}
		\rem(\beta + b, \beta^0 + b^0) \gtrsim &   \Delta^\top \frac{1}{n} \sum_{i=1}^n X_i X_i^\top \one_{\left\{|X_i^\top \Delta| \leq \tau', |\eta_i| \leq 2\tau'\right\}} \Delta \\
		&- c \left( (1 + c_{\lambda_2}^{-1})^2 \left(\frac{p (\log p)^2}{r_{\lo} n} + \frac{p^2 \log p}{r_{\lo}^2 n}\right) + c_{\lambda_2}^2\frac{(\log p)^2}{n}\right)  \|\Delta\|_1^2. 
	\end{split}
\end{equation}

We now turn to the first term in the above decomposition. Recall that
\begin{align*}
2 \lambda_2 b^0 &= \frac{1}{n} \sum_{i = 1}^n \left\{Y_i - f\left(X_i^\top (\beta^0 + b^0)\right)\right\} X_i,\\
2 \lambda_2 b &= \frac{1}{n} \sum_{i = 1}^n \left\{Y_i - f\left(X_i^\top (\beta + b)\right)\right\} X_i.
\end{align*}
Subtracting the above two equalities yields
\[
2 \lambda_2 (b^0 - b) = - \frac{1}{n} \sum_{i = 1}^n \left\{f\left(X_i^\top (\beta^0 + b^0)\right) - f\left(X_i^\top (\beta + b)\right)\right\} X_i.
\]
Now applying the mean value theorem, we have that there exists $t_i \in [0,1]$ such that
\begin{align*}
&f\left(X_i^\top (\beta^0 + b^0)\right) - f\left(X_i^\top (\beta + b)\right) \\
&\qquad = f'\left(t_i X_i^\top (\beta^0 + b^0) + (1 - t_i) X_i^\top (\beta + b)\right) \left(X_i^\top (\beta^0 + b^0) - X_i^\top (\beta + b)\right).
\end{align*}
Thus defining $\Lambda^{(1)} \in \R^{n \times n}$ to be a diagonal matrix with $\Lambda^{(1)}_{ii} := f'(t_i X_i^\top (\beta^0 + b^0) + (1 - t_i) X_i^\top (\beta + b))$, we have
\begin{equation}\label{eq:lambda1}
	f\left(X_i^\top (\beta^0 + b^0)\right) - f\left(X_i^\top (\beta + b)\right)  =  \sqrt{n} \left(\Lambda^{(1)} \mb X (\beta - \beta^0 + b - b^0)\right)_{i},
\end{equation}
so
\[
2\lambda_2 (b - b^0) = \mb X^\top \Lambda^{(1)} \mb X (\beta^0 - \beta  + b^0 - b).
\]
This then yields
\begin{equation} \label{eq:b_b^0}
b - b^0 = - \left(2 \lambda_2 I + \bX^\top \Lambda^{(1)} \mb X\right)^{-1} \bX^\top \Lambda^{(1)} \bX(\beta - \beta^0).
\end{equation}

From the above, we have that
\begin{align*}
X_i^\top \Delta & = X_i^\top (\beta - \beta^0) - X_i^\top \left( 2 \lambda_2 I + \bX^\top \Lambda^{(1)} \bX\right)^{-1} \bX^\top \Lambda^{(1)} \bX (\beta - \beta^0) \\
& = X_i^\top (\beta - \beta^0) - X_i^\top \bX^\top \left( 2 \lambda_2 (\Lambda^{(1)})^{-1} + \bX\bX^\top\right)^{-1} \bX (\beta - \beta^0) \\
&\leq \left(\|X_i\|_\infty + \max_j \abs{X_i^\top \bX^\top \left( 2 \lambda_2 (\Lambda^{(1)})^{-1} + \bX\bX^\top\right)^{-1} \bX_j}\right) \|\beta - \beta^0\|_1.
\end{align*}
using a similar argument to \eqref{eq:kernel_ridge} for the second equality, and H\"older's inequality in the final line.
Now as $\beta - \beta^0 \in C(S)$, we have by the Cauchy--Schwarz inequality that
\begin{equation} \label{eq:sparse_ineq}
\|\beta - \beta^0\|_1 \leq 4 \|\beta_S - \beta^0_S\|_1 \leq 4\sqrt{s} \|\beta - \beta^0\|_2.
\end{equation}
Thus using \eqref{eq:subgx2} and that we are on $\cT_1(c_t)$, we have
\[
\max_{i} |X_i^\top \Delta| \lesssim (1 + c_{\lambda_2}^{-1})\left( \sqrt{s \log p} + \frac{\sqrt{spn \log p}}{r_{\lo}}\right) \|\beta - \beta^0\|_2.
\]
Now by denoting the implicit constant in the above by $C>0$, let us set $r := C^{-1}\tau' / 2$. Then for $\beta$ satisfying  \eqref{eq:ballsize}, provided $c_{\lambda_2} \geq 1$, we have that $|X_i^\top \Delta| \leq \tau'$ for all $i$. In what follows, we will enforce \eqref{eq:ballsize}, and assume $c_{\lambda_2} \geq 1$: our choice of $c_{\lambda_2}$ will ensure this latter condition.

Returning to \eqref{eq:E_eq}, we have
\begin{align*}
\rem(\beta + b, \beta^0 + b^0) \gtrsim & \Delta^\top \frac{1}{n} \sum_{i=1}^n X_i^\top X_i \one_{\{|\eta_i| \leq 2\tau'\}} \Delta \\
& - c \left((1 + c_{\lambda_2}^{-1})^2 \left(\frac{p (\log p)^2}{r_{\lo} n} + \frac{p^2 \log p}{r_{\lo}^2 n}\right) + c_{\lambda_2}^2 \frac{(\log p)^2}{n}\right)  \|\Delta\|_1^2.
\end{align*}
Then as we are on the event $\Omega$, it holds that
\begin{align*}
& \rem( \beta + b, \beta^0 + b^0) \gtrsim  \kappa \|\beta - \beta^0\|_2^2 -  c \left((1 + c_{\lambda_2}^{-1})^2 \left(\frac{p (\log p)^2}{r_{\lo} n} + \frac{p^2 \log p}{r_{\lo}^2 n}\right) + c_{\lambda_2}^2 \frac{(\log p)^2 }{n}\right) \|\Delta\|_1^2 \\
&\qquad -  c_p r_{\ell} \sqrt{\frac{\log p}{n}} \|b - b^0\|_2^2\\
& \gtrsim \kappa \|\beta - \beta^0\|_2^2 -  c s \left((1 + c_{\lambda_2}^{-1})^2 \left(\frac{p (\log p)^2}{r_{\lo} n} + \frac{p^2 \log p}{r_{\lo}^2 n}\right) + c_{\lambda_2}^2 \frac{(\log p)^2 }{n}\right) \|\beta - \beta^0\|_2^2 \\
& \quad - c \left((1 + c_{\lambda_2}^{-1})^2 \left(\frac{p^2 (\log p)^2}{r_{\lo} n} + \frac{p^3 \log p}{r_{\lo}^2 n}\right) + c_{\lambda_2}^2 \frac{p (\log p)^2 )}{n}\right) \|b - b^0\|_2^2 - c_p r_{\ell} \sqrt{\frac{\log p}{n}} \|b - b^0\|_2^2.
\end{align*}
where for the last inequality we use that $\beta - \beta^0 \in C(S)$; note the constant $c$ is potentially different in the second line.
By Assumptions~\ref{cd:pns}(i), (ii) and~(iii), we have that
\[
\frac{p (\log p)^2}{n} = o\left( r_{\lo} \sqrt{\frac{\log p}{n}} \right) \qquad \text{ and } \qquad r_{\lo} \sqrt{\frac{\log p}{n}} \gtrsim \frac{p^2 (\log p)^2}{r_{\lo} n} + \frac{p^3 \log p}{r_\ell^2 n}.
\]
Then by choosing the constant $c_{\lambda_2} > 0$ large enough, we have that for $n$ sufficiently large,
\[
c_{\lambda_2} r_{\lo} \sqrt{\frac{\log p}{n}} - c \left((1 + c_{\lambda_2}^{-1})^2 \left(\frac{p^2 (\log p)^2}{r_{\lo} n} + \frac{p^3 \log p}{r_{\lo}^2 n}\right) + c_{\lambda_2}^2 \frac{p (\log p)^2}{n}\right) - c_p r_{\ell} \sqrt{\frac{\log p}{n}} > 0.
\]
Thus for such $n$, 
\begin{align*}
&\rem(\beta + b, \beta^0 + b^0) + \lambda_2 \|b - b^0\|_2^2 \geq \\
& \qquad \kappa \|\beta - \beta^0\|_2^2  - c s \left( (1 + c_{\lambda_2}^{-1})^2 \left(\frac{p (\log p)^2}{r_{\lo} n} + \frac{p^2 \log p}{r_{\lo}^2 n}\right) + c_{\lambda_2}^2 \frac{(\log p)^2}{n}\right) \|\beta - \beta^0\|_2^2.
\end{align*}
The desired result then follows from Assumptions~\ref{cd:pns}(iii) and (iv).
\end{proof}

%

\section{Theoretical analysis of prediction errors} \label{sec:pfinf}
We first introduce some notation used in the section. Note that these notations have also been used in Section~\ref{sec:pfmain}. We write $e_j$ for the $j$th standard basis vector, and write $I$ for the identity matrix, where the dimensions of these will be clear from the context. For symmetric matrices $A, B$ of the same dimensions, we write $A \succeq B$ to indicate that $A-B$ is positive semi-definite. We use $\|A\|_{\op}$ and $\lambda_{\max}(A)$ interchangeably for the maximum singular value of an arbitrary matrix $A$, and write $\lambda_{\min}(A)$ for its minimum singular value.

For notational simplicity, we rewrite $\bX/\sqrt{n}$ as $\bX$, and analogously rescale all the other matrices and vectors $\bU$, $\bZ$, $\be$, $\bY$ relating to the observed data. We retain the original (unscaled) definitions of the random variables $X_i$ etc.\ forming these matrices, and take $X_{ij}$ to be the $j$th component of $X_i$ rather than the $ij$th entry of $\mb X$.

We also set
\[
c_d := \sup_t f'(t).
\]
\subsection{Proof of Theorem~\ref{thm:inpred}}
Let $c_{\lambda_1}$ and $c_{\lambda_2}$ be as in Theorem~\ref{thm:main}. In the following, we work on the intersection of the event given by Theorem~\ref{thm:main}, $\Omega$ and $\mathcal{A}_1(c_t)\cap \mathcal{T}_2(c_t)$ (see Lemma~\ref{lem:subgx}) with $c_t>0$ chosen such that for some $c'>0$, the probability of this intersection is at least $\pr(\Omega) - c'n^{-m}$.
 We will use $a \lesssim b$ to denote that there exists constant $C>0$ potentially depending on $c_t, c_{\lambda_2}, c_{\lambda_1}$ and anything designated as a constant, such that $a \leq Cb$. 
 
Now observe that 
\begin{align}\label{eq:insampledecomp}
& \left(\frac{1}{n} \sum_{i = 1}^n (f(X_i^\top (\hat{\beta} + \hat{b})) - f(X_i^\top \beta^0 + U_i^\top \delta^0))^2 \right)^{1/2} \nonumber\\
&\qquad \leq \left(\frac{1}{n} \sum_{i = 1}^n (f(X_i^\top (\hat{\beta} + \hat{b})) - f(X_i^\top (\beta^0 + b^0))^2\right)^{1/2} + \left(\frac{1}{n} \sum_{i = 1}^n (f(X_i^\top (\beta^0 + b^0)) - f(X_i^\top \beta^0 + U_i^\top \delta^0))^2\right)^{1/2} \nonumber\\
&\qquad =: \mathrm{I} + \mathrm{II}.
\end{align}
We first control $\mathrm{I}$.

Let $\Lambda^{(1)} \in \R^{n \times n}$ be as in~\eqref{eq:lambda1} but where $\beta$ is set to $\hat{\beta}$ (and so $\hat{b}$ is $b$).
We have
\[
\mathrm{I} = \|\Lambda^{(1)} \bX (\hat{\beta} -\beta^0 +\hat{b} - b^0) \|_2.
\]
Now from \eqref{eq:b_b^0} we have
\begin{align} \label{eq:hat_b-b^0}
	\hat{b}-b^0 = - \left(2 \lambda_2 I + \bX^\top \Lambda^{(1)} \mb X\right)^{-1} \bX^\top \Lambda^{(1)} \bX(\hat{\beta} - \beta^0),
\end{align}
so
\[
\mathrm{I} =\lambda_2 \norm{ \Lambda^{(1)} \bX  \left( 2 \lambda_2 I + \bX^\top \Lambda^{(1)} \bX\right)^{-1} (\hat{\beta} - \beta^0)}_2.
\]
Then using \eqref{eq:kernel_ridge2} we have
\begin{align}\label{eq:insampledecomp2}
\mathrm{I}
&= \lambda_2 \norm{\left( 2 \lambda_2 (\Lambda^{(1)})^{-1} + \bX\bX^\top\right)^{-1} \bX (\hat{\beta} - \beta^0)}_2 \nonumber\\
&\leq \lambda_2 \norm{\left( 2 \lambda_2 (\Lambda^{(1)})^{-1} + \bX\bX^\top\right)^{-1} \bZ (\hat{\beta} - \beta^0)}_2 + \lambda_2 \norm{\left( 2 \lambda_2 (\Lambda^{(1)})^{-1} + \bX\bX^\top\right)^{-1} \bU \Gamma (\hat{\beta} - \beta^0)}_2  =: \mathrm{I}_1 + \mathrm{I}_2.\nonumber
\end{align}

To bound $\mathrm{I}_2$, we argue as follows. We have
\[
\mathrm{I}_2 \leq   \lambda_2\norm{ \left( 2 \lambda_2 (\Lambda^{(1)})^{-1} + \bX\bX^\top\right)^{-1} \bU}_{\op} \|\Gamma (\hat{\beta} - \beta^0)\|_2.
\]
Now $(\Lambda^{(1)})^{-1} \succeq c_d^{-1} I$. Thus noting that we are on $\mathcal{A}_1(c_t)$, by Lemma~\ref{lem:T_2_first}, we have that
\[
\mathrm{I}_2 \lesssim \sqrt{\frac{\log p}{n}} \|\Gamma (\hat{\beta} - \beta^0)\|_2 \leq C \sqrt{\frac{\log p}{n}} \sqrt{q} \|\Gamma (\hat{\beta} - \beta^0)\|_\infty \lesssim  \|\Gamma\|_\infty \|\hat{\beta} - \beta^0\|_1 \sqrt{\frac{\log p}{n}},
\]
applying H\"older's inequality. Note that by Assumption~\ref{cd:subg},
\begin{equation} \label{eq:Gamma_bd}
1 \gtrsim \Var(X_{ij}) \geq \E (\Gamma_j^{\top}U_iU_i^{\top}\Gamma_j) = \|\Gamma_j\|_2^2,
\end{equation}
where $\Gamma_j \in \R^q$ denotes the $j$th column of $\Gamma$. Thus $\|\Gamma\|_\infty \lesssim 1$ and so appealing to Theorem~\ref{thm:main}, we have
\begin{equation} \label{eq:pred_I2_bd}
	\mathrm{I}_2 \lesssim \frac{s \log p}{n}.
\end{equation}

Turning to $\mathrm{I}_1$, we have from Lemma~\ref{lem:Loewner} that
\begin{align}
	\mathrm{I}_1^2 &\lesssim \|\mb Z(\hat{\beta} - \beta^0)\|_2^2 \nonumber\\
	&=  (\hat{\beta} - \beta^0)^\top \Sigma (\hat{\beta} - \beta^0) + (\hat{\beta} - \beta^0)^\top (\bZ^\top \bZ - \Sigma) (\hat{\beta} - \beta^0)  \nonumber \\
	&\lesssim \|\hat{\beta} - \beta^0\|_2^2 + \|\hat{\beta} - \beta^0\|_1^2 \|\bZ^\top \bZ - \Sigma\|_\infty \nonumber\\
	&\lesssim s \frac{\log p}{n} +  s^2 \left( \frac{\log p}{n}\right)^{3/2} \lesssim s \frac{\log p}{n}, \label{eq:pred_I1_bd}
\end{align}
using Assumption~\ref{cd:subg} and H\"older's inequality in the penultimate line, and appealing to Theorem~\ref{thm:main} and \eqref{eq:subgx7} in the final line, with the final inequality coming from Assumption~\ref{cd:pns}(iv).

We now turn to $\mathrm{II}$. By the mean-value theorem and as we are on $\mathcal{T}_2(c_t)$, we have
\[
\mathrm{II} \leq \|\bX b^0 - \bU \delta^0\|_2 \lesssim \sqrt{\frac{p \log p}{r_{\lo} n}} + \frac{p}{r_\lo \sqrt{n}} + \sqrt{\frac{\log p}{n}}.
\]
In light of the decomposition~\eqref{eq:insampledecomp}, \eqref{eq:pred_I1_bd}, \eqref{eq:pred_I2_bd} and the above, the desired result follows.

\subsection{Proof of Theorem~\ref{thm:prediction}}
The proof of Theorem~\ref{thm:prediction} uses the following variant of Lemma~\ref{lem:term1}.
\begin{lemma} \label{lem:term1_var}
	Given $c_t>0$ and with $c_{\lambda_2}$ bounded away from $0$, there exists $c>0$ such that on the event $\mathcal{A}_1(c_t)$, we have
	\[
	\norm{\Gamma \bX^\top \left( 2 \lambda_2 (\Lambda^{(1)})^{-1} + \bX\bX^\top\right)^{-1} \bX }_\infty \leq c. 
	\]
\end{lemma}
\begin{proof}
	Observe that
		\[
		\norm{\Gamma \bX^\top \left( 2 \lambda_2 (\Lambda^{(1)})^{-1} + \bX\bX^\top\right)^{-1} \bX }_\infty \le \max_{j = 1,\cdots, p}\norm{\Gamma \bX^\top \left( 2 \lambda_2 (\Lambda^{(1)})^{-1} + \bX\bX^\top\right)^{-1} \bX_j }_2.
		\]
		Now for each $j$, we have
		\begin{align*}
			& \norm{\Gamma \bX^\top \left( 2 \lambda_2 (\Lambda^{(1)})^{-1} + \bX\bX^\top\right)^{-1} \bX_j }_2 = \norm{(\bU^\top \bU)^{-1} \bU^\top \bU\Gamma \bX^\top \left( 2 \lambda_2 (\Lambda^{(1)})^{-1} + \bX\bX^\top\right)^{-1} \bX_j }_2^2 \\
			& \le \norm{(\bU^\top \bU)^{-1}}_\op \norm{\bU}_\op \norm{\bU\Gamma \bX^\top \left( 2 \lambda_2 (\Lambda^{(1)})^{-1} + \bX\bX^\top\right)^{-1} }_\op \norm{\bX_j}_2 \\
			& \lesssim \norm{\bU\Gamma \bX^\top \left( 2 \lambda_2 (\Lambda^{(1)})^{-1} + \bX\bX^\top\right)^{-1} }_\op,
		\end{align*}
	where for the last inequality we apply~\eqref{eq:subgx3} and~\eqref{eq:subgx1}. Then using the decomposition $\bU\Gamma = \bX - \bZ$, we further have
	\begin{align*}
		 &\norm{\Gamma \bX^\top \left( 2 \lambda_2 (\Lambda^{(1)})^{-1} + \bX\bX^\top\right)^{-1} \bX_j }_2 \\
		 & \lesssim \norm{\bX \bX^\top \left( 2 \lambda_2 (\Lambda^{(1)})^{-1} + \bX\bX^\top\right)^{-1} }_\op + \norm{\bZ \bX^\top \left( 2 \lambda_2 (\Lambda^{(1)})^{-1} + \bX\bX^\top\right)^{-1} }_\op \\
		 & \lesssim 1 + \|\bZ \|_\op \norm{\bX^\top \left( 2 \lambda_2 (\Lambda^{(1)})^{-1} + \bX\bX^\top\right)^{-1}}_\op,
	\end{align*}
where for the last inequality we apply Lemma~\ref{lem:Loewner} to get that
\[
\norm{\bX \bX^\top \left( 2 \lambda_2 (\Lambda^{(1)})^{-1} + \bX\bX^\top\right)^{-1} }_\op \le \norm{\bX \bX^\top \left( 2 \lambda_2 c_d^{-1} I + \bX\bX^\top\right)^{-1} }_\op \le 1.
\]
From above, and also~\eqref{eq:subgx3},~\eqref{eq:subgx6} and Assumption~\ref{cd:pns}(iii), we obtain the desired result.
\end{proof}
Let us denote by $\E_{\star}$ an expectation over only $(Z_{\star}, U_{\star})$, conditioning on everything else. First note that by the mean value theorem, it suffices to control
\begin{align}\label{eq:preddecomp}
\E_{\star}[\{ X_{\star}^\top(\hat{\beta} + \hat{b}) \;& -  (X_{\star}^\top \beta^0 + U_{\star}^\top \delta^0)\}^2] \nonumber\\
& \lesssim \E_{\star}[\{X_{\star}^\top (\hat{\beta} - \beta^0)\}^2] + \E_{\star}[(X_{\star}^\top \hat{b} - U_{\star}^\top \delta^0)^2] \nonumber\\
& \lesssim \E_{\star}[\{X_{\star}^\top (\hat{\beta} - \beta^0)\}^2] +  \E_{\star}[Z_{\star}^\top \hat{b}] +  \E_{\star}[U_{\star}^\top(\Gamma \hat{b} - \delta^0)^2] \nonumber \\
& \lesssim \E_{\star}[\{X_{\star}^\top (\hat{\beta} - \beta^0)\}^2] +  \E_{\star}[Z_{\star}^\top \hat{b}] +  \E_{\star}[\{U_{\star}^\top(\Gamma b^0 - \delta^0)\}^2] + \E_{\star}[\{U_{\star}^\top\Gamma (b^0 - \hat{b})\}^2]\nonumber \\
& =: \mathrm{I} + \mathrm{II} + \mathrm{III} + \mathrm{IV}.
\end{align}

For $\mathrm{I}$, observe that 
\begin{align}
\mathrm{I} &\lesssim  \E[\{Z_{\star}^\top (\hat{\beta} - \beta^0)\}^2] +  \E[\{U_{\star}^\top \Gamma (\hat{\beta} - \beta^0)\}^2] \nonumber\\
&= (\hat{\beta} - \beta^0)^\top \Sigma (\hat{\beta} - \beta^0) + \|\Gamma(\hat{\beta} - \beta^0)\|_2^2 \nonumber\\
&\lesssim \|\hat{\beta} - \beta^0\|_2^2 + \|\Gamma\|_\infty^2 \|\hat{\beta} - \beta^0\|_1^2 \nonumber\\
&\lesssim s \frac{\log p}{n} + s^2 \frac{\log p}{n}, \label{eq:pred1}
\end{align}
using H\"older's inequality and Assumption~\ref{cd:subg} in the penultimate line, and appealing to Theorem~\ref{thm:main} and \eqref{eq:Gamma_bd} in the final line.

Turning to $\mathrm{II}$, first note that by the KKT conditions of the 
optimisation problem~\eqref{eq:gen_LAVA}, we have that for each $j=1,\ldots,p$,
\begin{align*}
	\left|\frac{1}{n} \sum_{i=1}^n X_{ij} \ell(Y_i, X_i^\top (\hat{\beta} + \hat{b})) \right| &\le \lambda_1 \\
	\frac{1}{n} \sum_{i=1}^n X_{i,} \ell(Y_i, X_i^\top (\hat{\beta} + \hat{b})) &= 2 \lambda_2 \hat{b}_j.
\end{align*}
Thus $|\hat{b}_j| \le \lambda_1 / (2 \lambda_2)$, and so 
$\|\hat{b}\|_2 \leq \sqrt{p} \frac{\lambda_1}{2 \lambda_2}$. Thus
\begin{equation}\label{eq:pred2}
\mathrm{II} = \hat{b}^\top \Sigma \hat{b} \lesssim \|\hat{b}\|_2^2 \leq p \frac{\lambda_1^2}{\lambda_2^2} \lesssim \frac{p}{r_\lo^2}.
\end{equation}

Next note that
\begin{equation}\label{eq:pred3}
\mathrm{III} = \|(\Gamma b^0 - \delta^0)\|_2^2 \lesssim \frac{p \log p}{r_\ell n} + \frac{p^2}{r_\lo^2 n} + \frac{\log p}{n}
\end{equation}
by Lemma~\ref{lem:Gamma_b^0}.

Finally, to control $\mathrm{IV}$, we use \eqref{eq:hat_b-b^0} and argue as follows. We have
\begin{align*}
\mathrm{IV} & = \|\Gamma (\hat{b} - b^0)\|_2^2 \\
&=\norm{\Gamma\left( 2 \lambda_2 I + \bX^\top \Lambda^{(1)} \bX\right)^{-1} \bX^\top \Lambda^{(1)} \bX (\hat{\beta} - \beta^0)}_2^2 \\
& =  \norm{\Gamma \bX^\top \left( 2 \lambda_2 (\Lambda^{(1)})^{-1} + \bX\bX^\top\right)^{-1} \bX (\hat{\beta} - \beta^0)}_2^2 \\
& \leq q \norm{\Gamma \bX^\top \left( 2 \lambda_2 (\Lambda^{(1)})^{-1} + \bX\bX^\top\right)^{-1} \bX (\hat{\beta} - \beta^0)}_\infty^2 \\
& \leq  \norm{\Gamma \bX^\top \left( 2 \lambda_2 (\Lambda^{(1)})^{-1} + \bX\bX^\top\right)^{-1} \bX }_\infty^2 \|\hat{\beta} - \beta^0\|_1^2,
\end{align*}
where for the final equality we appealed to \eqref{eq:kernel_ridge2}. Thus using Lemma~\ref{lem:term1_var} and appealing to Theorem~\ref{thm:main}, we have
\begin{equation}\label{eq:pred4}
\mathrm{IV} \lesssim s^2 \frac{\log p }{n}.
\end{equation}

In light of the decomposition~\eqref{eq:preddecomp}, the desired result follows from putting together~\eqref{eq:pred1},~\eqref{eq:pred2},~\eqref{eq:pred3},~\eqref{eq:pred4} and that we are under Assumption~\ref{cd:pns}(i).

\section{Proof of Theorem~\ref{thm:inference}} \label{sec:pfinference}

We begin by introducing some preliminary lemmas used in the proof. In all of the following, we assume the conditions of Theorem~\ref{thm:inference}.
\begin{lemma}\label{lem:nulldecomp}
	Defining
	\[
	\delta := \frac{1}{n} \sum_{i=1}^{n} \hat{\varepsilon}_i \hat{\varepsilon}_i^W - \frac{1}{n} \sum_{i=1}^{n} \varepsilon_i \varepsilon_i^W,
	\]
	given constants $m \in \N, \kappa, \tau, c_p > 0$, there exists a constant $c' > 0$ depending on $m, \kappa, \tau, c_p$ such that on an event $\cB_1$ with $\pr(\cB_1^c \cap \Omega(\tau, \kappa, c_p)) \le c' n^{-m}$,
	\[
	|\delta| \lesssim s \frac{\log p}{n} + \frac{p^2}{r_\lo^2 n} + \frac{p \log p}{r_\lo n}.
	\]
\end{lemma}

\begin{proof}
	We have the decomposition 
	\begin{align*}
		\frac{1}{n} \sum_{i=1}^{n} \hat{\varepsilon}_i \hat{\varepsilon}_i^W = & \frac{1}{n} \sum_{i=1}^{n} \varepsilon_i \varepsilon_i^W + \frac{1}{n} \sum_{i=1}^{n} \{f(X_i^\top \beta^0 + U_i^\top \delta^0) - f(X_i^\top (\hat{\beta} + \hat{b}))\} \varepsilon_i^W \\
		& + \frac{1}{n} \sum_{i=1}^{n} \{f^W(X_i^\top \beta^W + U_i^\top \delta^W) - f^W(X_i^\top (\hat{\beta}^W + \hat{b}^W))\} \varepsilon_i \\
		& + \frac{1}{n} \sum_{i=1}^{n} \{f^W(X_i^\top \beta^W + U_i^\top \delta^W) - f^W(X_i^\top (\hat{\beta}^W + \hat{b}^W))\} \{f(X_i^\top \beta^0 + U_i^\top \delta^0) - f(X_i^\top (\hat{\beta} + \hat{b}))\} \\
		& =: \frac{1}{n} \sum_{i=1}^{n} \varepsilon_i \varepsilon_i^W + \mathrm{I} + \mathrm{II} + \mathrm{III}.
	\end{align*}
Note that by the Cauchy--Schwarz inequality,
\begin{align*}
\mathrm{III} &\leq \left(\frac{1}{n} \sum_{i = 1}^n \{f(X_i^\top \beta^0 + U_i^\top \delta^0) - f(X_i^\top (\hat{\beta} + \hat{b}))\}^2 \right)^{1/2} \\
 & \qquad  \left(\frac{1}{n} \sum_{i = 1}^n \{f^W(X_i^\top \beta^W + U_i^\top \delta^W) - f^W(X_i^\top (\hat{\beta}^W + \hat{b}^W))\}^2\right)^{1/2}.
\end{align*}
Thus $\mathrm{III}$ is well controlled on the event
\begin{align*}
	\mathcal{E}_1(c) := \bigg\{ & \frac{1}{n} \sum_{i = 1}^n \{f(X_i^\top \beta^0 + U_i^\top \delta^0) - f(X_i^\top (\hat{\beta} + \hat{b}))\}^2 \le c \left(s \frac{\log p}{n} + \frac{p \log p}{r_{\lo} n} + \frac{p^2}{r_{\lo}^2 n}\right), \; \\ 
	& \frac{1}{n} \sum_{i = 1}^n \{f^W(X_i^\top \beta^W + U_i^\top \delta^W) - f^W(X_i^\top (\hat{\beta}^W + \hat{b}^W))\}^2 \le c \left(s \frac{\log p}{n} + \frac{p \log p}{r_{\lo} n} + \frac{p^2}{r_{\lo}^2 n}\right) \bigg\}.
\end{align*}
From Theorem~\ref{thm:inpred} and its proof we have that there exist constants $c, c' > 0$ such that $\pr(\mathcal{E}_1^c(c) \cap \Omega(\tau, \kappa, c_p)) \le c' n^{-m}$.

	To control $\mathrm{I}$ and $\mathrm{II}$, using the conditional sub-Gaussianity of the $\varepsilon_i$ and $\varepsilon_i^W$'s (Assumption~\ref{cd:subg}), we have that for any $t>0$,
	\begin{align*}
		\pr & \left(\left|\frac{1}{n} \sum_{i=1}^{n} \left(f(X_i^\top \beta^0 + U_i^\top \delta^0) - f(X_i^\top (\hat{\beta} + \hat{b}))\right) \varepsilon_i^W  \right| \geq t \given (\be, \bX, \bU)\right) \\
		& \qquad\qquad\qquad \leq 2 \exp\left(- \frac{\sigma_{\varepsilon} n^2 t^2}{\sum_{i = 1}^n (f(X_i^\top \beta^0 + U_i^\top \delta^0) - f(X_i^\top (\hat{\beta} + \hat{b})))^2}\right),
	\end{align*}
	where $(\hat{\beta}, \hat{b})$ are deterministic given $(\be, \bX, \bU)$; and similarly
	\begin{align*}
		\pr & \left(\left|\frac{1}{n} \sum_{i=1}^{n} \left(f^W(X_i^\top \beta^W + U_i^\top \delta^W) - f^W(X_i^\top (\hat{\beta}^W + \hat{b}^W))\right) \varepsilon_i \right| \geq t \given (\be^W, \bX, \bU) \right) \\
		& \qquad\qquad\qquad \leq 2 \exp\left(- \frac{\sigma_{\varepsilon} n^2 t^2}{\sum_{i = 1}^n (f^W(X_i^\top \beta^W + U_i^\top \delta^W) - f^W(X_i^\top (\hat{\beta}^W + \hat{b}^W)))^2}\right),
	\end{align*}
	where $(\hat{\beta}^W, \hat{b}^W)$ are deterministic given $(\be^W, \bX, \bU)$. In light of the above two sub-Gaussian tail bounds, we have that there exist constants $c, c' > 0$ such that the following event holds with probability at least $1 - c'n^{-m}$:
	\begin{align*}
		\mathcal{E}_2(c) := \bigg\{ & \mathrm{I} \le c \sqrt{\frac{1}{n} \sum_{i = 1}^n \left(f(X_i^\top \beta^0 + U_i^\top \delta^0) - f(X_i^\top (\hat{\beta} + \hat{b}))\right)^2} \cdot \sqrt{\frac{\log n}{n}} ,  \\ 
		& \mathrm{II} \le c \sqrt{\frac{1}{n} \sum_{i = 1}^n \left(f^W(X_i^\top \beta^W + U_i^\top \delta^W) - f^W(X_i^\top (\hat{\beta}^W + \hat{b}^W))\right)^2} \cdot \sqrt{\frac{\log n}{n}}\bigg\}.
	\end{align*}
	
	Working on the intersection of $\mathcal{E}_1$ and $\mathcal{E}_2$, the desired result follows by assigning $\delta := \mathrm{I} + \mathrm{II} + \mathrm{III}$.
\end{proof}

\begin{lemma}\label{lem:nullbnan}
	Given constants $m \in \N, \kappa, \tau, c_p, \epsilon  > 0$, we have that there exists constant $c > 0$ such that on an event $\cB_2$ with $\pr(\cB_2^c \cap \Omega(\tau, \kappa, c_p)) \le c n^{-m}$, it holds that
		\[
		\left|\frac{1}{n} \sum_{i = 1}^n \hat{\varepsilon}_i^2 (\hat{\varepsilon}_i^W)^2 - \left(\frac{1}{n} \sum_{i = 1}^n \hat{\varepsilon}_i \hat{\varepsilon}_i^W\right)^2 - \var(\varepsilon_1 \varepsilon_1^W)\right| \lesssim s\frac{\log p}{\sqrt{n}} + \frac{p \log p}{r_{\ell} \sqrt{n}} + \frac{p^2}{r_{\ell}^2 \sqrt{n}} + n^{-1/2 + \epsilon},
		\]
	and moreover $s \log p / \sqrt{n} + p \log p / (r_{\lo} \sqrt{n}) + p^2 / (r_{\lo}^2 \sqrt{n}) \to 0$.
\end{lemma}
\begin{proof}
	Let us write
	\begin{align*}
		\hat{\Delta}_i &:= f(X_i^\top (\hat{\beta} + \hat{b})) - f(X_i^\top \beta^0 + U_i^\top \delta^0) \\
		\hat{\Delta}^W_i &:= f^W(X_i^\top (\hat{\beta}^W + \hat{b}^W)) - f^W(X_i^\top \beta^W + U_i^\top \delta^W).
	\end{align*}
Let
\[
d_n := s\frac{\log p}{n} + \frac{p \log p}{r_{\ell}n} + \frac{p^2}{r_{\ell}^2 n}.
\]
In the following, we will write $\Omega = \Omega(\tau, \kappa, c_p)$ for simplicity and any constant $c$ may change from line to line.

Let $\mathcal{B}_{21} := \Omega \cup \tilde{\mathcal{B}}_{21}$ be the event that $\sum_{i=1}^n \hat{\Delta}_i^2 /n \leq c' d_n$ and $\sum_{i=1}^n (\hat{\Delta}^W_i)^2 /n \leq c' d_n$ with constant $c'>0$ chosen such that there exists $c_{21}>0$ with $\pr(\tilde{\mathcal{B}}_{21}^c) \leq c_{21} n^{-m}$; this is permitted by Theorem~\ref{thm:prediction} and its proof.

Let 
\[
\mathcal{B}_{22} := \left\{\abs{\frac{1}{n} \sum_{i=1}^n (\varepsilon_i \varepsilon_i^W)^2 - \Var(\varepsilon_1 \varepsilon^W_1)} \leq c n^{-1/2 + \epsilon}  \right\}.
\]
Now given $\epsilon > 0$, pick $k \in \mathbb{N} > m/\epsilon$. 
Note that by the argument of \eqref{eq:mom_bd},
$\E \{(\varepsilon_i \varepsilon_i^W)^{2} - \Var(\varepsilon_i \varepsilon_i^W)\}^{2k} \lesssim \E \varepsilon_i^{4k} + \E(\varepsilon_i^W)^{4k} \lesssim 1$. Thus by \citet[Ex.\ 2.20]{Wain19}  we have that there exists $c, c'>0$ such that $\pr(\mathcal{B}_{22}^c) \lesssim c'n^{-m}$.

Let $\mathcal{B}_{23} := \{\max_{i=1,\ldots,n} (\varepsilon_i^2 + (\varepsilon_i^W)^2) \leq  c \log n \}$ for constant  $c > 0$ such that $\pr(\mathcal{B}_{23}^c) \lesssim cn^{-m}$ for some constant $c  > 0$. Finally, let $\mathcal{B}_{24} := \{\frac{1}{\sqrt{n}} |\sum_{i = 1}^n \varepsilon_i \varepsilon_i^W | \leq c \sqrt{\log n}\} $. From Lemma~\ref{lem:subgbern}, we have that there exist $c, c'>0$ such that this occurs with probability at least $1 - c'n^{-m}$.

In the following we work on
\[
\mathcal{B}_2 := \mathcal{B}_1 \cap \bigcap_{j=1}^4 \mathcal{B}_{2j},
\]
where $\mathcal{B}_1$ is defined in Lemma~\ref{lem:nulldecomp}.
 
Now arguing as in \citet[D.1 of Supp.]{shah2020hardness}, we have that
	\begin{align*}
		|\hat{\varepsilon}_i^2 (\hat{\varepsilon}^W_i)^2 - \varepsilon_i^2 (\varepsilon^W_i)^2| &\lesssim
		 \underbrace{\hat{\Delta}_i^2 (\hat{\Delta}^W_i)^2}_{\mathrm{I}_i}
		+ \underbrace{|\varepsilon_i \varepsilon^W_i|  |\hat{\Delta}_i \hat{\Delta}^W_i|}_{\mathrm{II}_i}
		+  \underbrace{(\varepsilon^W_i)^2 \hat{\Delta}_i^2 + \varepsilon_i^2 (\hat{\Delta}^W_i)^2}_{\mathrm{III}_{1,i} + \mathrm{III}_{2,i}}
		+ \underbrace{|(\varepsilon^W_i)^2 \varepsilon_i \hat{\Delta}_i| + |\varepsilon_i^2 \varepsilon^W_i \hat{\Delta}^W_i|}_{\mathrm{IV}_{1,i} + \mathrm{IV}_{2,i}}.
	\end{align*}
To control the first term, note that
\begin{align*} 
	\frac{1}{n} \sum_{i=1}^n \mathrm{I}_i \leq \frac{1}{n} \left(\sum_{i=1}^n \hat{\Delta}_i^2 \right) \left( \sum_{i=1}^n (\hat{\Delta}^W_i)^2 \right) \lesssim n d_n^2.
\end{align*}
For the second term, we have by the Cauchy--Schwarz inequality, that
\begin{align*}
	\frac{1}{n} \sum_{i=1}^n \mathrm{II}_i &\leq  \left(\frac{1}{n} \sum_{i=1}^n (\varepsilon_i \varepsilon_i^W)^2\right)^{1/2} \left(\frac{1}{n}\sum_{i=1}^n \mathrm{I}_i \right)^{1/2} \lesssim \sqrt{n} d_n.
\end{align*}
For $\mathrm{III}_{1,i}$, we have
\[
\frac{1}{n}\sum_{i=1}^n \mathrm{III}_{1,i} \lesssim \frac{\log n}{n} \sum_{i=1}^n\hat{\Delta}_i^2 \lesssim d_n \log n.
\]
Turning to $\mathrm{IV}_{1,i}$, we have using the Cauchy--Schwarz inequality that
\begin{align*}
	\frac{1}{n} \sum_{i=1}^n \mathrm{IV}_{1,i}  \leq \left(\frac{1}{n} \sum_{i=1}^n (\varepsilon_i \varepsilon_i^W)^2\right)^{1/2} \left(\frac{1}{n}\sum_{i=1}^n \mathrm{III}_{1,i} \right)^{1/2}  \lesssim \sqrt{d_n \log n}.
\end{align*}
Identical bounds to the above hold for $\mathrm{III}_{2,i}$ and $\mathrm{IV}_{2,i}$.
Thus noting that $d_n \to 0$ and $d_n \gtrsim n/log n$, we have
\[
\abs{\frac{1}{n} \sum_{i = 1}^n \hat{\varepsilon}_i^2 (\hat{\varepsilon}_i^W)^2 - \frac{1}{n} \sum_{i = 1}^n \varepsilon_i^2 (\varepsilon_i^W)^2 } \lesssim \sqrt{n} d_n.
\]

Putting things together, we have that
\begin{align*}
&\left|\frac{1}{n} \sum_{i = 1}^n \hat{\varepsilon}_i^2 (\hat{\varepsilon}_i^W)^2 - \left(\frac{1}{n} \sum_{i = 1}^n \hat{\varepsilon}_i \hat{\varepsilon}_i^W\right)^2 - \var(\varepsilon_1 \varepsilon_1^W)\right| \\
 \leq & \abs{\frac{1}{n} \sum_{i = 1}^n \hat{\varepsilon}_i^2 (\hat{\varepsilon}_i^W)^2 - \frac{1}{n} \sum_{i = 1}^n \varepsilon_i^2 (\varepsilon_i^W)^2 } + \abs{\frac{1}{n} \sum_{i = 1}^n \varepsilon_i^2 (\varepsilon_i^W)^2 - \var(\varepsilon_1 \varepsilon_1^W) } + \left(\frac{1}{n} \sum_{i = 1}^n \hat{\varepsilon}_i \hat{\varepsilon}_i^W\right)^2 \\
\lesssim & \sqrt{n} d_n + n^{-1/2 + \epsilon} + d_n^2 + \frac{\log n}{n} \lesssim \sqrt{n} d_n + n^{-1/2 + \epsilon},
\end{align*}
using Lemma~\ref{lem:nulldecomp} in the final line. 

Moreover, with Assumptions~\ref{cd:pns}(ii)--(iv), we have that $s \log p / \sqrt{n} + p \log p / (r_{\lo} \sqrt{n}) + p^2 / (r_{\lo}^2 \sqrt{n}) \to 0$.
\end{proof}

Equipped with the lemmas, we are now ready to complete the proof of Theorem~\ref{thm:inference}.

\begin{proof}[Proof of Theorem~\ref{thm:inference}]
	We have the decomposition
	\begin{align*}
		T = & \frac{\frac{1}{\sqrt{n}} \sum_{i = 1}^n \varepsilon_i \varepsilon_i^W}{\var^{1/2}(\varepsilon_1 \varepsilon_1^W)} + \underset{=: \mathrm{I}}{\underbrace{\frac{\sqrt{n} \delta}{\sqrt{ \frac{1}{n} \sum_{i = 1}^n \hat{\varepsilon}_i^2 (\hat{\varepsilon}_i^W)^2 - \left(\frac{1}{n} \sum_{i = 1}^n \hat{\varepsilon}_i \hat{\varepsilon}_i^W\right)^2}}}} \\
		 & - \underset{=: \mathrm{II}}{\underbrace{\frac{\frac{1}{\sqrt{n}} \sum_{i = 1}^n \varepsilon_i \varepsilon_i^W \cdot \left(\sqrt{ \frac{1}{n} \sum_{i = 1}^n \hat{\varepsilon}_i^2 (\hat{\varepsilon}_i^W)^2 - \left(\frac{1}{n} \sum_{i = 1}^n \hat{\varepsilon}_i \hat{\varepsilon}_i^W\right)^2} -\var^{1/2}(\varepsilon_1 \varepsilon_1^W)\right)}{\sqrt{ \frac{1}{n} \sum_{i = 1}^n \hat{\varepsilon}_i^2 (\hat{\varepsilon}_i^W)^2 - \left(\frac{1}{n} \sum_{i = 1}^n \hat{\varepsilon}_i \hat{\varepsilon}_i^W\right)^2} \var^{1/2}(\varepsilon_1 \varepsilon_1^W)}}}.
	\end{align*}
Using tail bound in Lemma~\ref{lem:subgbern}, we have that with probability at least $1 - c'n^{-1}$, $\frac{1}{\sqrt{n}} |\sum_{i = 1}^n \varepsilon_i \varepsilon_i^W | \lesssim \sqrt{\log n}$. From this, and Lemmas~\ref{lem:nulldecomp} and \ref{lem:nullbnan} with $m = 1$, we have that there exists constants $c, c' > 0$ such that
\[
\pr\left(|\mathrm{I}| + |\mathrm{II}| \le c e_n\right) \ge 1 - \pr(\Omega^c(\tau, \kappa, c_p)) - c' n^{-1},
\]
where
\[
e_n := s\frac{\log p}{\sqrt{n}} + \frac{p \log p}{r_{\ell} \sqrt{n}} + \frac{p^2}{r_{\ell}^2 \sqrt{n}} + n^{-1/2 + \epsilon}.
\]

Now using the basic inequality that for any two events $\mathcal{A}, \mathcal{B}$, $\pr(\mathcal{A} \cap \mathcal{B}) \geq \pr(\mathcal{A}) - \pr(\mathcal{B}^c)$, we have that for any significance level $\alpha \in (0, 1)$, 
	\begin{align*}
		\pr \left(|T| \leq z_{1 - \alpha/2}\right) \geq &\; \pr\left(\left|\frac{\frac{1}{\sqrt{n}} \sum_{i = 1}^n \varepsilon_i \varepsilon_i^W}{\var^{1/2} (\varepsilon_1 \varepsilon_1^W)}\right| \leq z_{1 - \alpha/2} - c e_n, |\mathrm{I}| + |\mathrm{II}| \leq c e_n\right)\\
		\geq &\; \pr\left(\left|\frac{\frac{1}{\sqrt{n}} \sum_{i = 1}^n \varepsilon_i \varepsilon_i^W}{\var^{1/2} (\varepsilon_1 \varepsilon_1^W)}\right| \leq z_{1 - \alpha/2} - c e_n\right)  - \pr(|\mathrm{I}| + |\mathrm{II}| > c e_n) \\
		\geq &\; \pr\left(\left|\frac{\frac{1}{\sqrt{n}} \sum_{i = 1}^n \varepsilon_i \varepsilon_i^W}{\var^{1/2} (\varepsilon_1 \varepsilon_1^W)}\right| \leq z_{1 - \alpha/2} - c e_n\right) - \pr(\Omega^c(\tau, \kappa, c_p)) - c' n^{-1}.
	\end{align*}
	Using the Berry--Esssen bound \citep{Esseen42}, we have that for some universal constant $c''$,
	\[
	\pr\left(\left|\frac{\frac{1}{\sqrt{n}} \sum_{i = 1}^n \varepsilon_i \varepsilon_i^W}{\var^{1/2} (\varepsilon_i \varepsilon_i^W)}\right| \leq z_{1 - \alpha/2} - c e_n\right) \geq \pr\left(|\zeta| \leq z_{1 - \alpha/2} - c e_n\right) - \frac{c'' \E|\varepsilon_1 \varepsilon_1^W|^3}{\var^{3/2}(\varepsilon_1 \varepsilon_1^W)} \frac{1}{\sqrt{n}}
	\]
	where $\zeta$ denotes a standard Gaussian random variable. Then using that the standard Gaussian density is bounded by $1/\sqrt{2 \pi}$, we obtain the desired result.
\end{proof}

\section{Proof of Theorem~\ref{thm:resubg}}\label{sec:pfrsc}
Throughout the analysis we use the decomposition $X_i^\top \beta^0 = Z_i^\top \beta^0 + U_i^\top \Gamma \beta^0$.
Given $\tau > 0$,
define events
\begin{align*}
	A_i := \{|Z_i^\top \beta^0| \leq \tau/2\} \qquad \text{and}\qquad B_i := \{|U_i^\top (\Gamma \beta^0 + \delta^0)| \leq \tau/2\}.
\end{align*}
Let  $\mu_U := \E[U_i \given B_i]$, $\mu_Z := \E[Z_i \given A_i]$ and
\[
w := \frac{1}{n} \sum_{i=1}^n \one_{A_i \cap B_i}.
\]
Further set
\begin{equation}\label{eq:ziui}
	\tilde{Z}_i := (Z_i - \mu_Z) \one_{A_i}, \quad \tilde{U}_i := (U_i - \mu_U) \one_{B_i}.
\end{equation}
Note that then $\E \tilde{Z}_i = \E\tilde{U}_i = 0$. Let
\[
\tilde{\Sigma} = \E\left[\tilde{Z}_i \tilde{Z}_i^\top \one_{B_i}\right].
\]
Finally define $\mu := \mu_Z + \Gamma^\top \mu_U$ and let
\[
 \qquad \text{and} \qquad \bar{X} := \frac{1}{n} \sum_{i=1}^{n} (X_i - \mu) \one_{A_i \cap B_i}.
\]

The proof of Theorem~\ref{thm:resubg} makes use of the following lemma.

\begin{lemma}\label{lem:regaussian}
	Let $m \in \mathbb{N}$ be given. There exist constants $c, c', \tau > 0$ such that on an event $\mathcal{A}$ with probability at least $1-c'(n^{-m} + p^{-m})$,
 the following inequalities hold:
	\begin{align}
		\lambda_{\max} \left(\frac{1}{n} \sum_{i=1}^n \tilde{Z}_i \tilde{Z}_i^\top \one_{B_i}\right) &\leq c \frac{p}{n}, \label{eq:tildezeigen}\\
		\left\|\frac{1}{n} \sum_{i=1}^n \Gamma^\top \tilde{U}_i \tilde{U}_i^\top \Gamma \one_{A_i}\right\|_\infty &\leq c, \label{eq:sigu} \\
		\left\| \frac{1}{n} \sum_{i=1}^n \Gamma^\top \tilde{U}_i \tilde{Z}_i \right\|_\infty &\leq c \sqrt{\frac{\log p}{n}} ,\label{eq:sigtzu} \\
		\left\|\frac{1}{n} \sum_{i=1}^n \tilde{Z}_i \tilde{Z}_i^\top \one_{B_i} - \tilde{\Sigma} \right\|_\infty &\leq c \sqrt{\frac{\log p}{n}}, \label{eq:tildez} \\
		\|\bar{X}\|_\infty &\leq c \sqrt{\frac{\log p}{n}},  \label{eq:Xbar} \\
		|w| & \geq c^{-1} \label{eq:w}.
	\end{align}
	Moreover, $c^{-1} \leq \lambda_{\min}(\tilde{\Sigma}) \leq \lambda_{\max}(\tilde{\Sigma}) \leq c$.
\end{lemma}

\begin{proof}
	First note that by Lemma~\ref{lem:subg_trunc}, $\tilde{Z}_i \ind_{B_i}$ is a sub-Gaussian random vector with bounded variance proxy. Also $\tilde{V}_{ij} := (\Gamma^\top \tilde{U}_i)_j$ is sub-Gaussian with variance proxy depending only on that of $V_{ij} := (\Gamma^\top U_i)_j$. Note however that $X_{ij} = Z_{ij} + V_{ij}$ with $Z_{ij}$ and $V_{ij}$ independent, so $V_{ij}$ and hence also $\tilde{V}_{ij}$ have bounded variance proxies.
	
	The sub-Gaussianity of $\tilde{Z}_i \ind_{B_i}$ then yields \eqref{eq:tildezeigen} following the same argument used to derive \eqref{eq:subgx3}.
	
	 Next, similarly to~\eqref{eq:subgx1} it follows from Lemma~\ref{lem:subgbern} that with probability at least $1 - c' p^{-m}$,
\[
\max_{1 \leq j,k \leq p} \left|\frac{1}{n} \sum_{i = 1}^n \tilde{V}_{ij} \tilde{V}_{ik} - \E \tilde{V}_{ij} \tilde{V}_{ik}\right| \leq c \sqrt{\frac{\log p}{n}}.
\]
Thus \eqref{eq:sigu} follows from noting that by the Cauchy--Schwarz inequality,
\[
|\E \tilde{V}_{ij} \tilde{V}_{ik}| \leq (\E\tilde{V}_{ij}^2 \E\tilde{V}_{ik}^2)^{1/2} \leq (\E X_{ij}^2 \E X_{ik}^2)^{1/2} \leq \sigma_x^2.
\]
Inequalities \eqref{eq:sigtzu} and \eqref{eq:tildez} follow similarly.

Turning to \eqref{eq:Xbar}, observe that
\begin{align*}
	\E \{(X_i - \mu)\ind_{A_i \cap B_i}\} &= \E\left\{\left( Z_i + \Gamma^\top U_i - \frac{\E(Z_i \ind_{A_i}) }{\pr(A_i)} + \Gamma^\top \frac{\E(U_i \ind_{B_i})}{\pr(B_i)}\right)\ind_{A_i \cap B_i} \right\}.
\end{align*}
Then as $U_i \independent Z_i$, $A_i \independent B_i$ whence $\E \ind_{A_i \cap B_i} = \pr(A_i) \pr(B_i)$. We therefore see that the last display is $0$. Then from Lemma~\ref{lem:subg_trunc}, we have that $(X_i - \mu)\ind_{A_i \cap B_i}$ is a mean-zero sub-Gaussian random variable with bounded variance proxy, and so we obtain \eqref{eq:Xbar}.

For all $\tau$ sufficiently large, we have $\pr(A_i) \pr(B_i) >2 c^{-1}$ for some $c>0$, and so Hoeffding's inequality yields \eqref{eq:w}.

Lemma~\ref{lem:subg_trunc} gives the final claim.
\end{proof}

\subsection{Proof of Theorem~\ref{thm:resubg}}
\begin{proof}
	Let $\Delta_1 \in C(S)$ and $\Delta_2 \in \R^p$
and let $\mu := \mu_Z + \Gamma^\top \mu_U$. In the following we work on the event $\mathcal{A}$ given in Lemma~\ref{lem:regaussian}.

	Recalling the decomposition $X_i^\top \beta^0 + U_i^\top \delta^0 = Z_i^\top \beta^0 + U_i^\top (\Gamma \beta^0 + \delta^0)$, we have that
	\begin{align*}
	\mathrm{I} := & (\Delta_1 + \Delta_2)^\top \frac{1}{n} \sum_{i=1}^{n} X_i X_i^\top \one_{\{|X_i^\top \beta^0 + U_i^\top \delta^0 | \leq \tau\}} (\Delta_1 + \Delta_2)\\
	\geq & (\Delta_1 + \Delta_2)^\top \frac{1}{n} \sum_{i=1}^{n} X_i X_i^\top \one_{A_i \cap B_i} (\Delta_1 + \Delta_2) \\
	= & (\Delta_1 + \Delta_2)^\top \frac{1}{n} \sum_{i=1}^{n} (X_i - \mu) (X_i - \mu)^\top \one_{A_i \cap B_i} (\Delta_1 + \Delta_2) \\
	& + 2 (\Delta_1 + \Delta_2)^\top \frac{1}{n} \sum_{i=1}^{n} (X_i - \mu) \mu^\top \one_{A_i \cap B_i} (\Delta_1 + \Delta_2) \\
	& + (\Delta_1 + \Delta_2)^\top \frac{1}{n} \sum_{i=1}^{n} \mu \mu^\top \one_{A_i \cap B_i} (\Delta_1 + \Delta_2).
	\end{align*}
Then
\begin{align*}
	\mathrm{I} \geq & (\Delta_1 + \Delta_2)^\top \frac{1}{n} \sum_{i=1}^{n} (X_i - \mu) (X_i - \mu)^\top \one_{A_i \cap B_i} (\Delta_1 + \Delta_2) \\
	& \qquad + 2 (\Delta_1 + \Delta_2)^\top \bar{X} \mu^\top (\Delta_1 + \Delta_2)\\
	& \qquad + w \{(\Delta_1 + \Delta_2)^\top \mu\}^2.
\end{align*}
Note that for all $\alpha > 0$, $a, b \in \R$,
\begin{equation} \label{eq:elem}
	\frac{1}{\alpha}a^2 + 2ab + \alpha b^2 = \left(\frac{a}{\sqrt{\alpha}} + \sqrt{\alpha} b \right)^2 \geq 0.
\end{equation}
Applying this with $\alpha = w$, $\bar{X}^\top (\Delta_1 + \Delta_2) = a$ and $\mu^\top (\Delta_1 + \Delta_2) = b$, we have
that
\begin{align*}
	2 (\Delta_1 + \Delta_2)^\top \bar{X} \mu^\top (\Delta_1 + \Delta_2) + w \{(\Delta_1 + \Delta_2)^\top \mu\}^2  
	\geq - \frac{1}{w} \{(\Delta_1 + \Delta_2)^\top \bar{X}\}^2.
\end{align*}
Thus we have
\begin{align*}
	\mathrm{I} &\geq  (\Delta_1 + \Delta_2)^\top \frac{1}{n} \sum_{i=1}^{n} (X_i - \mu) (X_i - \mu)^\top \one_{A_i \cap B_i} (\Delta_1 + \Delta_2)
	 - \frac{1}{w} \{(\Delta_1 + \Delta_2)^\top \bar{X}\}^2 \\
	 &=: \mathrm{I}_1 - \mathrm{I}_2.
\end{align*}
Then recalling the definition in~\eqref{eq:ziui}, we have that
\begin{align*}
	\mathrm{I}_1 = \;& (\Delta_1 + \Delta_2)^\top \frac{1}{n} \sum_{i=1}^{n} \tilde{Z}_i \tilde{Z}_i^\top \one_{B_i} (\Delta_1 + \Delta_2) \\
	& \quad + (\Delta_1 + \Delta_2)^\top \frac{1}{n} \sum_{i=1}^{n} \left(\Gamma^\top \tilde{U}_i \tilde{Z}_i^\top + \tilde{Z}_i \tilde{U}_i^\top \Gamma + \Gamma^\top \tilde{U}_i \tilde{U}_i^\top \Gamma \one_{A_i}\right) (\Delta_1 + \Delta_2) \\
	= \;& \frac{1}{n} \sum_{i=1}^n \Bigg\{\left(\Delta_2^\top \Gamma^\top \tilde{U}_i \tilde{U}_i^\top \Gamma \Delta_2 + 2 \Delta_2^\top \Gamma^\top \tilde{U}_i \tilde{U}_i^\top \Gamma \Delta_1 + \Delta_1^\top \Gamma^\top \tilde{U}_i \tilde{U}_i^\top \Gamma \Delta_1\right) \one_{A_i} \\
	& + 2\Delta_2^\top  \Gamma^\top \tilde{U}_i \tilde{Z}_i^\top \Delta_2 + \Delta_2^\top \tilde{Z}_i \tilde{Z}_i^\top \Delta_2 \one_{B_i}\\
	& + \Delta_1^\top \tilde{Z}_i \tilde{Z}_i^\top \Delta_1 \one_{B_i} + 2\Delta_1^\top \Gamma^\top \tilde{U}_i \tilde{Z}_i^\top \Delta_1 + 2 \Delta_2^\top \tilde{Z}_i \tilde{Z}_i^\top \Delta_1 \one_{B_i} + 2 \Delta_2^\top (\Gamma^\top \tilde{U}_i \tilde{Z}_i^\top + \tilde{Z}_i \tilde{U}_i^\top \Gamma) \Delta_1 \Bigg\}. \nonumber
\end{align*}
Now note that from \eqref{eq:elem}, we have that for $a_1, a_2, b_2 \in \R$ and for all $t \in (0,1)$,
\begin{align*}
a_2^2 + 2a_2 a_1 + a_1^2 + 2a_2 b_2 + b_2^2 &= \{ta_2^2 + 2a_2 a_1 + a_1^2\} + \{(1-t)a_2^2 + 2a_2 b_2 + b_2^2\} \\
&\geq \left(1 - \frac{1}{t}\right)a_1^2 + \left(1 - \frac{1}{1-t}\right)b_2^2.
\end{align*}
Applying this with $a_j = \tilde{U}_i^\top \Gamma \Delta_j \one_{A_i}$ and $b_2 = \tilde{Z}_i^\top \Delta_2 \one_{B_i}$, we have that for any $t \in [0,1]$,
\begin{align*}
	\mathrm{I}_1 \geq \;& \Delta_1^\top \frac{1}{n} \sum_{i=1}^n \tilde{Z}_i \tilde{Z}_i^\top \one_{B_i} \Delta_1 - \frac{1 - t}{t} \Delta_1^\top \frac{1}{n} \sum_{i=1}^n \Gamma^\top \tilde{U}_i \tilde{U}_i^\top \Gamma  \one_{A_i}\Delta_1 \\
	& - \frac{t}{1 - t} \Delta_2^\top \frac{1}{n} \sum_{i=1}^n \tilde{Z}_i \tilde{Z}_i^\top \one_{B_i} \Delta_2  + 2 \Delta_2^\top \frac{1}{n} \sum_{i=1}^n \tilde{Z}_i \tilde{Z}_i^\top \one_{B_i} \Delta_1 \\ 
	& + 2 \Delta_1^\top \frac{1}{n} \sum_{i=1}^n \Gamma^\top \tilde{U}_i \tilde{Z}_i^\top \Delta_1 + 2 \Delta_2^\top \frac{1}{n} \sum_{i=1}^n (\Gamma^\top \tilde{U}_i \tilde{Z}_i^\top + \tilde{Z}_i \tilde{U}_i^\top \Gamma) \Delta_1.
\end{align*}
Now using the fact that we are on $\mathcal{A}$, we have for example by H\"older's inequality that
\begin{align*}
\Delta_2^\top \frac{1}{n} \sum_{i=1}^n \tilde{Z}_i \tilde{Z}_i^\top \one_{B_i} \Delta_1 &= \Delta_2^\top \left(\frac{1}{n} \sum_{i=1}^n \tilde{Z}_i \tilde{Z}_i^\top \one_{B_i} - \tilde{\Sigma}\right) \Delta_1 + \Delta_2^\top \tilde{\Sigma} \Delta_1 \\
&\lesssim \sqrt{\frac{\log p}{n}} \|\Delta_1\|_1 \|\Delta_2\|_1 + \|\Delta_1\|_2 \|\Delta_2\|_2.
\end{align*}
Using such arguments, we see that there exists a constant $c_1 > 0$ such that
\begin{align*}
	\mathrm{I}_1 \geq & \Delta_1^\top \tilde{\Sigma} \Delta_1 -  c_1 \left(\frac{1 - t}{t} \|\Delta_1\|_1^2 +  \frac{t}{1-t} \frac{p}{n} \|\Delta_2\|_2^2 +  \sqrt{\frac{\log p}{n}} \|\Delta_1\|_1 \|\Delta_2\|_1 + \sqrt{\frac{\log p}{n}} \|\Delta_1\|_1^2 + \|\Delta_1\|_2\|\Delta_2\|_2\right) \\
	\gtrsim & \|\Delta_1\|_2^2 - c_1 \left(\frac{1 - t}{t} s \|\Delta_1\|_2^2 +  \frac{t}{1-t} \frac{p}{n} \|\Delta_2\|_2^2 +  \left(\sqrt{s p\frac{\log p}{n}} + 1\right) \|\Delta_1\|_2 \|\Delta_2\|_2 + s \sqrt{\frac{\log p}{n}} \|\Delta_1\|_2^2 \right) \\
\end{align*}
where for the last inequality, we use that $\Delta_1$ is within the cone $C(S)$ and so $\|\Delta_1\|_1 \leq 4\sqrt{s} \|\Delta_1\|_2$, see \eqref{eq:sparse_ineq}.

Note that by \eqref{eq:elem},
\[
\left(\sqrt{s p\frac{\log p}{n}} + 1\right) \|\Delta_1\|_2 \|\Delta_2\|_2 \leq \frac{1}{3c_1} \|\Delta_1\|_2^2 + c_2\left(s p\frac{\log p}{n} + 1\right)  \|\Delta_2\|_2^2,
\]
for some constant $c_2 > 0$.
Now, let $t$ be such that $(1-t)/t = 1 /(3c_1 s)$. Then by Condition~\ref{cd:omesparsity}, we have that for some constant $c_3 >0$ and for all $n$ sufficiently large, 
\[
\mathrm{I}_1 \gtrsim \|\Delta_1\|_2^2
- c_3 r_{\ell} \sqrt{\frac{\log p}{n}} \|\Delta_2\|_2^2.
\]
Turning to $\mathrm{I}_2$, as we are on $\mathcal{A}$, we have
\begin{align*}
\mathrm{I}_2 &\lesssim \|\Delta_1 + \Delta_2\|_1^2 \|\bar{X}\|_2^2 \\
&\lesssim \frac{\log p}{n}  (\|\Delta_1\|_1^2 + \|\Delta_2\|_1^2) \\
&\lesssim \frac{\log p}{n}  (s\|\Delta_1\|_2^2 + p\|\Delta_2\|_2^2),
\end{align*}
using the Cauchy--Schwarz inequality and the fact that $\Delta_1 \in C(S)$ in the final line. Using Condition~\ref{cd:omesparsity} and putting things together, we see that for some constant $c_4 > 0$,
\[
\mathrm{I} = \mathrm{I}_1 - \mathrm{I}_2 \gtrsim \|\Delta_1\|_2^2
- c_4 r_{\ell} \sqrt{\frac{\log p}{n}} \|\Delta_2\|_2^2,
\]
as required.
\end{proof}

\section{Some basic results concerning sub-Gaussian random variables} \label{sec:subG}
\begin{lemma}\label{lem:subgbern}
	Let $(V_1, W_1), \ldots, (V_n, W_n) \in \R^2$ be a sequence of mean-zero independent random vectors that are each sub-Gaussian with variance proxy $\sigma^2$. 
	Then there exist constants $\sigma_g, b > 0$ such that for all $t \geq 0$,
	\[
	\pr\left(\left|\frac{1}{n} \sum_{i = 1}^{n} \{V_i W_i - \E(V_i W_i)\} \right| \geq t\right) \leq 2\exp\left( -\frac{nt^2}{2 (\sigma_g^2 + bt)}\right).
	\]
\end{lemma}
\begin{proof}
	This follows from the facts that products of sub-Gaussian random variables are sub-exponential random variables, and sub-exponential random variables satisfy Bernstein's condition and hence Bernstein's inequality applies; see for example \citet[Chap.~3]{Wain19}.
\end{proof}

\begin{lemma} \label{lem:subg_trunc}
	Let $V$ be a mean-zero sub-Gaussian random vector with variance proxy $\sigma^2_V$ and suppose $\lambda_{\min}(\Var(V)) \geq c > 0$. Let $\{A_\tau\}_{\tau \geq 0}$ be a family of events, potentially depending on $V$, such that $\pr(A_\tau) \to 1$ as $\tau \to \infty$. Then for all $\tau$, defining
	$W := V \ind_{A_\tau}$ and $\mu_W := \E(V \ind_{A_\tau})$, $W-\mu_W$ is a sub-Gaussian vector with variance proxy depending only on $\sigma_V^2$. Moreover, for all $\rho <1$, there exists $\tau$ depending only on $\sigma^2_V$ such that $\lambda_{\min}(\Var(W)) \geq \rho c$.
\end{lemma}
\begin{proof}
	Fix unit vector $u$. We have, writing $\tilde{W} := u^\top W$, $\tilde{V} := u^\top V$ and $\mu_{\tilde{W}} := \E \tilde{W}$, we have
	\begin{equation} \label{eq:mom_bd}
		\begin{split}
					\E (\tilde{W}-\mu_{\tilde{W}})^{2k} &= 2^{2k} \E (\tilde{W}/2-\mu_{\tilde{W}}/2)^{2k} \\
			&\leq 2^{2k} \E (|\tilde{W}|/2 + |\mu_{\tilde{W}}|/2)^{2k} \\
			&\leq 2^{2k-1} \{\E \tilde{W}^{2k} + (\E \tilde{W})^{2k}\} \\
			&\leq 2^{2k}\E \tilde{W}^{2k} \leq 2^{2k} \E \tilde{V}^{2k},
		\end{split}
	\end{equation}
	using Jensen's inequality twice. Thus by \citet[Thm.~2.6]{Wain19}, the first claim is proved.
	
	Next observe that
	\begin{align*}
		\var(\tilde{W}) & = \E[\tilde{V}^2 \one_{A_\tau}] - (\E[ \tilde{V} \one_{A_\tau}])^2 \\
		& = \Var(\tilde{V}) - \E[(\tilde{V})^2 \one_{A_\tau^c}]  - (\E[\tilde{V} \one_{A_\tau^c}])^2 \\
		& \geq \Var(\tilde{V}) - \sqrt{\E[\tilde{V}^4] \pr(A_\tau^c) } - \Var(\tilde{V})\pr(A_\tau^c)
	\end{align*}
	where for the second equality, we used the fact that $\E \tilde{V} = 0$, and with the final line following from the Cauchy--Schwarz inequality.
	Now $\E[\tilde{V}^4] \lesssim 1$ (where the implicit constant is independent of $u$), so putting things together, we then see that for a constant $\tau$ sufficiently large (such that $\pr(A_\tau^c)$ is sufficiently small), for all unit vectors $u$,  $\var(u^\top W) \geq  \rho \Var(u^\top V)$. Taking infima over $u$, we obtain the final claim.
\end{proof}

\section{Discussion of Theorem~\ref{thm:resubg}} \label{sec:bound_comments}
It is interesting to compare Theorem~\ref{thm:resubg} with the bound on the restricted eigenvalue of \citet[Prop.~1]{GCB22}, which studies estimation and inference in a linear regression setting with confounding. To facilitate this comparison, observe that writing $\tilde{\bX} \in \R^{n \times p}$ for the matrix with $i$th row $\tilde{X}_i := X_i \one_{\{|X_i^\top \beta^0 + U_i^\top \delta^0 | \leq \tau \}}$, we have that $\Omega(\tau, \kappa, c_p)$ occurring is equivalent to having that for any $\Delta \in C(S)$,
\[
\frac{\|Q^{\frac{1}{2}} \tilde{\bX} \Delta \|_2^2}{n} \ge \kappa \|\Delta \|_2^2,
\]
where $Q := \bs{I} - \tilde{\bX} \left(c_p r_\lo \sqrt{n \log p} I + \tilde{\bX}^\top \tilde{\bX}\right)^{-1} \tilde{\bX}^\top$. The restricted eigenvalue in \citet{GCB22} takes this form but with $\mb X$ in place of $\tilde{\mb X}$ and $Q$ given by a spectral transformation referred to as the trim transform in \citet{CBM20}. Compared to \citet[Prop.~1]{GCB22}, our result has a stronger sparsity requirement. On the other hand, \citet[Prop.~1]{GCB22} requires  that at least a constant proportion of the non-zero eigenvalues of the matrix $\bX^\top \bX / n$ are of order at least $p / n \vee 1$. 

 To further rationalise this assumption, in \citet[Appendix~B, Lemma~7]{GCB22}, it is shown that a sufficient condition for this is that all components in the random vector $\Cov^{1/2}(X) \cdot X$ are independent sub-Gaussian random variables. This independence would be particularly implausible in our context as our analogue of $X_i$, given by $\tilde{X}_i$, is constructed through a truncation that involves the whole vector $X_i$.

\section{Additional simulation results} \label{sec:simu}

\subsection{Estimation of the number of factors} \label{sec:fac_num}
In this section, we present additional simulation results concerning the factor-based approach of \citet{ouyang2023high} in the settings described in Section~\ref{sec:estimation}. Specifically, Figures~\ref{fig:logitoepfactor1} and \ref{fig:logiexpfactor1} show the number of factors selected by the approach in the various settings. In particular, we see how when the confounding strength $\nu$ is low, accurate estimation of the number of factors is challenging.

\begin{figure}[t!]
	\subfigure[Toeplitz, $s = 5, q = 5$]{\includegraphics[width=0.33\textwidth]{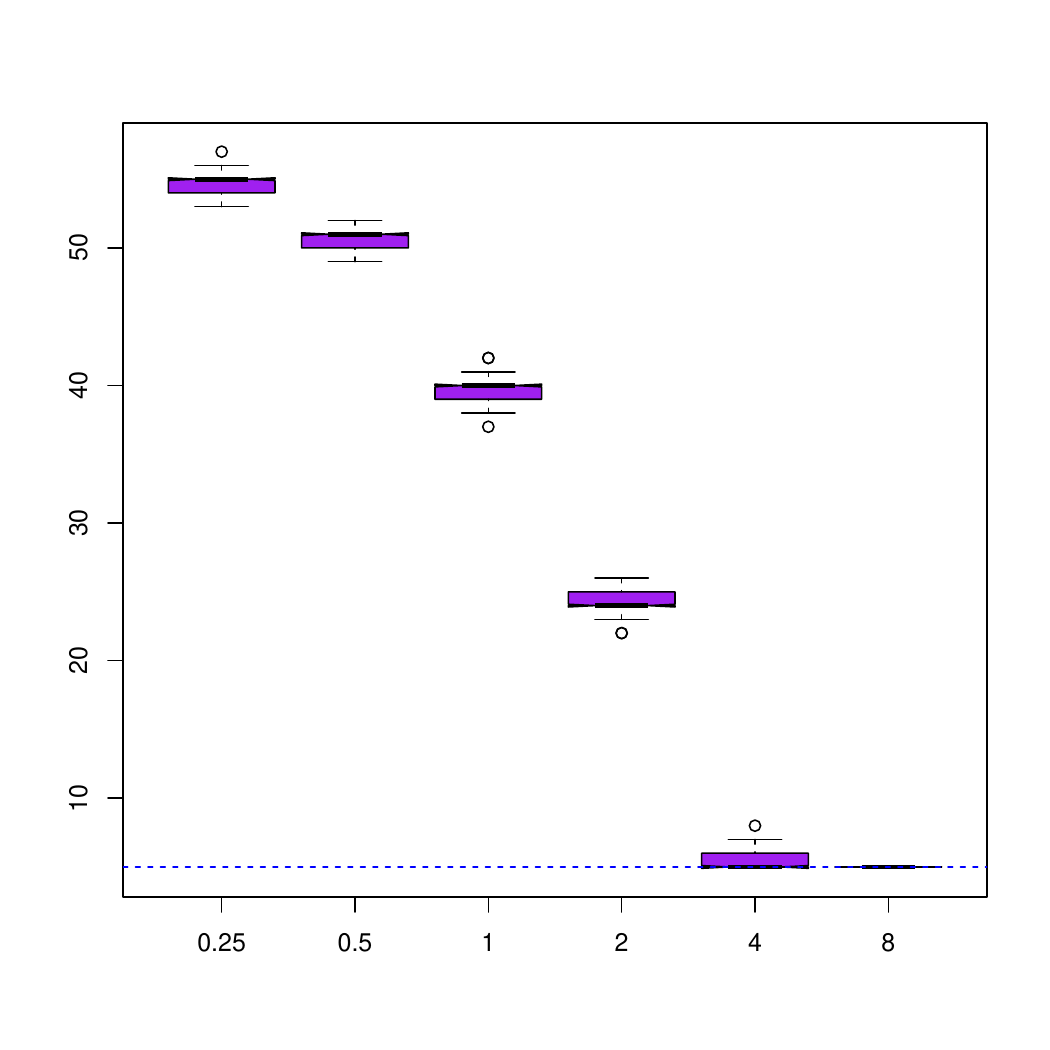}}
	\subfigure[Toeplitz, $s = 5, q = 10$]{\includegraphics[width=0.33\textwidth]{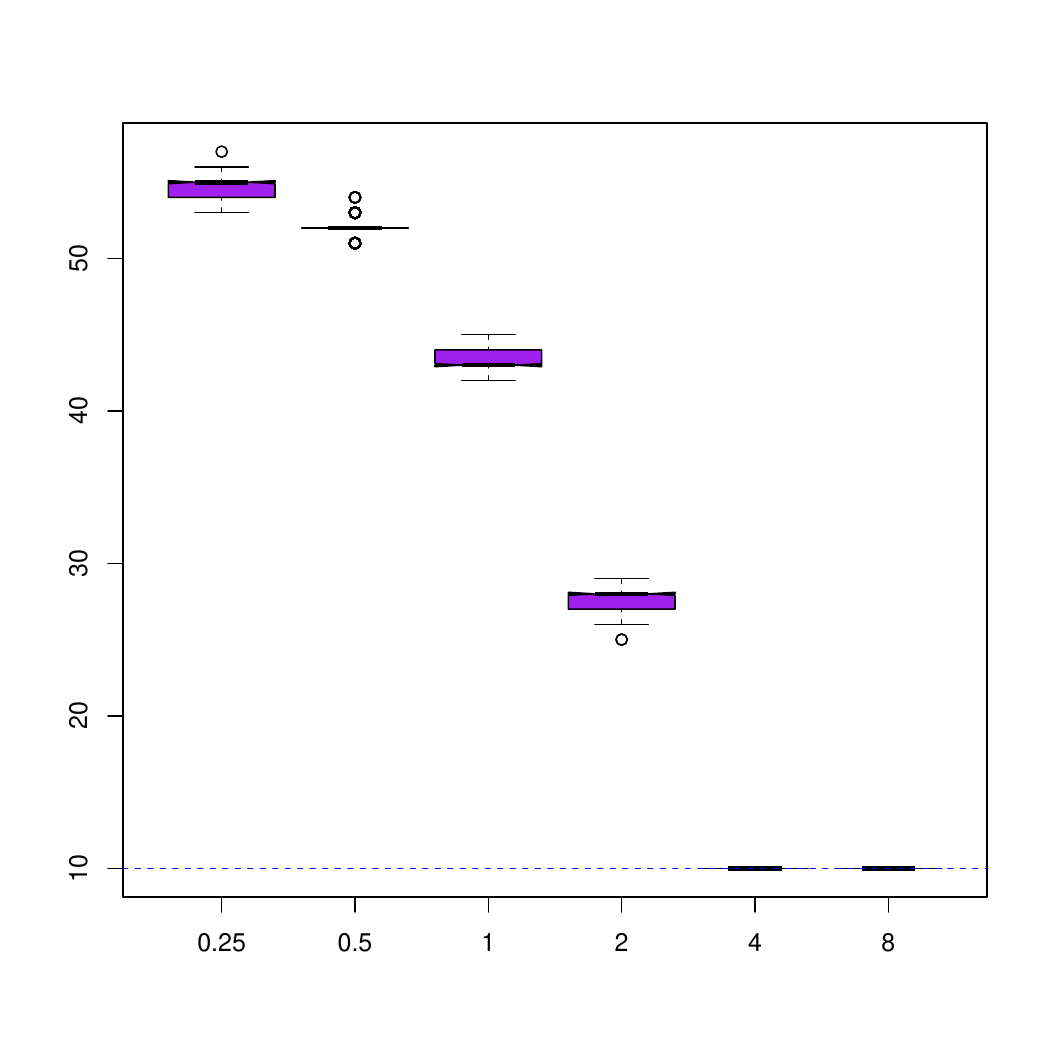}}
	\subfigure[Toeplitz, $s = 5, q = 20$]{\includegraphics[width=0.33\textwidth]{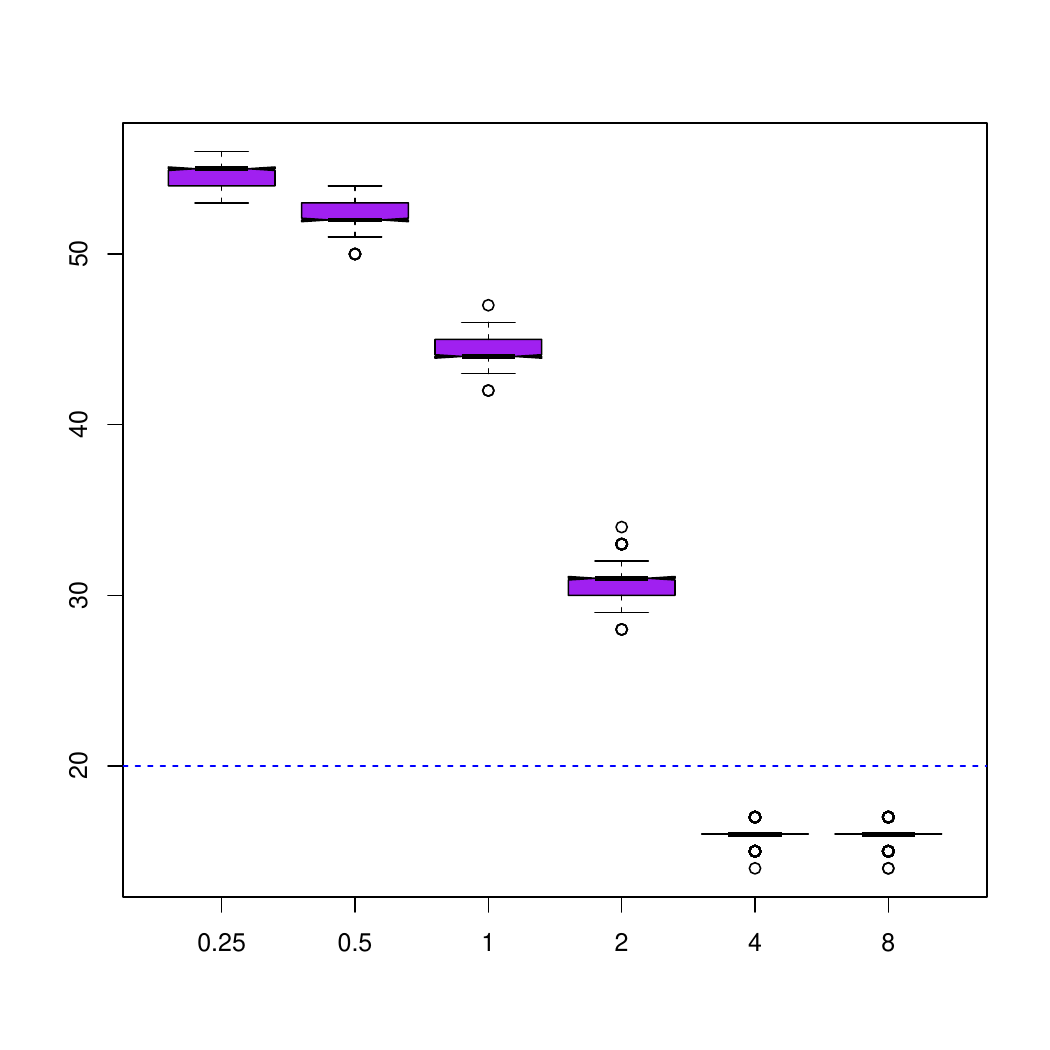}}
 \subfigure[Toeplitz, $s = 10, q = 5$]{\includegraphics[width=0.33\textwidth]{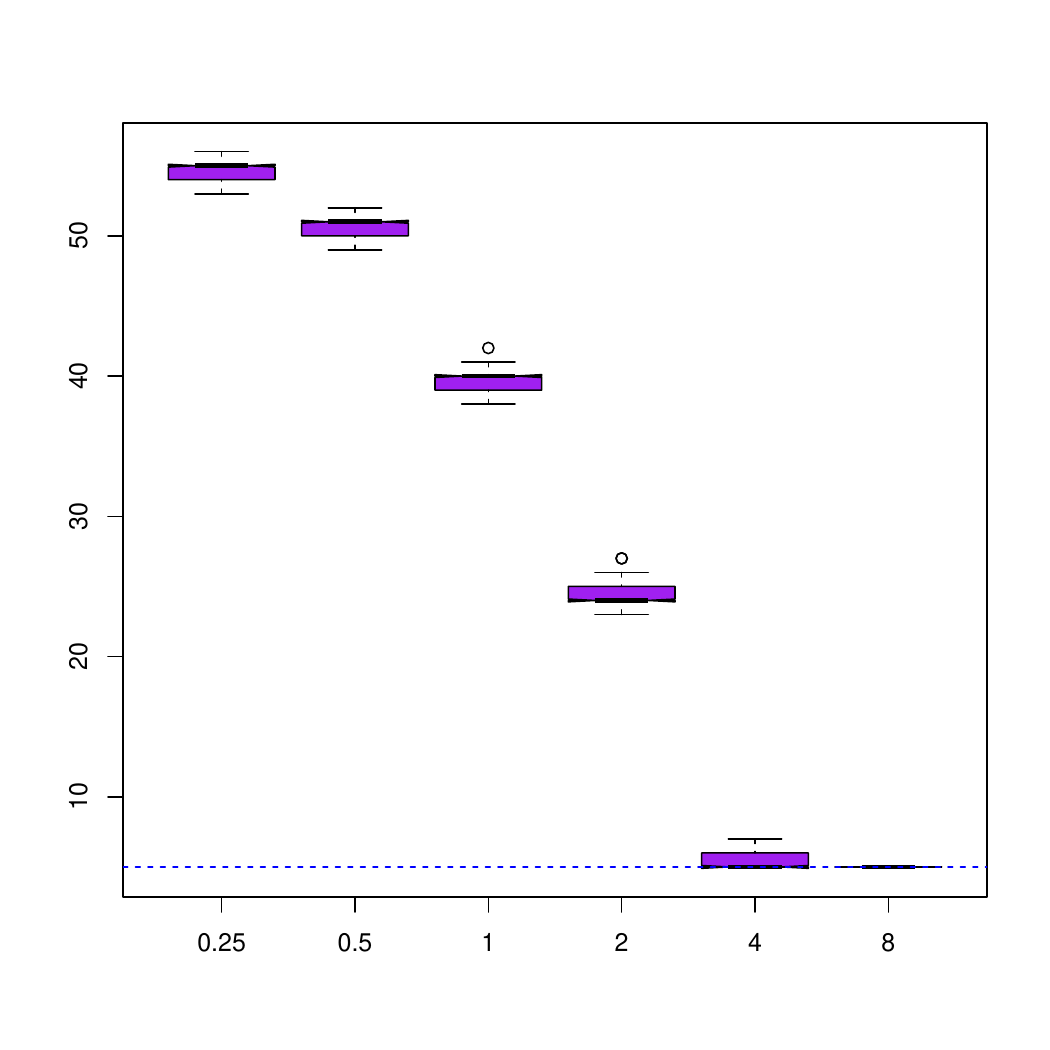}}
	\subfigure[Toeplitz, $s = 10, q = 10$]{\includegraphics[width=0.33\textwidth]{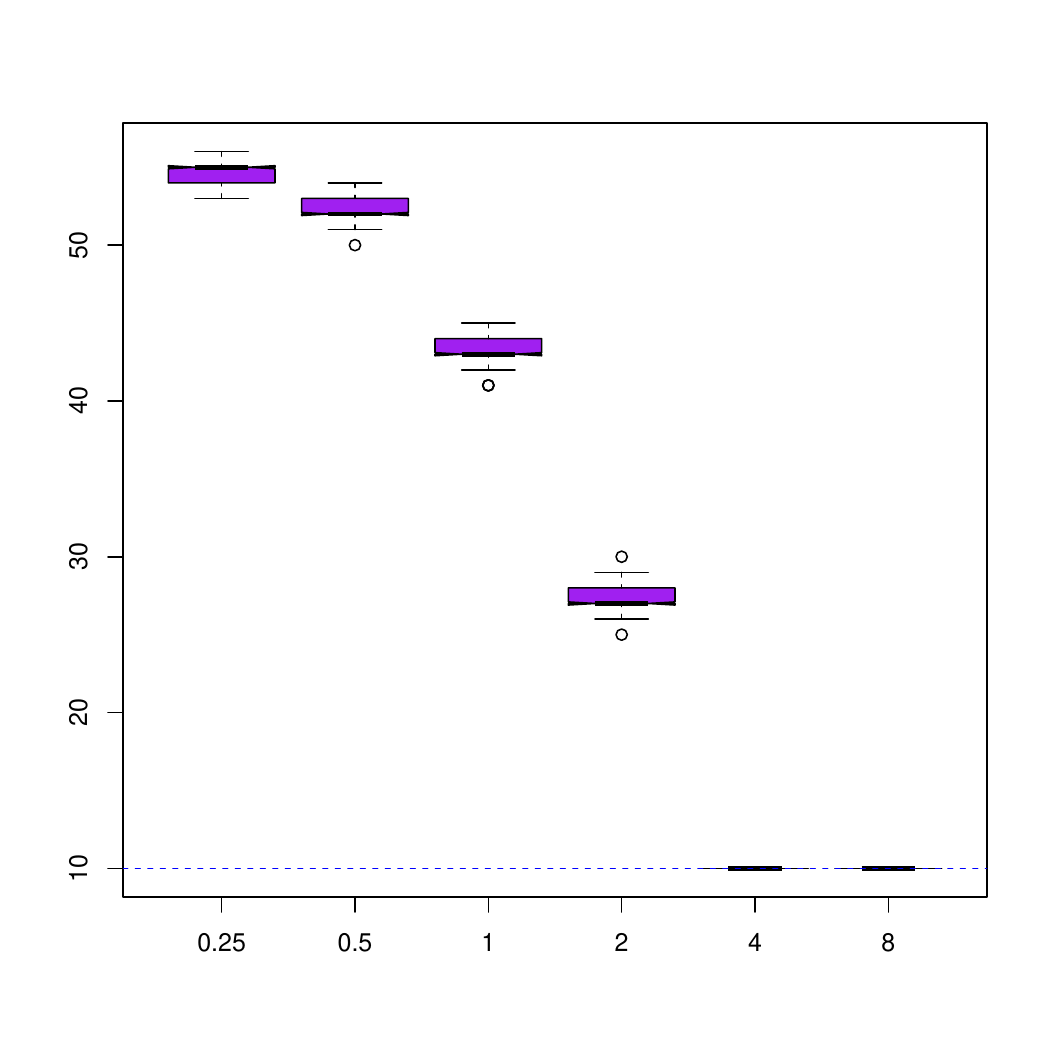}}
	\subfigure[Toeplitz, $s = 10, q = 20$]{\includegraphics[width=0.33\textwidth]{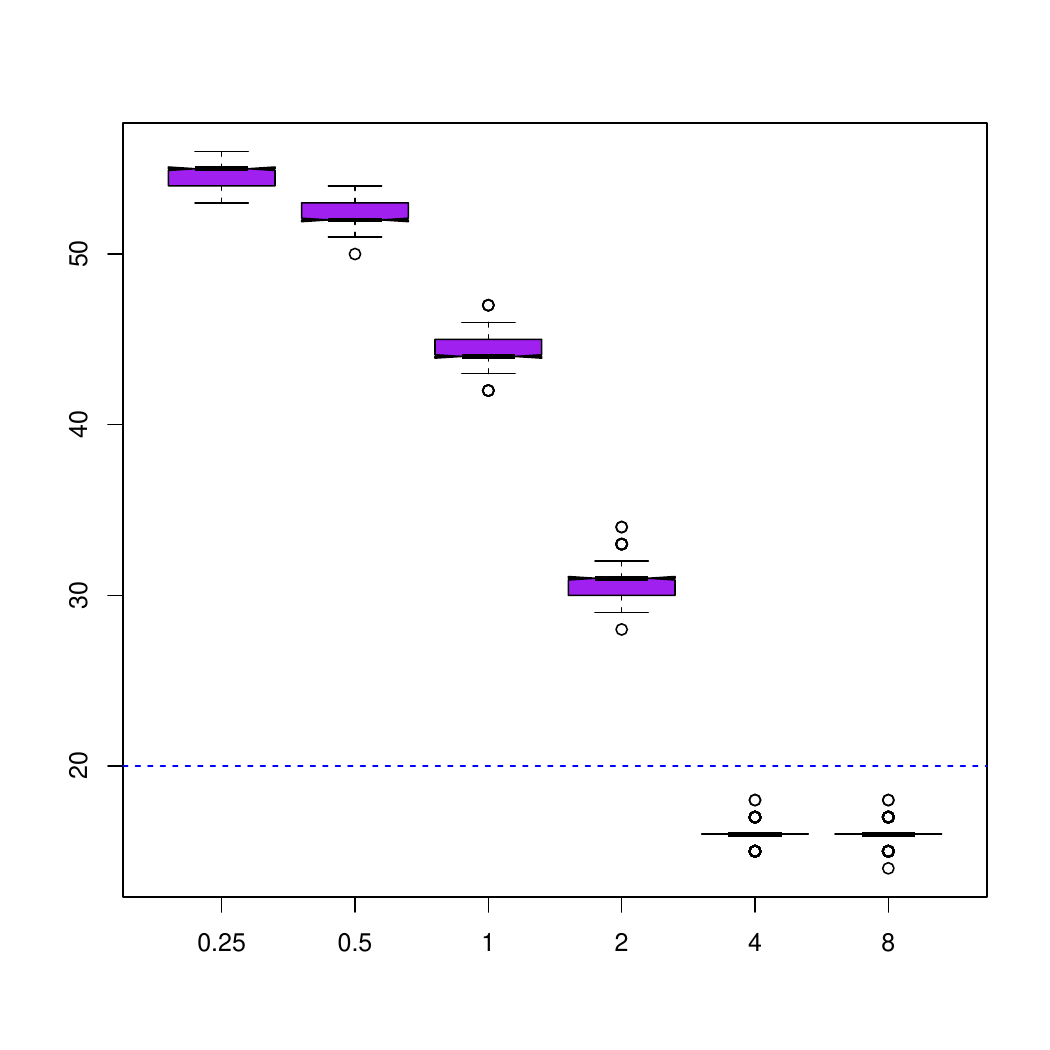}}
	\caption{Numbers of factors selected by the factor-based method in the Toeplitz simulation examples for the different confounding strengths $\nu$ on the $x$-axis.}\label{fig:logitoepfactor1}
\end{figure}


\begin{figure}[t!]
	\subfigure[Exp. decaying, $s = 5, q = 5$]{\includegraphics[width=0.33\textwidth]{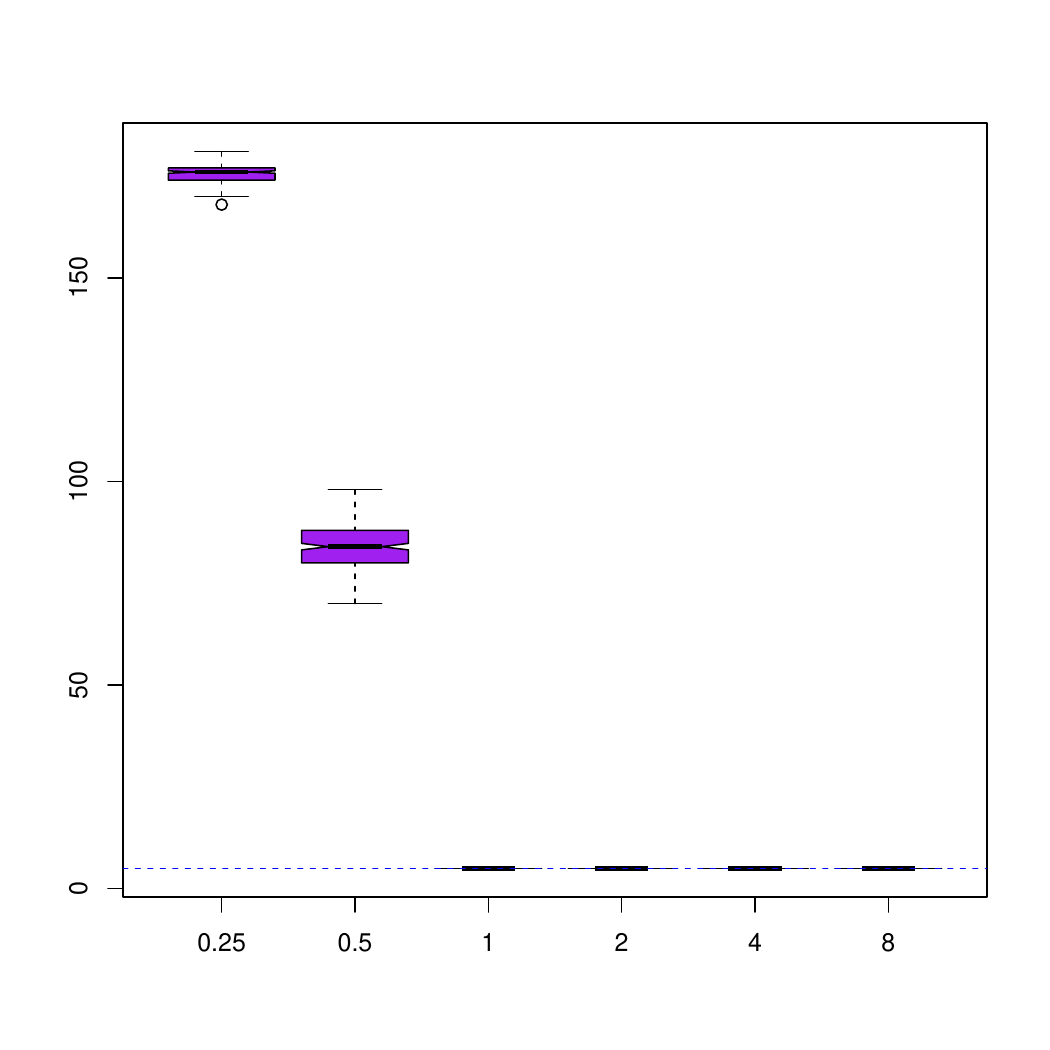}}
 \subfigure[Exp. decaying, $s = 5, q = 10$]{\includegraphics[width=0.33\textwidth]{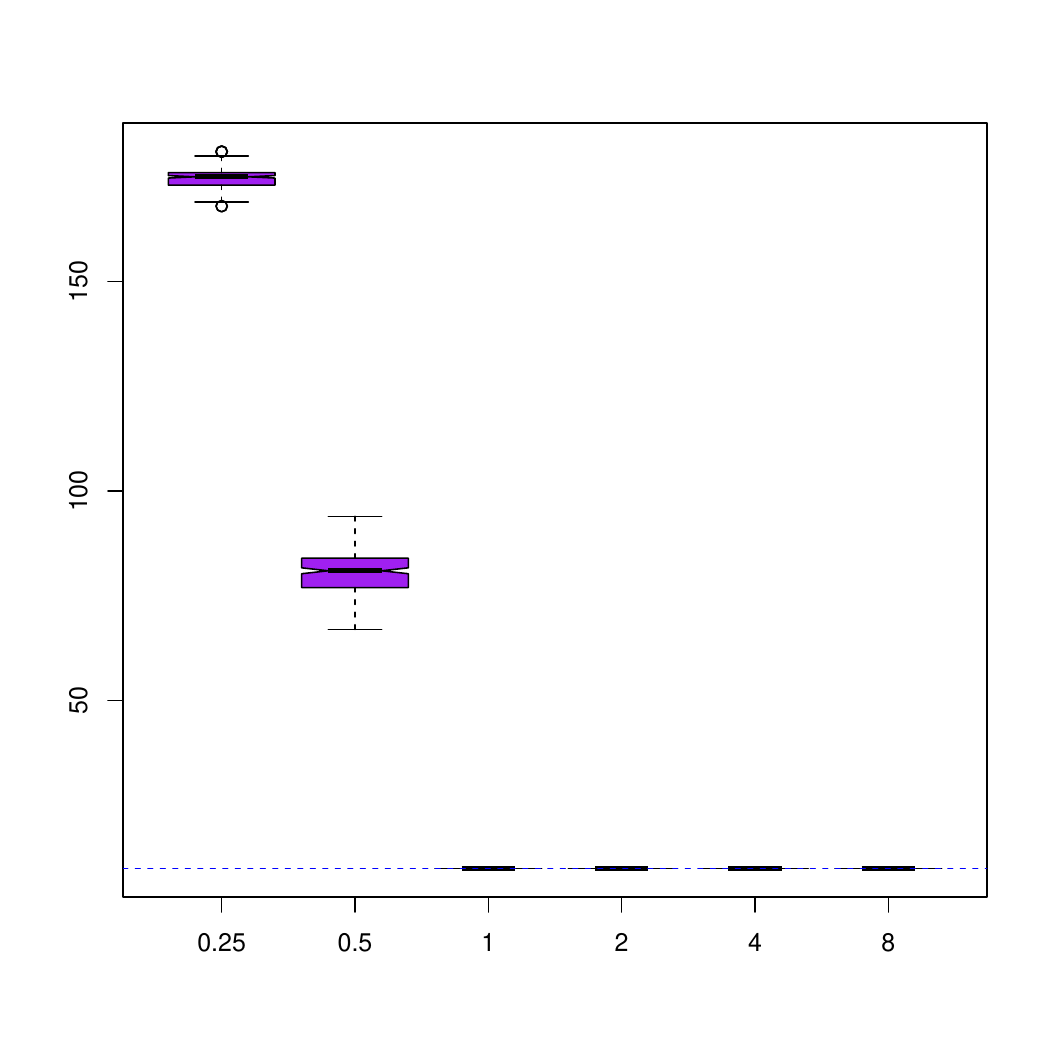}}
	\subfigure[Exp. decaying, $s = 5, q = 20$]{\includegraphics[width=0.33\textwidth]{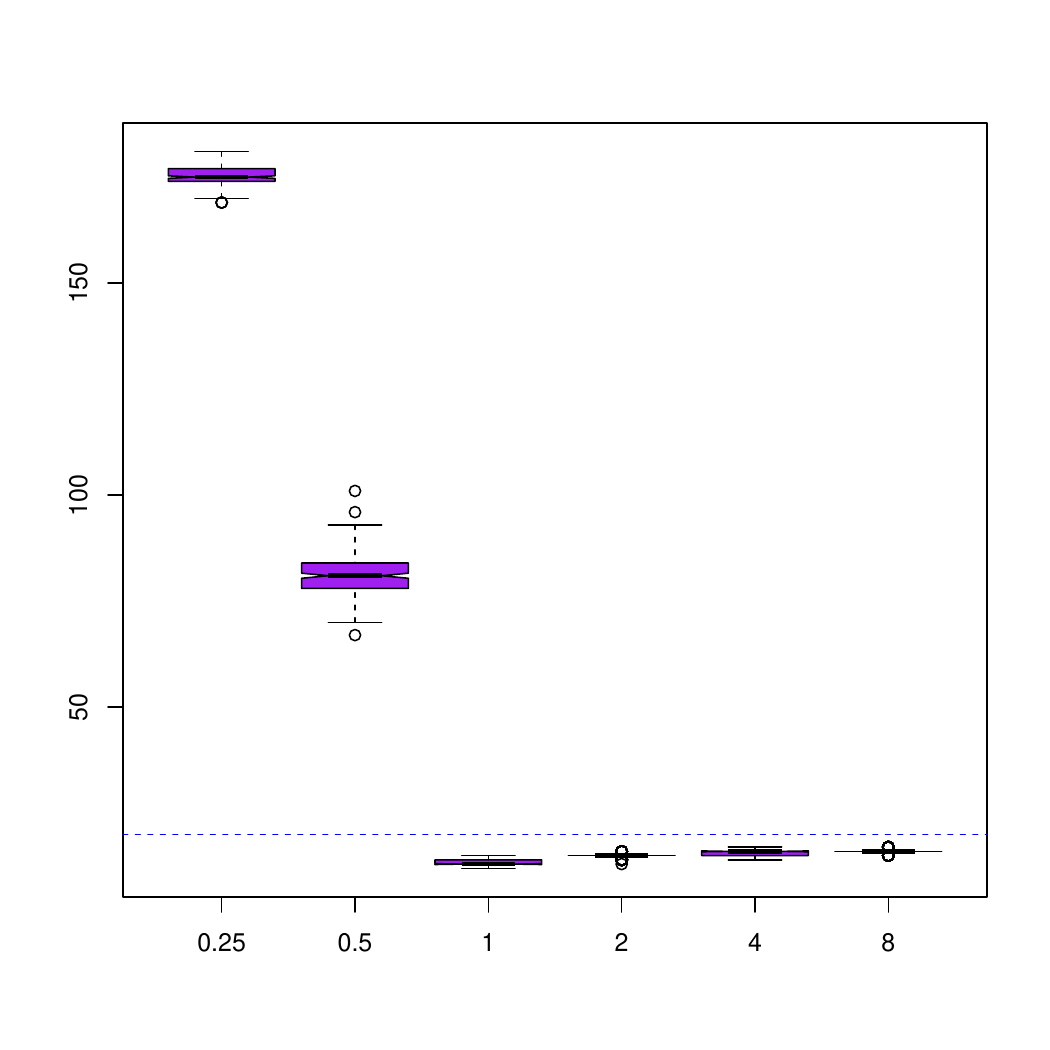}}
	\subfigure[Exp. decaying, $s = 10, q = 5$]{\includegraphics[width=0.33\textwidth]{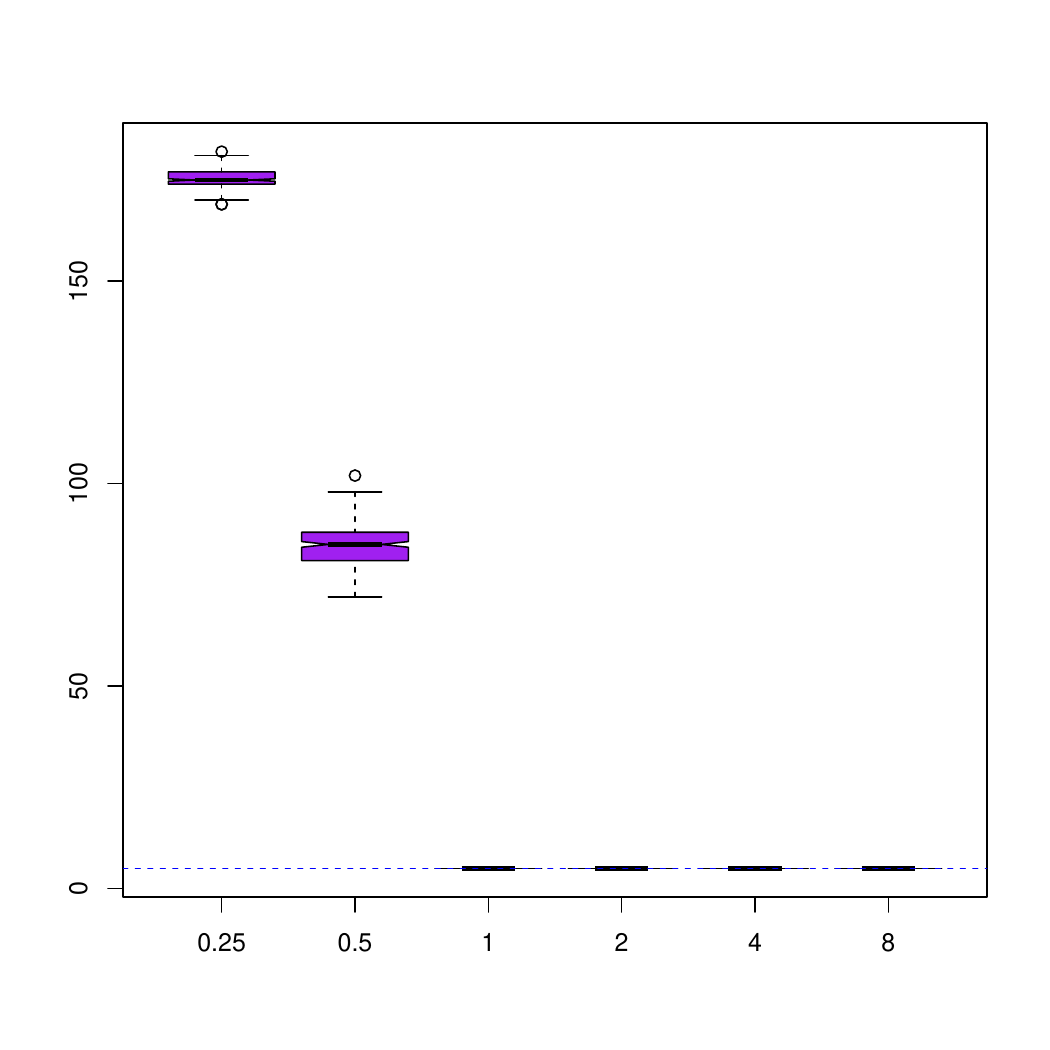}}
	\subfigure[Exp. decaying, $s = 10, q = 10$]{\includegraphics[width=0.33\textwidth]{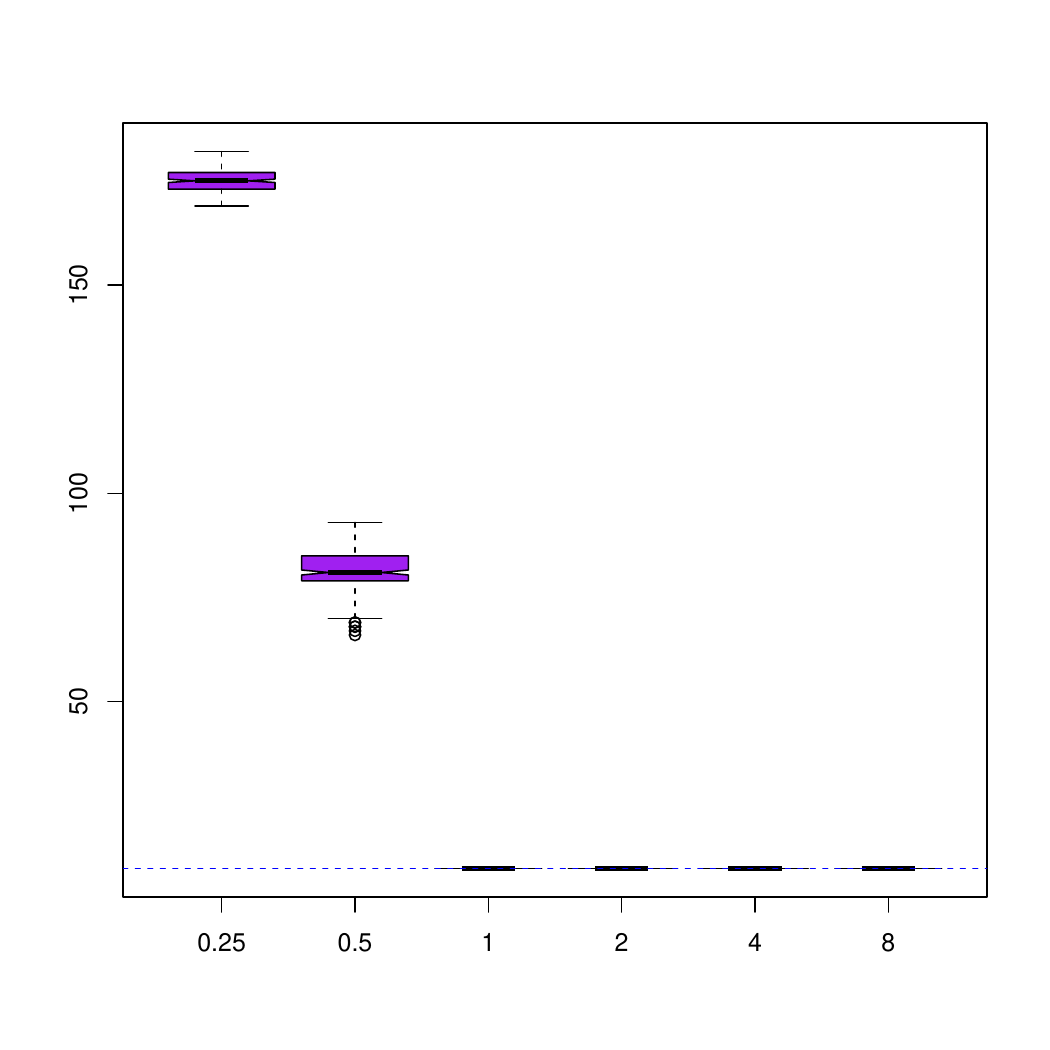}}
	\subfigure[Exp. decaying, $s = 10, q = 20$]{\includegraphics[width=0.33\textwidth]{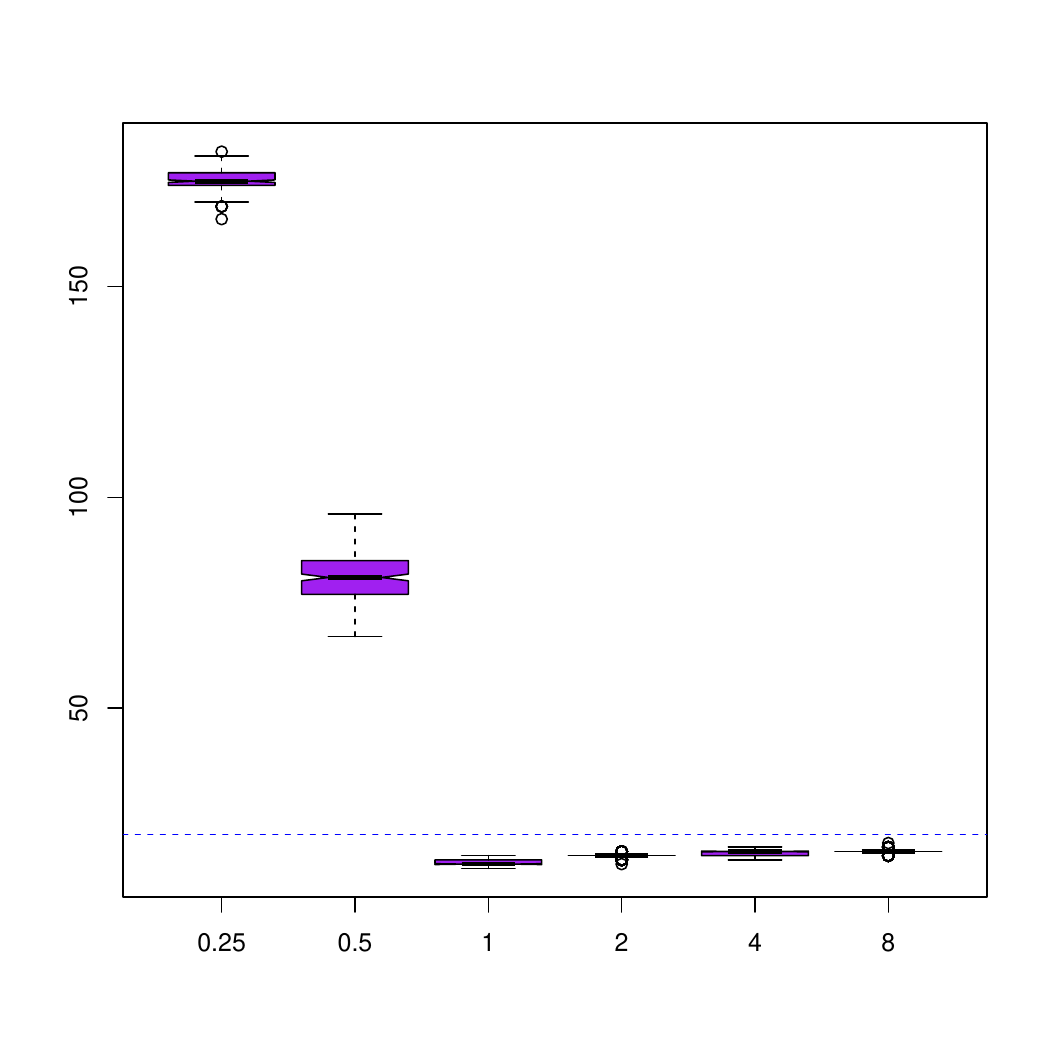}}
	\caption{Numbers of factors selected by the factor-based method in the exponential simulation examples for the different confounding strengths $\nu$ on the $x$-axis.}\label{fig:logiexpfactor1}
\end{figure}

\subsection{Inference} \label{sec:inf_exp2}
In Figure~\ref{fig:power1} we present results analogous to those of Section~\ref{sec:inf_exp} but for the case where $\delta^W$ has components generated as i.i.d.\ Rademacher variables. As can be seen, the conclusions are similar to those in Section~\ref{sec:inf_exp}.

\begin{figure}[t!]
    \centering
    \subfigure[$p = 200, s = 5$]{\includegraphics[width=0.45\textwidth]{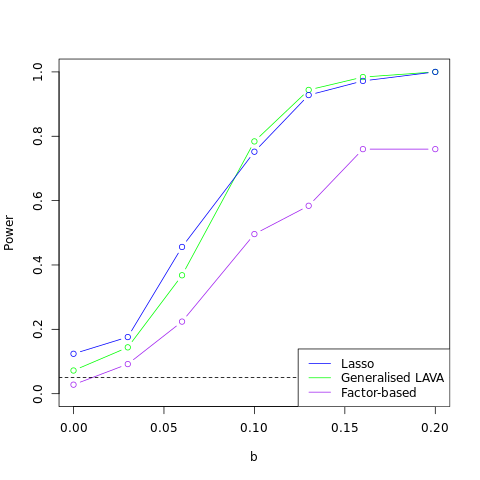}}
    \subfigure[$p = 200, s = 10$]{\includegraphics[width=0.45\textwidth]{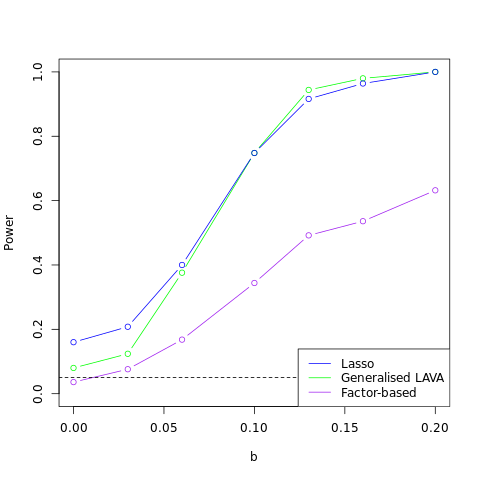}}
    \subfigure[$p = 400, s = 5$]{\includegraphics[width=0.45\textwidth]{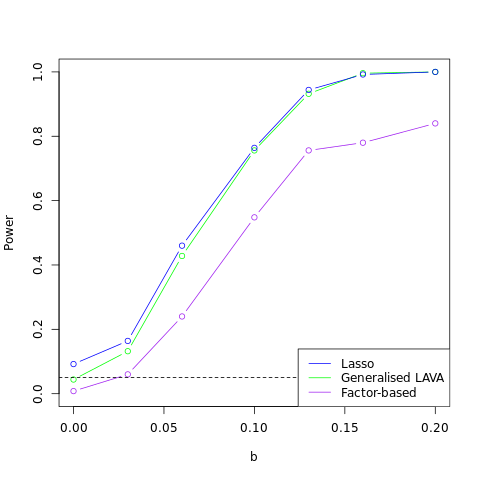}}
    \subfigure[$p = 400, s = 10$]{\includegraphics[width=0.45\textwidth]{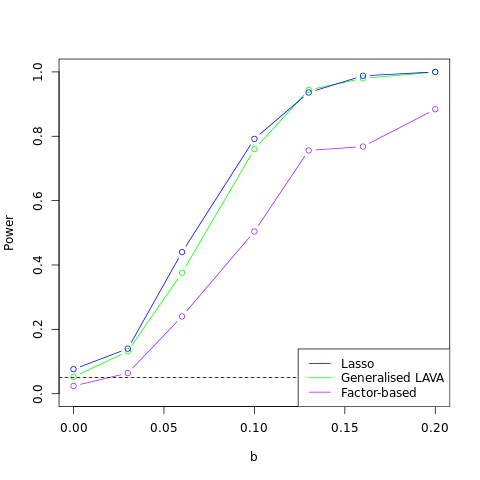}}
    \caption{Average power of the different methods for $b$ ranging from $0$ to $0.2$. The dashed line signifies $0.05$.}
    \label{fig:power1}
\end{figure}
\end{document}